\documentclass[11pt]{article}

\usepackage{amssymb}
\usepackage{latexsym}
\usepackage{amsmath}
\usepackage{amscd}
\usepackage{amsbsy}
\usepackage{graphicx}
\usepackage[square, sort, comma, numbers]{natbib}
\usepackage{dsfont}
\usepackage{relsize,exscale}

\newtheorem{theorem}{Theorem}

\newtheorem{lemma}{Lemma}

\newtheorem{proposition}{Proposition}

\newenvironment{proof}[1][Proof]{\textbf{#1.} }{\ \rule{0.5em}{0.5em}}

\newcommand{\Var}{\text{Var}}
\newcommand{\Cov}{\text{Cov}}
\newcommand{\logdet}{\text{logdet}}

\def\ds{\displaystyle}

\begin{document}

\title{Sharp minimax tests for large covariance matrices and adaptation
}

\author{Cristina Butucea$^{1,2}$, Rania Zgheib$^{1,2}$ \\
$^1$ Universit\'e Paris-Est Marne-la-Vall\'ee, \\
LAMA(UMR 8050), UPEM, UPEC, CNRS,
F-77454, Marne-la-Vall\'ee, France\\
$^2$ ENSAE-CREST-GENES \\
3, ave. P. Larousse
92245 MALAKOFF Cedex, FRANCE
}

\date{}
\maketitle

\begin{abstract}

We consider the detection problem of correlations in a $p$-dimensional Gaussian vector, when we observe $n$ independent, identically distributed random vectors, for $n$ and $p$ large.
We assume that the covariance matrix varies in some ellipsoid with parameter
$\alpha >1/2$ and total energy bounded by $L>0$.

We propose a test procedure based on a U-statistic of order 2 which is weighted in an optimal way. The weights are the solution of an optimization problem, they are constant on each diagonal and non-null only for the $T$ first diagonals, where $T=o(p)$. We show that this test statistic is asymptotically Gaussian distributed under the null hypothesis and also under the alternative hypothesis  for matrices close to the detection boundary. We prove upper bounds for the total error probability of our test procedure, for $\alpha>1/2$ and under the assumption $T=o(p)$ which implies that $n=o(p^{ 2 \alpha})$.
We illustrate via a numerical study the  behavior of our test procedure.

Moreover, we prove lower bounds for the maximal type II error and the total error probabilities. Thus we obtain the asymptotic and the sharp asymptotically  minimax separation rate $\widetilde{\varphi} = (C(\alpha, L) n^2 p )^{- \alpha/(4 \alpha + 1)}$, for  $\alpha>3/2$ and for $\alpha >1$ together with the additional assumption $p= o(n^{4 \alpha -1})$, respectively. We deduce rate asymptotic minimax results for testing the inverse of the covariance matrix.

We construct an adaptive test procedure with respect to the parameter $\alpha$ and show that it attains the rate $\widetilde{\psi}= ( n^2 p / \ln\ln(n \ds\sqrt{p}) )^{- \alpha/(4 \alpha + 1)}$.

\end{abstract}

\noindent {\bf Mathematics Subject Classifications 2000:} 62G10, 62H15, 62G20\\

\noindent {\bf Key Words:} adaptive test, covariance matrix, goodness-of-fit tests, high-dimensional data, minimax separation rate, sharp asymptotic rate, U-statistic
\\



\section{Introduction}\label{introduction}

A large variety of applied fields collect and need to recover information from high-dimensional data. Among these we can cite  communications and signal theory (functional magnetic resonance imaging, spectroscopic imaging), econometrics, climate studies, biology (gene expression micro-array) and finance (portfolio allocation).
Testing large covariance matrix is an important problem and has recently been approached via several techniques: corrected likelihood ratio test using the theory of large random matrices, methods based on the sample covariance matrix and so on.

Let $X_1,\dots,X_n$, be $n$ independent and identically distributed $p$-vectors following a  multivariate normal distribution $\mathcal{N}_p(0,\Sigma)$, where $\Sigma=[\sigma_{ij} ]_{ 1\leq i,j \leq p}$ is the normalized covariance matrix, with $\sigma_{ii} = 1$, for all $i=1$ to $p$.
Let us denote by $X_k = (X_{k,1},\dots, X_{k,p})^T$ for all $k=1,\dots,p$.
In this paper we also assume that the size $p$ of the vectors grows to infinity as well as the sample size $n$, $p \to \infty$ and $n \to \infty$.

We consider the following goodness-of-fit test, where we test the null hypothesis
\begin{equation}\label{h0}
  H_0 : \Sigma = I, ~~\text{ where } I \text{ is the } p \times p \text{ identity matrix }
\end{equation}
against the composite alternative hypothesis
\begin{equation*}
  H_1 : \Sigma \in \mathcal{F}(\alpha, L), ~~ \text {such that }
   \ds\frac{1}{2p}\| \Sigma -I \|_F^2  \geq \varphi^2.
   \end{equation*}
For any covariance matrix $\Sigma $, we recall that the Frobenius norm is computed as
$$\|\Sigma-I\|_F^2 = tr[(\Sigma - I)^2]= 2 \sum_{1 \leq i < j \leq p}\sigma_{ij}^2.$$
The class of matrices  $ \mathcal{F}(\alpha,L)$ is defined as follows, for $\alpha > 0$,
\begin{equation*}
     \mathcal{F}(\alpha,L)= \left\{\Sigma \geq 0~ ;  \ds\frac{1}{p}\ds\sum_{1 \leq i<j \leq p} \sigma_{ij}^2 |i-j|^{2 \alpha}
     \leq L \text{ for all $p$ and }   \sigma_{ii} = 1 ~\text{for all}~ i=1,\dots, p \right\}
\end{equation*}
In order to test $H_0 : \Sigma = \Sigma_0$, for some given non negative definite covariance matrix $\Sigma_0$, we suggest rescaling the data $Z_i= \Sigma_0^{- 1/2} X_i$ and then apply the same test procedure provided that
$ \Sigma_0^{- 1/2} \Sigma \Sigma_0^{- 1/2}$ belongs to $\mathcal{F}(\alpha, L)$.
Let us denote by
 \begin{equation}
  \label{classq}
     Q(\alpha, L, \varphi)= \left\{ \Sigma \in \mathcal{F}(\alpha,L) ~;~ \ds\frac{1}{p}\ds\sum_{1 \leq i<j \leq p}
     \sigma_{ij}^2 \geq   \varphi^2~ \right\} \, ,
\end{equation}
where $ \varphi = \varphi_{n,p}(\alpha, L)$ is related to $n$ and $p$, but also to $\alpha$ and $L$ assumed fixed.
The set of covariance matrices under the alternative hypothesis consists of matrices of size $p \times p$, whose elements
decrease polynomially when moving away from the diagonal. This assumption is natural for covariances matrices and has been considered for estimation problems, see e.g \citep{BickelLevina08}, \citep{CaiZhangZhou10}. Regularization techniques, originally used for nonparametric estimation of functions, were successfully employed to the estimation of large covariance matrices. Among these works, let us mention minimax and adaptive minimax results: via banding the covariance matrix \citep{BickelLevina08}, thresholding the entries of the empirical covariance matrix \citep{BickelLevina08B}, block-thresholding \cite{CaiYuan12}, tapering \citep{CaiZhangZhou10}, $\ell_1$-estimation \cite{CaiZhou12} and so on. Unlike the estimation of the covariance matrix, there are very few works for testing in a minimax setup in the existing literature.

Several types of test statistics have been proposed in the literature in order to test the null hypothesis \eqref{h0}.
The likelihood ratio (LR) statistic,  was first designed for fixed $p$ and $n \to + \infty$. To treat the high dimensional case when $n,p \to + \infty$,  \cite{bai2009} proposed a correction to the LR statistic and showed its convergence in law under the null hypothesis, as soon as $p/n \to c$, for some fixed $c\in (0,1)$. Indeed, this correction is based on the asymptotic behavior of the spectrum of the covariance matrix.
Another approach is  based on the largest magnitude of the off-diagonal entries of the sample
correlation matrix and  was introduced  by \cite{Jiang04}. Later, \citep{CaiJiang11} and  \cite{XiaoWu13} show an original limit behavior of Gumbel type for a self-normalized version of the maximum deviation of the sample covariance matrix.
We also note that a  non-asymptotic sphericity test for Gaussian vectors was studied by \cite{BerthetRigolet13}. The alternative is given by a model with rank-one and sparse additive perturbation in the variance.

Furthermore, an approach based on the quadratic form  $U_n=(1/p) tr(S_n -I)^2$, where $S_n= (1/n) \sum_{i=1}^n X_i X_i^\top$ is the sample covariance matrix, was  proposed  by \cite{Nagao1973}, to test \eqref{h0}. Later, \cite{LedoitWolf02} shows that the test of $H_0$ based  on $U_n$ is not consistent for large $p$. They introduce  a corrected version  of $U_n$ and study its asymptotic behavior when $n,p \to \infty$ and  $p/n \to c \in ( 0+ \infty)$.
In order to deal with non Gaussian random vectors, and without specifying any relation between $n$ and $p$, \cite{ChenZZ10} proposed a U-statistic of order 2, as a new correction of the previous quadratic form. They do moment assumptions in order to show the asymptotic behavior of their U-statistic, under the null and under a fixed   alternative hypothesis.  Motivated by their work,  \cite{CaiMa13} used the U-statistic given in \cite{ChenZZ10} to test \eqref{h0} from a sample of Gaussian vectors, and studied the testing problem from a minimax point of view. They consider the alternative hypothesis $H_1 : \Sigma \text{  such that } \| \Sigma -I \|_F \geq \varphi$ and they establish the minimax rates of order $\ds\sqrt{p/n}$ in this case. In our setup the restriction to the ellipsoid $\mathcal{F}(\alpha, L)$ leads to different rates for testing.

In this paper, we introduce a U-statistic, which is weighted in an optimal way for our problem. This can also be seen as a regularization technique for estimating a quadratic functional, as it is often the case in minimax nonparametric test theory (see \citep{IngsterSuslina03}). We use this test statistic to construct an asymptotically minimax test procedure. Let us stress the fact that we study the type II error probability uniformly over the set of all matrices $\Sigma$ under the alternative and that induces a separation rate saying how close $\Sigma$ can be to the identity matrix $I$ and still be distinguishable from $I$. We describe the sharp separation rates for fixed unknown $\alpha$ and give an adaptive procedure free of $\alpha$ that allows to test at the price of a logarithmic loss in the rate.

\bigskip

We describe here the rate asymptotics of the error probabilities from the minimax point of view.
We recall that a test procedure $\Delta$ is a measurable function with respect to
the observations, taking values in $[0,1]$.
Set $\eta (\Delta)= \mathbb{E}_I(\Delta)= \mathbb{P}_I( \Delta =1)$ its type I error probability,
 $\beta (\Delta, Q(\alpha, L, \varphi))   =  \sup\limits_{ \Sigma \in Q(\alpha, L, \varphi)  }
\mathbb{E}_{\Sigma}(1- \Delta) = \sup\limits_{ \Sigma \in Q(\alpha, L, \varphi)  }
\mathbb{P}_{\Sigma} ( \Delta  = 0 )$ its maximal type II error probability over the set
$Q(\alpha, L, \varphi)$, and by
\[
 \gamma ( \Delta , Q(\alpha, L, \varphi)) = \eta(\Delta) +  \beta (\Delta, Q(\alpha, L, \varphi))
\]
the total error probability of $\Delta$. Let us denote by $\gamma$ the minimax total error
probability over $ Q(\alpha, L, \varphi)$ which is defined by
\[
  \gamma = \gamma(\varphi) := \gamma(Q(\alpha, L, \varphi)) = \inf\limits_{ \Delta }
  \gamma ( \Delta , Q(\alpha, L, \varphi))
\]
where the infimum is taken over all test procedures.
We want to describe the separation rate $\widetilde{\varphi} = \widetilde{\varphi}(n,p)$
such that, on the one hand,
\[
\gamma \to 1 ~~\text{ if } \ds\frac{\varphi}{\widetilde{\varphi}} \to 0.
\]
In this case we say that we can not distinguish between the two hypotheses.
On the other hand, we exhibit an explicit test procedure $\Delta^*$ such that its total error probability tends to $0$
\[
\gamma(\Delta^*, Q(\alpha, L,\varphi)) \to 0 ~~ \text{ if }
\ds\frac{\varphi}{\widetilde{\varphi}} \to + \infty.
\]
We say that $\Delta^*$ is asymptotically minimax consistent test and $\widetilde{\varphi}$ is the asymptotically minimax separation rate.

\bigskip

In this paper, we find asymptotically minimax rates for testing over the class $\mathcal{F}(\alpha,L)$. The minimax consistent test procedure is based on a U-statistic of second order, weighted in an optimal way. In this, our procedure is very different from known corrected procedures based on quadratic forms of the sample covariance matrix, see e.g. \cite{LedoitWolf02}. This is the first time a weighted test-statistic is used for testing covariance matrices.

\bigskip

Moreover, our rates are sharp minimax. We show a Gaussian asymptotic behaviour of the test statistic in the neighbourhood of the separation rate. We get the following sharp asymptotic expression for the maximal type II probability error, under some assumptions relating $\varphi$, $n$ and $p$,
$$
\inf_{\Delta: \eta(\Delta) \leq w} \beta(\Delta, Q(\alpha, L, \varphi)) = \Phi(z_{1-w}-n\sqrt{p}b(\varphi)) + o(1),
$$
where $\Phi$ denotes the cumulative distribution function (cdf) of the standard Gaussian distribution and $z_{1-w}$ is the $1-w$ quantile of the standard Gaussian distribution for any $w \in (0,1)$. We deduce that the sharp minimax total error probability is of the type
$$
\gamma (\varphi) = 2 \Phi(-n \sqrt{p}\, b(\varphi) /2) + o(1),
$$
where $b^2(\varphi) = C(\alpha, L) \varphi ^{4+1/\alpha}$ as $\varphi \to 0$, $C(\alpha,L)$ is explicitly given.
It is usual to call the asymptotically sharp minimax rate 
$$
\widetilde{\varphi} = (C(\alpha, L) n^2 p )^{- \alpha/(4 \alpha + 1)},
$$
corresponding to $n^2p b^2(\widetilde \varphi) = 1$ and to the asymptotic testing constant $C(\alpha,L)$.

Analogous results were obtained by \cite{ButuceaZgheib15} in the particular case where the covariance matrix is Toeplitz, that is $\sigma_{i,j} = \sigma_{|i-j|}$ for all different $i$ and $j$ from 1 to $p$. We note a gain of a factor $p$ in the minimax rate. The asymptotically sharp minimax rate for Toeplitz covariance matrices is
$$
\widetilde{\varphi}_T = (C(\alpha, L) n^2 p^2 )^{- \alpha/(4 \alpha + 1)}.
$$
This additional factor $p$ can be heuristically explained by the number of parameters $p-1$ for a Toeplitz matrix, instead of $p(p-1)/2 $ for an arbitrary covariance matrix.
%
For $n=1$ the test problem for Toeplitz covariance matrices was solved in the sharp asymptotic framework, as $p \to \infty$, by \cite{Ermakov94}. Let us also recall that the adaptive rates (to $\alpha$) for minimax testing are obtained for the spectral density problem by \cite{GolubevNussbaumZhou10} by a non constructive method using the asymptotic equivalence with a Gaussian white noise model. We also give an adaptive procedure for testing without prior knowledge on $\alpha$, for $\alpha$ belonging to a closed subset of $(1, + \infty)$.

\bigskip

 Important generalizations of this problem include testing in a minimax setup of composite null hypotheses like sphericity, $H_0: \Sigma = v^2 \cdot I$, for unknown $v^2$ in some compact set separated from 0, or bandedness, $ H_0: \Sigma = \Sigma_0 $ such that $[\Sigma_0]_{ij} = 0$ for all $i \neq j$ with $|i-j| > K$.
Our proofs rely on the Gaussian distribution of Gaussian vectors. Generalizations to non Gaussian distributions with finite moments of some order can be proposed under additional assumptions on the behaviour of higher order moments, like e.g. \cite{ChenZZ10}.

\bigskip

Section~\ref{sec:test} introduces the test statistic and studies its asymptotic properties. Next we give upper bounds for the maximal type II error probability and for the total error probability and refine these results to sharp asymptotics under the condition that $n = o(1) p^{2\alpha}$.
In Section~\ref{simu} we implement our test procedure and estimate its power.
In Section~\ref{sec:lowerbounds} we prove sharp asymptotic optimality  and deduce the optimality of the minimax separation rates for all $\alpha >1$ and as soon as $p=o(n^{4 \alpha -1})$.
In Section~\ref{sec:inverse} we present the rate minimax ressults for testing the inverse of the covariance matrix.
In Section \ref{adaptivity} we define an adaptive test procedure and show that the price of adaptation is a loss of $(\ln\ln(n \sqrt{p}) )^{\alpha/(4 \alpha +1)}$ in the  separation rate.

Proofs are given in Section~\ref{sec:proofs} and in the Appendix.

\section{Test procedure and sharp asymptotics} \label{sec:test}

In the minimax theory of tests developed since \citep{IngsterI93} it is well understood that optimal test statistics are estimators (suitably normalized and tuned) of the functional which defines the separation of an element in the alternative from the element of the null hypothesis. In our case this is the Frobenius norm $\|\Sigma - I\|_F^2 = tr[(\Sigma - I)^2]$.

Weighting the elements of the sample covariance matrix appeared first as hard thresholding in minimax estimation of large covariance matrices. Let us mention \cite{BickelLevina08} for banding i.e. truncation of the matrix to its $k$ first diagonals (closest to the main diagonal), \cite{BickelLevina08B} for hard thresholding, then \cite{CaiZhangZhou10} where tapering was studied. It is a natural idea when coming from minimax nonparametric estimation.

However, that was never used for tests concerning large covariance matrices. In this section, we introduce a weighted U-statistic of order 2 for testing large covariance matrices, study its asymptotic properties and give asymptotic upper bounds for the minimax rates of testing.

\bigskip

From now on asymptotics and  symbols  $o$, $O$, $\sim$ and $\asymp$ are considered $n$ and $p$ tend to infinity.
Recall that, given sequences of real numbers $u$ and real positive numbers $v$, we say that they are
asymptotically equivalent,  $u \sim v$, if $\lim u/v= 1$. Moreover, we say that the sequences are asymptotically of the same order, $u \asymp v$, if there exist two constants $0<c \leq C < \infty$ such that $c \leq \lim\inf u/v$ and $\limsup u/v \leq C$.

\subsection{Test statistic and its asymptotic behaviour}
\label{testprocedurenontpz}
Our test statistic is a weighted U-statistic of order 2. It can be also seen as a weighted functional of the sample covariance matrix.
The weights $w^*_{ij}$ are constant on each diagonal (they depend on $i$ and $j$ only through $i-j$), non-zero only for $|i-j|\leq T$ for some large integer $T$ and decreasing polynomially for elements further from the main diagonal (as $|i-j|$ is increasing).  More precisely,
we consider the following test statistic:
\begin{equation}
  \label{est}
      \widehat{\mathcal{D}}_n = \frac{1}{n(n-1)p} \ds\sum_{1 \leq k \ne l \leq n} \ds\sum_{1 \leq i<j \leq p} w^*_{ij} X_{k,i}X_{k,j}X_{l,i}X_{l,j}
\end{equation}
where
\begin{equation}
\label{weights}
\begin{array}{lcl}
w_{ij}^* &= &\ds\frac{\lambda}{2b(\varphi)} \Big( 1 - \Big(\ds\frac{|i-j|}{T} \Big)^{2 \alpha} \Big)_+ , \quad  T = \lfloor C_T( \alpha, L) \cdot \varphi^{- \frac{1}{\alpha}} \rfloor \\ \\
 \lambda &=&  C_{\lambda}( \alpha, L ) \cdot \varphi^{\frac{2\alpha + 1}{ \alpha}} , \quad
b(\varphi)  = C^{1/2}(\alpha, L) \cdot \varphi^{2+ \frac{1}{2 \alpha}}
\end{array}
\end{equation}
with
\begin{equation}
\label{constants}
\begin{array}{lcl}
C_T(\alpha, L) &=& \left( \ds(4\alpha +1)L \right)^{\frac{1}{2\alpha}}, \quad
 C_{\lambda}(\alpha,L) \, = \,\ds\frac{2\alpha+1}{2 \alpha } \left( \ds(4\alpha +1)L \right)^{-\frac{1}{2\alpha}},
\\ \\
 C(\alpha,L) &=& \ds\frac{ 2\alpha +1 }{ (4\alpha+1)^{1+ {1}/(2\alpha)}}
\ds L ^{-\frac{1}{2\alpha}}.
\end{array}
\end{equation}

\noindent
The weights $\{ w_{ij}^* \}_{i,j}$ and the parameters $T, \lambda , b^2(\varphi)$ are obtained by solving  the following optimization problem  :
\begin{equation}
\label{optprob}
\ds\frac{1}{p}\sum_{1 \leq i < j \leq p} w^*_{ij} \sigma_{ij}^{*2}=
  \sup\limits_{ \left\{\substack{ (w_{ij})_{ij} ~~:~~ w_{ij} \geq 0 ;\\ \\
  \frac{1}{p}\sum_{1 \leq i < j \leq p} w_{ij}^2 =\frac{1}{2}} \right\} }
  \inf\limits_{ \left\{ \substack{ \Sigma ~:~ \Sigma=(\sigma_{ij})_{i,j} ;\\ \\
  \Sigma \in Q(\alpha,L,\varphi)} \right\} }
  \ds\frac{1}{p}\sum_{1 \leq i < j \leq p} w_{ij} \sigma_{ij}^2
\end{equation}
Indeed our test statistic $ \widehat{\mathcal{D}}_n$ will concentrate asymptotically around  the value   $\mathbb{E}_\Sigma(\widehat{\mathcal{D}}_n) = (1/p) \sum_{1 \leq i <j \leq p} w_{ij} \sigma_{ij}^2$ which is $0$ for $\Sigma= I$. The minimax paradigm considers the worst parameter $\Sigma^*$ in the class $Q(\alpha, L, \varphi)$, that will give the smallest value $\mathbb{E}_{\Sigma^*}( \widehat{\mathcal{D}}_n(w_{ij})) $ and then finds the parameters $\{ w^*_{ij} \}_{i<j}$  of the test statistic to provide the largest value $\mathbb{E}_{\Sigma^*} (\widehat{\mathcal{D}}_n(w_{ij}^*))$. Such procedure performs uniformly well over all parameters $\Sigma \in Q(\alpha, L, \varphi)$. That explains why we solve the optimization problem \eqref{optprob}.

\noindent
Note that the weights in (\ref{weights}) have further properties:
\begin{eqnarray*}
 && w^*_{ij} \geq 0~, \quad  \ds\frac{1}{p}  \ds\sum_{1 \leq i < j \leq p}  w_{ij}^{*2} =\frac{1}{2}~, ~~
 \sup\limits_{i,j} w_{ij}^*  \asymp \ds\frac{1}{\ds\sqrt{T}}, \mbox{ as } \varphi \to 0 \mbox{ and } p\, \varphi^{1 /\alpha}\to \infty.
\end{eqnarray*}

\vspace{0.5cm}

The following Proposition gives the moments of  $\widehat{\mathcal{D}}_n $ under the null and their bounds under the alternative hypothesis, respectively, as well as the asymptotic normality under the null hypothesis.
\begin{proposition}
\label{prop:espvar}
The test statistic $\widehat{\mathcal{D}}_n $ defined by (\ref{est}) with parameters given by (\ref{weights}) and (\ref{constants}) has the following moments, under the null hypothesis:
\begin{eqnarray*}
  \label{var0}
   \mathbb{E}_{I}(\widehat{\mathcal{D}}_n) = 0,& \quad &  \Var_{I}(\widehat{\mathcal{D}}_n)=\displaystyle\frac{2}{n(n-1)p^2} \ds\sum_{ 1 \leq i< j \leq p}w_{ij}^{*2} = \frac{1}{n(n-1)p}\\
\end{eqnarray*}
Also we have that
\[
  n\sqrt{p} \,\, \widehat{\mathcal{D}}_n  \overset{d}{\to} \mathcal{N}(0,1).
\]
Moreover, under the alternative, if we assume that $\varphi \to 0,\, p \, \varphi^{1/\alpha} \to \infty$  and $\alpha >1/2$, we have, for all $\Sigma$ in  $Q(\alpha, L, \varphi)$:
\begin{equation*}
  \label{EH1}
       \mathbb{E}_{\Sigma}(\widehat{\mathcal{D}}_n) = \ds\frac{1}{p} \sum_{1 \leq i < j \leq p} w_{ij}^* \sigma_{ij}^2 \geq b(\varphi) \quad \text{ and }  \quad
\Var_{\Sigma} (\widehat{\mathcal{D}}_n) = \ds\frac{T_1}{n(n-1)p^2}+ \frac{T_2}{np^2},
\end{equation*}
where
\begin{eqnarray}
\label{T_1}
T_1 & \leq & p \cdot ( 1 + o(1)) +    p  \cdot \mathbb{E}_{\Sigma}(\widehat{\mathcal{D}}_n)  \cdot O(T\ds\sqrt{T}) , \label{T1}\\
\label{T_2}
T_2 & \leq &  p \cdot  \mathbb{E}_{\Sigma}(\widehat{\mathcal{D}}_n) \cdot O(\sqrt{T})  + p^{3/2} \cdot  \Big( \mathbb{E}^{3/2}_{\Sigma}(\widehat{\mathcal{D}}_n) \cdot  O( T^{3/4}) + \mathbb{E}_{\Sigma}(\widehat{\mathcal{D}}_n) \cdot o(1) \Big) .\label{T2}
\end{eqnarray}
\end{proposition}
Note that, under the alternative, we have the additional assumption that $p \varphi^{1/\alpha} \asymp p/T  \to + \infty$, when $p$ grows to infinity. This is natural in order to a have a meaningful weighted statistic. 

Let us look closer at the optimization problem (\ref{optprob}): for given $\varphi >0$, $b(\varphi)$ is the least value that $\mathbb{E}_\Sigma(\widehat{\mathcal{D}}_n)$ can take over $\Sigma$ in the alternative set of hypotheses.

Under the alternative, we shall establish the asymptotic normality under additional conditions that the underlying covariance matrix is  close to the border of $\mathcal{F}(\alpha,L)$.  This will be sufficient to give upper bounds of the total error probability of Gaussian type in next Section.
\begin{proposition}\label{AN}
The test statistic $\widehat{\mathcal{D}}_n $ defined by (\ref{est}) with parameters given by (\ref{weights}) and (\ref{constants}), such that $\varphi \to 0$, $p\varphi^{1/\alpha} \to \infty$ and under the aditionnal assumption that $n^2pb^2(\varphi) \asymp 1$, is asymptotically normal:
$$
n\sqrt{p} (\widehat{\mathcal{D}}_n-\mathbb{E}_{\Sigma} (\widehat{\mathcal{D}}_n))  \overset{d}{\to} \mathcal{N}(0,1),
$$
for any $\Sigma$ in $Q(\alpha, L, \varphi)$ such that $\mathbb{E}_{\Sigma}(\widehat{\mathcal{D}}_n) = O(b(\varphi))$.
\end{proposition}

\subsection{Upper bounds for the error probabilities}

 In order to distinguish between the two hypothesis $H_0$ and $H_1$ defined previously,
we propose the following test procedure
\begin{equation}
\label{test*}
  \Delta^* = \Delta^*(t) = \mathds{1} ( \widehat{\mathcal{D}}_n > t) ,  \quad t>0
 \end{equation}
where $\widehat{\mathcal{D}}_n$ is the estimator defined in (\ref{est}).

The following theorem proves that the previously defined test procedure is minimax consistent if $t$ is conveniently chosen.
\begin{theorem}
 \label{theo:bornesupMinimax}
The test procedure $\Delta^*$ defined in (\ref{test*}) with $t>0$ has the following properties :

Type I error probability : if $n\ds\sqrt{p} \cdot t \to + \infty $ then $  \eta(\Delta^*) \to 0$.

Type II error probability : if  $\alpha >1/2$ and if
\begin{equation*}
\label{conditionbornesupmM}
\varphi \to 0 , p \varphi^{1/\alpha} \to \infty \mbox{ and }
 \, 
 n^2 p b^2(\varphi) \to + \infty
\end{equation*}
then, uniformly over  $t$ such that $t \leq c  \cdot C^{1/2}( \alpha ,L) \cdot \varphi^{2+ \frac{1 }{2\alpha}}$ , for some constant $c $ in $(0, 1)$, we have
$$
\beta( \Delta^*(t), Q( \alpha, L, \varphi)) \to 0.
$$
If $t$ verifies all previous assumptions, then $\Delta^*(t)$ is asymptotically minimax consistent:
$$
\gamma( \Delta^*(t), Q( \alpha, L, \varphi)) \to 0.
$$
\end{theorem}

In the next Theorem we give  sharp  upper bounds of error probabilities of Gaussian type. The proof of this result explains the choice of the weights as solution of the optimization problem \eqref{optprob}. Moreover, we will see that the Gaussian behavior is obtained near the separation rates.

Recall that $\Phi$ is the cumulative distribution function (cdf) of standard Gaussian random variable and, for any $w \in (0,1)$, $z_{1-w}$ is defined by $\Phi({z_{1-w}}) = 1-w$.
\begin{theorem}
 \label{theo:bornesup}
The test procedure $\Delta^*$ defined in (\ref{test*}) with $t>0$ has the following properties :

Type I error probability : we have  $  \eta(\Delta^*(t)) = 1 - \Phi(n\sqrt{p} \cdot t) + o(1)$.

Type II error probability : if  $\alpha >1/2$ and  if
\begin{equation}
\label{conditionbornesup}
\varphi \to 0, \, p \varphi^{1/\alpha} \to \infty \mbox{ and }
n^2 p \, b^2(\varphi) \asymp 1 \, ,
\end{equation}
then, uniformly over $t$, 
we have
$$
\beta( \Delta^*(t), Q( \alpha, L, \varphi)) \leq \Phi(n\sqrt{p} \cdot ( t - b(\varphi)))+o(1).
$$
\end{theorem}

 In particular, for $t=t^w$ such that $n\ds\sqrt{p} \cdot t^w = z_{1-w} $ we have  $ \eta(\Delta^*(t^w)) \leq w+o(1)$
and
$$
\beta( \Delta^*(t^w), Q( \alpha, L, \varphi)) \leq \Phi(z_{1-w } - n\sqrt{p} \cdot b(\varphi))+o(1).
$$
Another important consequence of the previous theorem, is that the test procedure $\Delta^*(t^*)$, with $t^* = b(\varphi)/2$ has total error probability
$$
\gamma ( \Delta^*(t^*) , Q(\alpha,L, \varphi))  \leq 2 \, \Phi\left(- n \sqrt{p} \,\frac{b(\varphi)}2 \right)+o(1).
$$

\subsection{Simulation study}
\label{simu}

We include two examples, to illustrate the numerical behavior of our test procedure. First, we test the null hypothesis $ \Sigma =I $ against the alternative hypothesis defined by the symmetric positive  matrices  $
\Sigma(M)= \Big( \mathds{1}_{ \{ i=j \}} +  \mathds{1}_{ \{ i \ne j \}} \cdot (|i-j|^{- \frac 32} \cdot |i+j|^{ \frac 1{100}} )/ M \Big)_{1 \leq i,j \leq p}$ .  We implement the test statistic $\widehat{\mathcal{D}}_n$ defined in \eqref{est} and \eqref{weights} for $\alpha =L =1$, and $\varphi = \psi(M)=(1/p)\Big(  \sum_{i<j} |i-j|^{- 3} \cdot |i+j|^{ \frac 2{100}} \Big)^{1/2}/ M $. We choose  the threshold $t$ of the test empirically, under the null hypothesis $\Sigma =Id$, from 1000 repetead samples of size $n$, such that the type I error probability is fixed at 0.05. We use $t$ to estimate the type II error probability, also from 1000 repetitions
and then plot the power as function of $\psi(M)$.
\begin{center}
 \begin{figure}[hb!]
\includegraphics[width= 10cm, height=7cm]{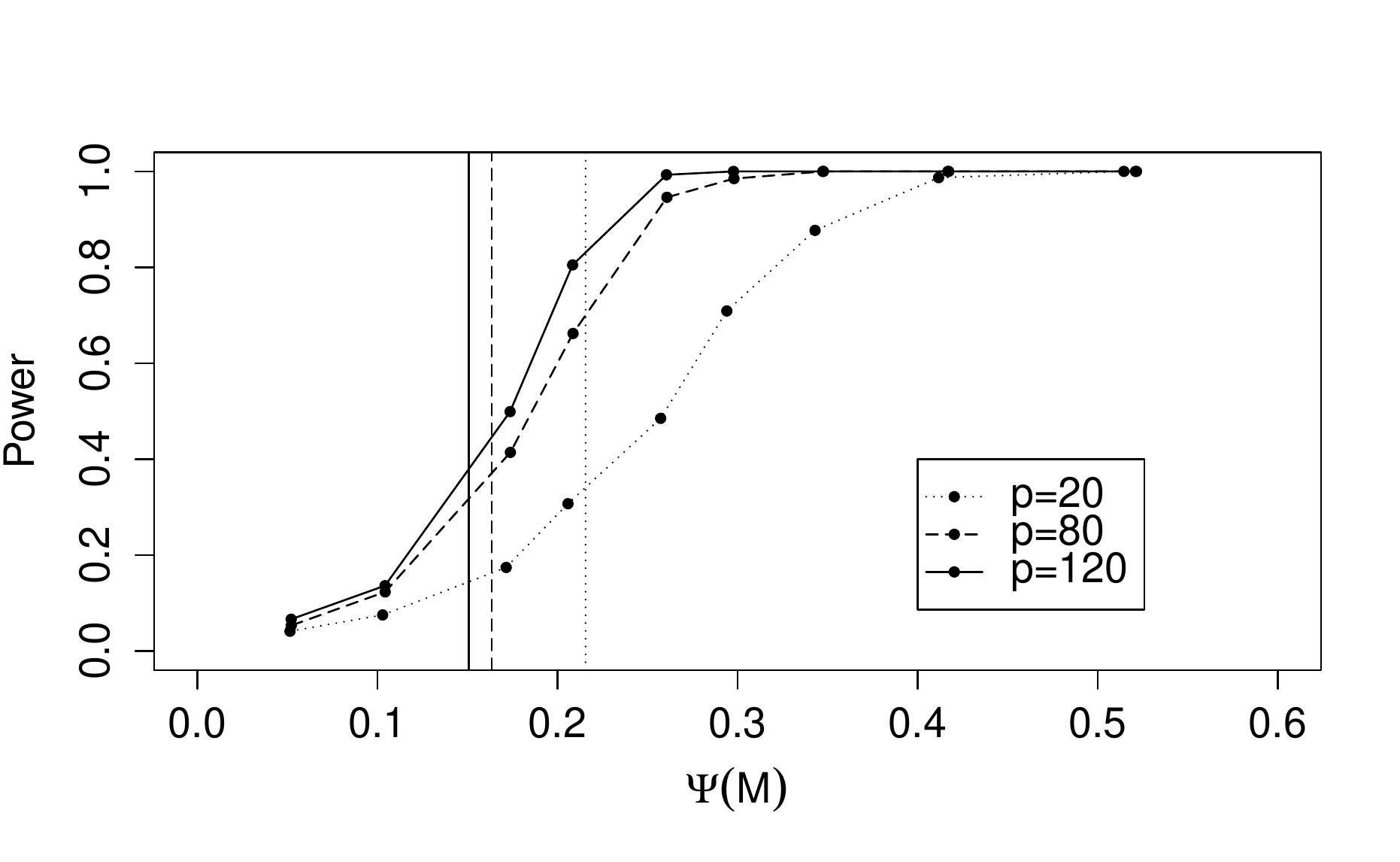}
\caption{Power curves of the $\Delta$-test as function of $\psi(M)$ for $n=20 $ and $p \in \{20, 80, 120\}$}
\label{plusieursp}
\end{figure}
\end{center}
Figure \ref{plusieursp} shows that the power is an increasing function of $\psi(M)$. Also, we can see that for a fixed value of $\psi(M)$ the power increases with $p$. Indeed, our procedure benefits from large values of $p$, which is not a nuisance parameter here.

Second we consider the tridiagonal matrices $\Sigma(\rho)= (\mathds{1}_{ \{ i=j \}} + \rho \mathds{1}_{ \{ |i-j|=1 \} })_{1 \leq i,j \leq p}$, for $\rho \in (0, 0.35]$ under the  alternative hypothesis. We compare our test procedure to the one given in  \citep{CaiMa13}, which is based on a $U$-statistic of order 2, we denote it CM-test. Moreover, the matrices $\Sigma(\rho)$ are Toeplitz, we also compare our the procedure to the one proposed in \citep{ButuceaZgheib15} for Toeplitz covariance matrices, that we denote by BZ-test. The thresholds are evaluated empirically for each procedure at type I error probability smaller than 0.05. We finally plot the powers curves of the three test procedures.
\begin{center}
 \begin{figure}[hpt!]
\includegraphics[width= 7cm, height=6cm]{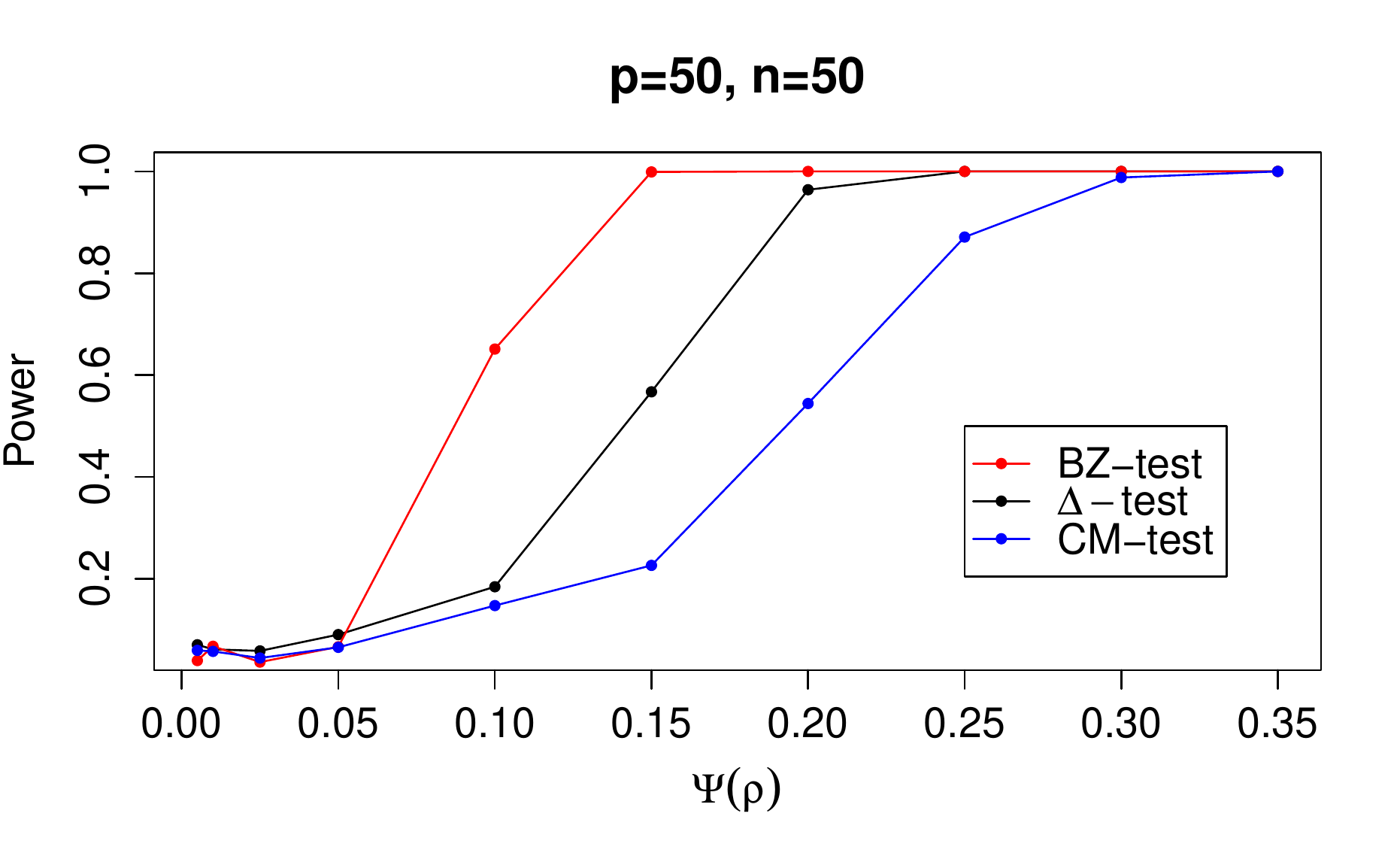}
\includegraphics[width= 7cm, height=6cm]{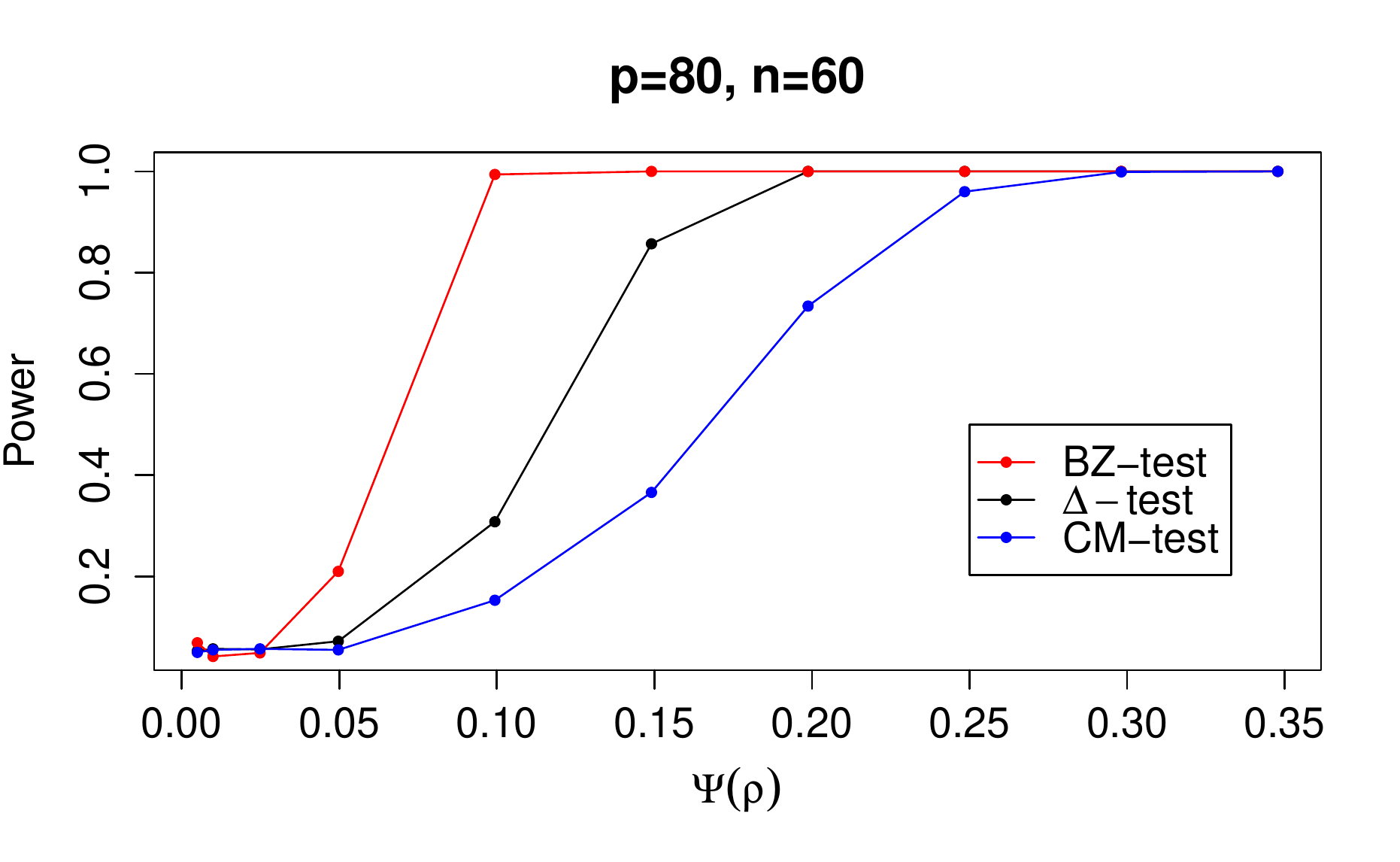}
\caption{Power curves of the BZ-test, $\Delta$-test and CM-test as function of $\psi(\rho)$}
\label{comparaison}
\end{figure}
\end{center}
Figure \ref{comparaison} shows that, when the alternative hypothesis consists of Toeplitz matrices the BZ-test has the better performance. However if we miss  the information that the matrix is Toeplitz, we see that  the $\Delta$-test  is not bad and its power dominates the power of the CM-test.

\section{Asymptotic optimality}\label{sec:lowerbounds}
In this section, we first state the lower bound for testing, which, in addition to the test procedure exhibited in the previous section, shows that the asymptotically minimax separation rate is
\begin{equation} \label{phitilde}
\widetilde\varphi =\left(n\ds\sqrt{p}\,  C^{1/2}(\alpha,L) \right)^{-\frac{2\alpha }{4\alpha+1}},
\end{equation}
where the constant $C(\alpha,L)$ is given by (\ref{constants}).  
\begin{theorem}
\label{theo:borneinfO}
Assume that, either $\alpha > 3/2$, or $\alpha >5/8$ and $np \varphi^{6 - \frac 2\alpha} \to 0 $. If
\begin{equation*}
\label{conditionborneinfO}
\varphi \to 0, \, p \varphi^{1/\alpha} \to \infty, \, \mbox{ and }
n^2 p \, b^2(\varphi) \to 0 ,
\end{equation*}
then
\[\gamma =\inf\limits_{ \Delta }  \gamma ( \Delta , Q(\alpha,L, \varphi))
\to 1,
\]
where the infimum is taken over all test statistics $\Delta$.
\end{theorem}
Together with Theorem~\ref{theo:bornesupMinimax}, the proof that $\widetilde{\varphi}$ is asymptotically minimax, under our assumptions, is complete. Note that the condition $ np \varphi^{6 - \frac 2\alpha} \to 0  $ is verified when $\alpha >3/2$  for all $n$ and $p \to + \infty$ giving a general result in this case. When $5/8<\alpha< 3/2$, the same condition holds for $p =o(n^{\frac{8 \alpha -5}{3 -2\alpha}})$. This result is proven by showing that the $\chi^2$ distance between the null hypothesis and an averaged likelihood under the alternative (that we explicitly construct) tends to 0.

Moreover, we give a sharp lower bound for the type II error probability which is of Gaussian type.
\begin{theorem}
\label{theo:borneinf}
Assume that  $\alpha >1$ and if
 \begin{equation}
\label{conditionborneinf}
\varphi \to 0, \, p \varphi^{1/\alpha} \to \infty, \, \sqrt{p} \varphi^{2 -\frac 1{2 \alpha}} \to 0 \mbox{ and }
n^2 p \, b^2(\varphi) \asymp 1 ,
\end{equation}
$$
\inf\limits_{ \Delta: \eta(\Delta) \leq w }  \beta ( \Delta , Q(\alpha,L, \varphi)) \geq \Phi(z_{1-w} - n\sqrt{p} b(\varphi) ) +o(1),
$$
where the infimum is taken over all test statistics $\Delta$ with type I error probability less than or equal to $w$.
Moreover,
\[\gamma =\inf\limits_{ \Delta }  \gamma ( \Delta , Q(\alpha,L, \varphi)) \geq 2 \Phi(-n \sqrt{p} \,\frac{b(\varphi)}2) + o(1).
\]
\end{theorem}

Theorems \ref{theo:bornesup} and \ref{theo:borneinf} imply that for $\alpha >1$, the sharp separation rate for minimax testing is
$ \widetilde{\varphi} $, under the additionnal assumptions (\ref{conditionbornesup}) and (\ref{conditionborneinf}). Note that a sufficient condition is that the separation rate verifies these assumptions, in particular  $p \widetilde \varphi^{1/\alpha} \to \infty$ holds if $n = o(1) p^{2\alpha}$, and  $\ds\sqrt{p} \widetilde\varphi^{2 -\frac 1{2 \alpha}} \to 0$ holds if $p = o(n^{4 \alpha - 1}) $.

Note that, there is a more general test procedure independent of $\varphi$, for which it is possible to derive the upper bounds as in Theorems \ref{theo:bornesupMinimax} and \ref{theo:bornesup}. It suffices to use the test statistic $ \widehat{\mathcal{D}}_n $ with the weights $w_{ij}^*$ replaced by the weights $w_{ij}^*( \widetilde\varphi) $ defined as  in \eqref{weights} and \eqref{constants} for $\varphi$ replaced by $\widetilde{\varphi}$. For more details see section 4.2 in \citep{ButuceaGayraud2013}.

The proof of the lower bounds is given in  Section~\ref{sec:proofs}.
We construct a family of $n$ large centered Gaussian vectors with covariance matrices based on $\{\sigma^*_{ij}\}_{1 \leq i,j \leq p}$ given by the optimization problem \eqref{optprob} and a prior measure $P_\pi$ on these covariance matrices.
The logarithm of the likelihood ratio  associated to an arbitrary $\Sigma$ with respect to $I$ under the null hypothesis is known to drift away to infinity (see \cite{bai2009}, who corrected this ratio to get a proper limit). However, we show that the logarithm of the Bayesian likelihood ratio with our prior measure $P_\pi$ verifies
\begin{equation*}
\log \frac{f_\pi}{f_I}(X_1,...,X_n) = u_n Z_n -\frac{u_n^2}2 + \xi, \quad \text{ in $P_I$ probability}
\end{equation*}
where $u_n = n \sqrt{p} b(\varphi)$,  $Z_n$ is asymptotically distributed as a standard Gaussian distribution and $\xi$ is a random variable which converges to zero under $P_I$ probability.

\section{Testing the inverse of the covariance matrix}\label{sec:inverse}

Let us consider the same model, but the following test problem
$$
H_0: \Sigma^{-1} = I
$$
against the alternative
$$
H_1: \Sigma \in \mathcal{G}(\alpha, L, \lambda) \mbox{ such that } \frac 1{2p} \|\Sigma ^{-1}- I\|_F^2 \geq \psi^2,
$$
where $\mathcal{G}(\alpha, L, \lambda)$ is the class of covariance matrices $\Sigma$ in $\mathcal{F}(\alpha, L)$ with the additional constraint that the eigenvalues $\lambda_i(\Sigma)$ are bounded from below by some $\lambda \in (0,1)$ for all $i$ from 1 to $p$ and all $\Sigma$ in the set.

We prove here that previous results apply to this setup and we get the same rates, but not the sharp asymptotics.
Note that, the additional hypothesis is mild enough so that it does not change the rates for testing. Indeed, we see this case as a well-posed inverse problem. The cases of ill-posed inverse problem where the smallest eigenvalue can be allowed to tend to 0 will most certainly imply a loss in the rate and is beyond the scope of this paper.

\begin{theorem} Suppose $\alpha >3/2$, $L>0$ and $\lambda \in (0,1)$. If $n$ and $p$ tend to infinity, such that $ n = o(1) p^{2\alpha } $, then
$\widetilde \varphi$ defined in (\ref{phitilde}) is the asymptotically minimax rate for the previous test.
\end{theorem}

\begin{proof} Note that $\Sigma^{-1} = I$ if and only if $\Sigma = I$. Moreover, if $\Sigma$ belongs to $\mathcal{G}(\alpha, L, \lambda)$ such that $\frac 1{2p} \|\Sigma ^{-1}- I\|_F^2 \geq \psi^2$, then $\Sigma$ obviously belongs to $\mathcal{F}(\alpha, L)$ and is such that
$$
\frac 1{2p} \|\Sigma - I\|_F^2 \geq \frac{\lambda^2}{2 p} \|\Sigma ^{-1}- I\|_F^2 \geq \lambda^2\psi^2.
$$
Thus we can proceed with our former test procedure, with $\varphi$ replaced by $\lambda \psi$ and we obtain the upper bounds in the definition of the separation rates.

The lower bounds in the previous Section will also remain valid. Indeed, this proof is based on the construction of a subfamily $\{\Sigma^*_U: u\in \mathcal{U}\}$ on the set of alternatives. We have proven in Proposition~\ref{cor1}, that
$$
\min_i \lambda_i(\Sigma^*_U) \geq 1 - O(\varphi^{1-1/(2\alpha)}),
$$
and we have $\alpha >1$ and $\varphi = \lambda \psi \to 0$ as $\psi \to 0$ and therefore, $1-O(\varphi^{1-1/(2\alpha)}) \geq \lambda$ for $\psi >0$ small enough. Thus, this family belongs to the set of alternatives we consider here, as well. Moreover, Proposition~\ref{cor1} proves also that
$$
\|\Sigma^*_U\|_2:=\max_i \lambda_i(\Sigma^*_U) \leq 1 + O(\varphi^{1-1/(2\alpha)}) \leq \lambda_{max},
$$
for some fixed $\lambda_{max}$ free of $\alpha$ and $L$. Thus,
$$
\frac{1}{2 p} \|(\Sigma^*_U) ^{-1}- I\|_F^2  \geq  \frac{1}{2 p \cdot \|\Sigma^*_U\|_2^2 } \cdot  \|\Sigma^*_U - I\|_F^2\geq \frac 1{\lambda_{max}^2} \frac{1}{2 p  } \cdot  \|\Sigma^*_U - I\|_F^2.
$$
Thus we proceed the same way with $\varphi$ replaced by $\lambda_{max} \psi$.
\end{proof}


\section{Adaptive testing procedure}
\label{adaptivity}

We want to built a test procedure of $H_0$  in \eqref{h0} which is free of the parameter $\alpha$ belonging to some closed interval $\mathcal{A} = [\underline{\alpha}, \bar{\alpha}] \subset (1, + \infty)$. The radius $L$ plays a minor role in the procedure and we suppose that it is known (w.l.o.g we assume that $L=1$). Such a procedure is called adaptive and it solves the test problem $H_0$ in \eqref{h0} against a much larger set of alternative hypotheses:
\begin{equation}
\label{alteradap}
H_1: \Sigma \in \underset{\alpha \in \mathcal{A}}{\cup} Q(\alpha, L,\mathcal{C} \psi_\alpha) \, ,
\end{equation}
where $\mathcal{C}$ is a large enough positive constant and
\begin{equation}
\label{rateadap}
\psi_\alpha=  \Big(  \ds\frac{ \rho_{n,p}}{ n \sqrt{p}} \Big)^{ \frac{2 \alpha}{4 \alpha +1}} \,  ,  \quad  \rho_{n,p}= \ds\sqrt{\ln\ln(n \ds\sqrt{p})} \, ,
\end{equation}
depend on $n$ and $p$, but also on $\alpha$.
In order to construct the adaptive test procedure, we define a finite regular grid over the set $\mathcal{A}= [ \underline{\alpha}, \bar{\alpha}]$ :
$$
\mathcal{A}_N = \{ \alpha_r = \underline{\alpha} +  \ds\frac{\bar{\alpha} - \underline{\alpha}}{N} \cdot r \, ; r = 1, \dots , N \}, \text{ where } N = \lceil \ln(n \sqrt{p}) \rceil .
$$
To each $r \in\{1, \cdots, N \}$, we associate the weights :
\[
w_{ij,r}^*= \ds\frac{\lambda_r}{b_r} \Big( 1 - \Big(\ds\frac{|i-j|}{T_r} \Big)^{2 \alpha_r} \Big)_+ ,
\]
where the parameters $\lambda_r, b_r$ and $T_r$ are given in \eqref{weights} and  \eqref{constants} with $\alpha$ replaced by $\alpha_r$ and $\varphi$ by $\psi_\alpha$.
Define the adaptive test procedure, for some constant $ \mathcal{C}^* >0$ large enough
\begin{equation}
\label{adptivetest}
\Delta_{ad}^* = \max\limits_{r=1, \dots, N} \mathds{1}( \widehat{\mathcal{D}}_{n,r} > \mathcal{C}^* t_r ) , \quad \text{ where } t_r = C_{\lambda_r} \cdot \rho_{n,p}/(n \sqrt{p}),
\end{equation}
and  where $\widehat{\mathcal{D}}_{n,r}$ is the test statistic in \eqref{est} with weights $\{ w_{ij, r} \}_{i<j}$. Note that the test $\Delta^*_{ad}$ rejects the null hypothesis as soon as there exists at least on $r \in \{1, \dots, N\}$ for which $ \widehat{\mathcal{D}}_{n,r} > \mathcal{C}^* t_r$.
\begin{theorem}
\label{theo:adaptivity}
Assume that
\[
p \cdot  \Big( \ds\frac{\rho_{n,p}}{n \ds\sqrt{p}} \Big)^{ \frac 2{4 \underline{\alpha}+1} } \to + \infty  \quad \text{and} \quad \ds\frac{\ln p}n \to 0
\]
The test statistic  defined in \eqref{adptivetest}  with $\mathcal{C}^*$ large enough
verifies :
\[
\gamma ( \Delta_{ad}^* , \underset{\alpha \in \mathcal{A} }{\cup} Q(\alpha,L, \mathcal{C}\psi_\alpha) ) \to 0,
\]
 for all  $ \mathcal{C} >  \Big(\mathcal{C}^* + \ds\frac{1}{C(\underline{\alpha}, \bar{\alpha})} \Big)$,
 where  $\psi_\alpha$ is given in \eqref{rateadap} and $
C( \underline{\alpha}, \bar{\alpha} ) = \exp(-8(\bar \alpha - \underline{\alpha})/(4 \underline{\alpha}+1))$.
 \end{theorem}

The proof that the  adaptive procedure we propose  attains the above rate is given in Section \ref{sec:proofs}. By analogy to nonparametric testing of functions, we expect the loss $\rho_{n,p}$ to be optimal uniformly over the class in the alternative hypothesis \eqref{alteradap} .

\section{Proofs}\label{sec:proofs}

\label{Sectionproofs}

\begin{proof}[Proof of Theorems \ref{theo:bornesupMinimax} and \ref{theo:bornesup}]
The proof is based on the Proposition \ref{prop:espvar} and the asymptotic normality of the weighted test statistic $n \sqrt{p} \widehat{\mathcal{D}}_n $ in Proposition~\ref{AN}.
We get for the type I error probability of $\Delta^*$
\[
  \eta(\Delta) = \mathbb{P} ( \widehat{\mathcal{D}}_n > t) =1- \Phi(n\ds\sqrt{p} \cdot t) + o(1).
\]

\noindent
For the type II error probability of $\Delta^*$,
uniformly in $\Sigma $ over $ Q(\alpha, L, \varphi)$, we have 
\begin{eqnarray*}
   \mathbb{P}_{\Sigma} (\widehat{\mathcal{D}}_n \leq t )
    &\leq& \mathbb{P}_{\Sigma} (|\widehat{\mathcal{D}}_n - \mathbb{E}_{\Sigma} (\widehat{\mathcal{D}}_n)| \geq \mathbb{E}_{\Sigma} (\widehat{\mathcal{D}}_n) - t)
   \leq \ds\frac{\Var_{\Sigma}(\widehat{\mathcal{D}}_n)}{(\mathbb{E}_{\Sigma} (\widehat{\mathcal{D}}_n) - t)^2} ,
   \end{eqnarray*}
   for $t \leq c \cdot b(\varphi) $ and $0<c<1$. It implies that
   $n \sqrt{p} \cdot t \leq c n\sqrt{p} b(\varphi)$.
Therefore, we distinguish the cases where $n^2 p b^2(\varphi)$ tends to infinity or is bounded.

   We use the fact that, under the alternative, $\mathbb{E}_{\Sigma} (\widehat{\mathcal{D}}_n) \geq b(\varphi)$. We bound from below as follows:
   $$
   \mathbb{E}_{\Sigma} (\widehat{\mathcal{D}}_n) - t \geq (1-c) \mathbb{E}_{\Sigma} (\widehat{\mathcal{D}}_n) .
   $$
Then, it gives  \begin{eqnarray*}
   \mathbb{P}_{\Sigma} (\widehat{\mathcal{D}}_n \leq t )
   &\leq& \ds\frac{T_1}{n(n-1) p^2(1-c)^2\mathbb{E}^2_\Sigma(\widehat{\mathcal{D}}_n)} + \ds\frac{T_2}{np^2(1-c)^2 \mathbb{E}^2_\Sigma(\widehat{\mathcal{D}}_n)} =: S_1 + S_2.
  \end{eqnarray*}
Let us bound from above $S_1$ using (\ref{T_1}):
\[
\begin{array}{lcl}
 S_1 \leq \ds\frac{1+ o(1)}{n(n-1)p(1-c)^2b^2(\varphi)}    + \ds\frac{O(T^{3/2})}{n(n-1)p\, b(\varphi)} .
 \end{array}
\]
We have $T^{3/2} b(\varphi) \asymp \varphi^{2 - \frac{1}{\alpha}} = o(1) , \text{ for all } \alpha > 1/2,$
which proves that :
\[
 S_1 \leq \ds\frac{1+ o(1)}{n(n-1)p(1-c)^2b^2(\varphi)}
\]
which tends to $0$ provided that  $n^2p b^{2}(\varphi) \to + \infty $.
We will see using (\ref{T_2})  that the  term $S_2$ tends to 0 as well:
\[
\begin{array}{lcl}
 S_2 &\leq &   \ds\frac{O(\ds\sqrt{T})}{np \, b(\varphi)} + \frac{O(T^{3/4} b^{{1}/{2}}(\varphi))}{n \ds\sqrt{p}\, b(\varphi)} + \ds\frac{o(1)}{n \ds\sqrt{p}\, b(\varphi)}  \\[0.3 cm]
  &=& o(1) \text{ for all } \alpha > 1/2, \text{ as soon as } n \ds\sqrt{p} b(\varphi) \to + \infty.
  \end{array}
\]

Now, if $\varphi$ is  close to the separation rate: $n^2pb^2(\varphi)  \asymp 1$, we see that whenever $\mathbb{E}_\Sigma(\widehat{\mathcal{D}}_n)/b(\varphi) $ tends to infinity, the bound is trivial ($S_1+S_2 \to 0$).

The nontrivial bound is obtained when $\Sigma$ under the alternative is close to the optimal matrix $\Sigma^* =( \sigma^*_{ij})_{1 \leq i, j \leq p}$, in the sense that  $\mathbb{E}_\Sigma(\widehat{\mathcal{D}}_n) = O(b(\varphi))$ together with the fact that $\varphi$ is close to the separation rate: $n^2pb^2(\varphi) \asymp 1$.
We apply Proposition~\ref{AN} to get the asymptotic normality
$$
n\sqrt{p} (\widehat{\mathcal{D}}_n-\mathbb{E}_{\Sigma} (\widehat{\mathcal{D}}_n)) \to \mathcal{N}(0,1).
$$
Thus,
\begin{eqnarray*}
\sup_{\Sigma \in Q(\alpha, L, \varphi))}\mathbb{P}_{\Sigma} (\widehat{\mathcal{D}}_n \leq t )
& \leq & \sup_{\Sigma \in Q(\alpha, L, \varphi))}\Phi(n\sqrt{p}\cdot (t - \mathbb{E}_{\Sigma} (\widehat{\mathcal{D}}_n  ))) +o(1)\\
& \leq & \Phi (n\sqrt{p} \cdot (t - \inf_{\Sigma \in Q(\alpha, L, \varphi))} \mathbb{E}_{\Sigma} (\widehat{\mathcal{D}}_n  ))) +o(1).
\end{eqnarray*}
At this point, choosing optimal weights translates into
\begin{eqnarray*}
\inf_{w_{ij}>0: \sum_{i\ne j}w^2_{ij}=1/2} && \sup_{\Sigma \in Q(\alpha, L, \varphi))}\mathbb{P}_{\Sigma} (\widehat{\mathcal{D}}_n \leq t )\\
&\leq &
\Phi (n\sqrt{p} \cdot (t - \sup_{w_{ij} > 0: \sum_{i\ne j}w^2_{ij}=1/2}\inf_{\Sigma \in Q(\alpha, L, \varphi))} \mathbb{E}_{\Sigma} (\widehat{\mathcal{D}}_n  ))) +o(1)\\
& \leq & \Phi(n\sqrt{p} \cdot (t -b(\varphi)))+o(1),
\end{eqnarray*}
after solving the optimization problem in the Appendix, which ends the proof of the Theorem.
\hfill \end{proof}


\begin{proof}[Proof of Theorems \ref{theo:borneinfO} and \ref{theo:borneinf}]
The first step of the proof is to reduce the set of parameters to a convenient parametric family. Let $\Sigma^*=[\sigma_{ij}^*]_{1 \leq i, j \leq p}$ be the matrix which has 1 on the diagonal and off-diagonal entries $\sigma_{ij}^*$ where
\begin{equation}
\label{sigma*}
 \sigma_{ij}^*= \sqrt{\lambda}\left( 1- (\ds\frac{|i-j|}{T})^{2\alpha}\right)_+^{\frac{1}{2}}
 \text{ ~~~ for }  i \neq j,
 \end{equation}
with $\lambda$  and $T$ are given by (\ref{weights}) and (\ref{constants}).

Let us define
$Q^*$  a subset of $Q(\alpha,L,\varphi)$ as follows
\[
Q^* = \{ \Sigma^*_U : [\Sigma^*_U]_{ij} = I(i=j)+u_{ij}\sigma_{ij}^*I(i \neq j) \text{ for all } 1 \leq i, j \leq p~,  ~U = [u_{ij}]_{1 \leq i ,\, j\leq p } \in \mathcal{U \}} ,
\]
where
\begin{equation*}
\mathcal{U} = \{U= [u_{ij}]_{1 \leq i ,\, j\leq p } :  u_{ii} =0 , \forall i \mbox{ and } \, u_{ij} = u_{j\, i}= \pm 1 \cdot I(|i-j|\leq T), \mbox{ for } i\neq j \}.
\end{equation*}
The cardinality of $\mathcal{U}$ is $p(T-1)/2$.
\begin{proposition}
\label{cor1}
 For $\alpha > 1/2$, the symmetric matrix $\Sigma^*_U= [u_{ij} \sigma^*_{ij} ]_{1 \leq i,j \leq p}$, with $\sigma_{ii}^* =1 $, for all $i$ from 1 to $p$, and $\sigma_{ij}^* $  defined in (\ref{sigma*}) is non-negative definite, for $\varphi>0$ small enough, and for all $U \in \mathcal{U}$.

Moreover,  denote by $\lambda_{1,U}, ..., \lambda_{p,U}$ the eigenvalues of $\Sigma^*_U$, then
$|\lambda_{i,U}-1| \leq O(1) \varphi^{1-1/(2\alpha)}$, for all $i$ from 1 to $p$.
\end{proposition}
We deduce that
\begin{equation}\label{sigma}
\|\Sigma^*_U \| \leq 1+ O(\varphi^{1-\frac 1{2\alpha}}) \mbox{ and }
\|\Sigma^*_U -I \| \leq O(\varphi^{1-\frac 1{2\alpha}}).
\end{equation}
Indeed,
$
\|\Sigma^*_U \| = \max_{i =1,...,p} \lambda_{i,U} \leq 1 +O(\varphi^{1+\frac 1{2 \alpha}})
$
and $\Sigma^*_U - I$ has eigenvalues $\lambda_{i,U}-1$.

Proposition~\ref{cor1} shows that for all $ \Sigma^*_U \in  Q^*$,  $\Sigma^*_U $ is non-negative definite, for $\varphi >0$ small enough.

Assume that $X_1, \dots,X_n \sim N(0, I)$ under the null hypothesis and denote by $P_I$ the likelihood of these random variables. We assume that $X_1, \dots,X_n \sim N(0, \Sigma^*_U)$, under the alternative, and we denote $P_U$ the associated likelihood. In addition let
\[
 P_{\pi} = \frac{1}{2^{p(T-1)/2}} \ds\sum_{ U \in \mathcal{U}} P_U
\]
 be the average likelihood  over $Q^*$.

The problem can be reduced to the test $H_0: X_1,...,X_n \sim P_I$ against the averaged distribution $H_1: X_1,...,X_n \sim P_\pi$, in the sense that
\begin{equation*}
\inf\limits_{ \Delta: \eta(\Delta) \leq w }  \beta ( \Delta(t) , Q(\alpha,L, \varphi)) \geq \inf\limits_{ \Delta: \eta(\Delta) \leq w }  \beta ( \Delta(t) , P_\pi ) + o(1)
\end{equation*}
and that
\begin{equation*}
\inf\limits_{ \Delta }  \gamma ( \Delta , Q(\alpha,L, \varphi)) \geq
\inf\limits_{ \Delta }  \gamma ( \Delta , P_\pi)+ o(1).
\end{equation*}
It is, therefore, sufficient to show that, when $u_n  \asymp 1$,
\begin{equation}\label{beta}
\inf\limits_{ \Delta: \eta(\Delta) \leq w }  \beta ( \Delta(t) , P_\pi ) \geq \Phi(n\sqrt{p} \cdot (t-b(\varphi))) +o(1)
\end{equation}
and that
\begin{equation}\label{gamma}
\inf\limits_{ \Delta }  \gamma ( \Delta , P_\pi) \geq 2 \Phi(-n \sqrt{p} \,\frac{b(\varphi)}2) + o(1).
\end{equation}
While, for $u_n = o(1) $, we need to show that
\begin{eqnarray}\label{gamma0}
\inf\limits_{ \Delta }  \gamma ( \Delta , P_\pi) \to 1.
\end{eqnarray}
In order to obtain (\ref{beta}) and (\ref{gamma}), we apply results in Section 4.3.1 of \cite{IngsterSuslina03} giving the sufficient condition that, in $P_I$ probability:
\begin{equation}\label{lan}
L_{n,p} := \log \frac{f_\pi}{f_I}(X_1,...,X_n) = u_n Z_n -\frac{u_n^2}2 + \xi,
\end{equation}
where $u_n = n \sqrt{p} b(\varphi) 
 \asymp 1$, $b(\varphi) = C^{\frac 12}(\alpha, L) \cdot \varphi^{2 + \frac 1{2 \alpha}}$,  $Z_n$ is asymptotically distributed as a standard Gaussian distribution and $\xi$ is a random variable which converges to zero under $P_I$ probability.
Moreover, to show \eqref{gamma0}, it suffices to show that 
\begin{equation}
\label{integrale}
\mathbb{E}_I \left( \ds\frac{dP_{\pi}}{dP_I} \right)^2 \leq 1 +o(1),
\end{equation}
since 
\[
\gamma \geq 1 - \frac 12 \| P_I -P_\pi \|_1  \text{ and } \| P_I -P_\pi \|_1^2 \leq \mathbb{E}_I \left( \ds\frac{dP_{\pi}}{dP_I} \right)^2 -1   .
\]
We first begin by  showing \eqref{integrale}, in order to finish the proof of Theorem \ref{theo:borneinfO}. 
Let,
\begin{eqnarray}
 H_{n,p} &:=& \mathbb{E}_{I}   \left( \ds\frac{dP_{\pi}}{dP_{I}} (X_1, \cdots , X_n ) \right)^2 \nonumber \\
 & =&  \mathbb{E}_{I}\mathbb{E}_{U,V} \left( \ds\frac{\exp \Big(- \frac{1}{2} \sum_{k=1}^n  X^\top_k   ((\Sigma_U)^{-1}+  (\Sigma_V)^{-1} - 2 I) X_k \Big) }{(2\pi)^{ \frac{np}2} \, \det^{ \frac n2}(\Sigma_U \Sigma_V)} \right). \label{chi2}
\end{eqnarray}
We have
\begin{eqnarray*}
H_{n,p}&=& \mathbb{E}_{U,V} \left( \ds\frac{ \det^{ - \frac n2} \Big( (\Sigma_U)^{-1}+  (\Sigma_V)^{-1} - I \Big)}{  \det^{ \frac n2}(\Sigma_U \Sigma_V)} \right) =  \mathbb{E}_{U,V} \Big( \text{det}^{ - \frac n2} \Big( \Sigma_U +  \Sigma_V - \Sigma_U   \Sigma_V  \Big) \Big)\nonumber .
 \end{eqnarray*}
We define  $\Delta_{U} = \Sigma_{U} - I$ and  note that $\Sigma_U + \Sigma_V - \Sigma_U \Sigma_V= I- \Delta_U \Delta_V$. As the matrix $\Delta_U \Delta_V$ is not necessarily symmetric, we write
\[
(I- \Delta_U \Delta_V) (I- \Delta_U \Delta_V)^\top= I-M
\]  
where $M=M_{U,V} := \Delta_{U} \Delta_{V} + \Delta_{V}\Delta_{U} - \Delta_{U} \Delta_{V}^2\Delta_{U} $ is  symmetric. Moreover, we prove that for all $U $ and $V \in \mathcal{U}$ the eigenvalues of $M$ are in $(-1, 1)$ for all $\alpha >1/2$ and $\varphi$ small enough. Indeed, by Gershgorin's theorem, for each eigenvalue $\lambda_M$ of $M$ 
there exists at least one  $i \in \{1, \ldots, p \}$ such that
\[
| \lambda_M - M_{ii} | \leq \sum_{j ; j \ne i } |M_{ij}|.
\]
We can show that $\sum_{j ; j \ne i } |M_{ij}| = O(\varphi^{2 - \frac{1}{\alpha}}) $ and $|M_{ii}| \leq O(\varphi^2 ) + O(\varphi^{4 - \frac 1\alpha})$.
Thus, 
\begin{eqnarray*}
H_{n,p} = \mathbb{E}_{U,V}( \text{det}^{- \frac n4}(I-M)) =  \mathbb{E}_{U,V} \exp \Big( - \ds\frac n4 \logdet(I-M) \Big)
\end{eqnarray*}
The Taylor expansion for the logdet of a symmetric matrix writes
\[
- \ds\frac 14 \logdet (I-M)= \frac 14 tr(M) + \frac 18 tr(M^2) + O(tr(M^3)).
\]
In more  details,
\begin{eqnarray*}
\ds\frac{1}{4} tr(M) &=& \frac 12 tr (\Delta_U \Delta_V) - \ds\frac{1}{4} tr(\Delta^{2}_U \Delta^{2}_V)\\
  \frac{1}{8}tr(M^2) &=& \ds\frac{1}{4} tr(\Delta_U \Delta_V)^2 + \ds\frac{1}{4} tr(\Delta^{2}_U \Delta^{2}_V) + \frac{1}{8}tr(\Delta_U \Delta^2_V \Delta_U)^2 \\ [0.2 cm]
&&- \ds\frac{1}{4}tr(\Delta_U \Delta^{2}_V \Delta^{2}_U \Delta_V ) - \ds\frac{1}{4} tr(\Delta_V \Delta^{2}_U \Delta^{2}_V \Delta_U )
\end{eqnarray*}
Recall that $\forall A, B \in \mathbb{R}^{p \times p}$  we have $\|AB\|_F \leq \|A\|_2 \|B\|_F$.  For all $U, V \in \mathcal{U}$, we use the last inequality and the Cauchy-Schwarz inequality to get
\[
 tr( \Delta_U\Delta^2_V \Delta^2_U\Delta_V) \leq \|\Delta_U \Delta_V\|_F \, \| \Delta_V \Delta^2_U\Delta_V \|_F \leq \| \Delta_V \|_2 \, \| \Delta_U \|_F  \, \| \Delta_V \Delta^2_U \|_2  \, \| \Delta_V \|_F \leq p \cdot \varphi^{6 - \frac{2}{\alpha}} ,
\]
\[
tr(\Delta_U \Delta^2_V \Delta_U)^2  = \| \Delta_U \Delta^2_V \Delta_U \|^2_F \leq \|  \Delta_U \Delta^2_V \|^2_2 \,\, \| \Delta_U  \|^2_F \, 
\leq  p \cdot \varphi^{8 - \frac{3}{\alpha}}. 
\]
Finally, using similar arguments we can show that
\[
tr(M^3_{U,V})= O( p \varphi^{6 -\frac 2{\alpha}}).
\]
Thus, 
\[
- \frac 14 \logdet ( I-M) =  \frac{1}{2} tr (\Delta_U \Delta_V) +  \ds\frac{1}{4}tr (\Delta_U \Delta_V)^2 +   O(p  \varphi^{6 - \frac{2}{\alpha}} ).
\]
Now we develop the terms on the right hand side of the previous equation. We obtain  
\begin{eqnarray*}
tr (\Delta_U \Delta_V) &=&  \sum_{ \substack{ 1 \leq i ,j \leq p \\ 1 < |i-j| <T} }  u_{ij} v_{ij} \cdot \sigma_{ij}^{*2} \nonumber = 2 \sum_{ \substack{ 1 \leq i <j \leq p \\ 1 < |i-j| <T} }  u_{ij} v_{ij} \cdot \sigma_{ij}^{*2} \nonumber
\end{eqnarray*}
and
\begin{eqnarray*}
tr  (\Delta_U \Delta_V)^2  &=&   \sum_{ \substack{ 1 \leq i, j, h, l \leq  p \\ 1 < |i-h| , |h-j|, |i-l| , |l-j| < T}}  u_{ih}v_{hj}  u_{jl} v_{li} \cdot \sigma^*_{ih} \sigma^*_{hj} \sigma^*_{jl}  \sigma^*_{li} \nonumber \\
&=&  \sum_{ \substack{ 1 \leq i,  l \leq  p \\ 1 <  |i-l| < T}}   \sigma_{ij}^{*4}  
+ 2 \sum_{ \substack{ 1 \leq i , j, l \leq  p \\ i < j \\ 1 < |i-l| , |l-j| < T}} u_{il} u_{jl} v_{lj} v_{li} \cdot \sigma_{il}^{*2} \sigma_{lj}^{*2}  \nonumber \\
& + &  4 \sum_{ \substack{ 1 \leq i, j, h, l \leq  p \\ i <j  , l < h \\ 1 < |i-h| , |h-j|, |i-l| , |l-j| < T}}  u_{ih} u_{jl} v_{hj} v_{li} \cdot \sigma^*_{ih} \sigma^*_{jl} \sigma^*_{hj} \sigma^*_{li}. 
\end{eqnarray*}
Now, we can write \eqref{chi2} as follows:
\begin{eqnarray*}
H_{n,p} &=& \mathbb{E}_{U,V}  \exp \Big( - \ds\frac{n \logdet(I- \Delta_U \Delta_V)}{2} \Big)   \nonumber \\
&=&   \mathbb{E}_{U,V}  \exp \Big(  n  \sum_{ \substack{ 1 \leq i < j \leq p \\ 1 < |i-j| <T} }  u_{ij} v_{ij} \cdot \sigma_{ij}^{*2} + \frac n2 \sum_{ \substack{ 1 \leq i , j, l \leq  p \\ i < j \\ 1 < |i-l| , |l-j| < T}} u_{il} u_{jl} v_{lj} v_{li} \cdot \sigma_{il}^{*2}   \sigma_{lj}^{*2}  \nonumber \\
&+&  n \sum_{ \substack{ 1 \leq i, j, h, l \leq  p \\ i < j , l < h \\ 1 < |i-h| , |h-j|, |i-l| , |l-j| < T}} \hspace{- 0.7cm} u_{ih} u_{jl} v_{hj} v_{li} \cdot \sigma^*_{ih} \sigma^*_{jl} \sigma^*_{hj} \sigma^*_{li}  \Big)  
+ \ds\frac n4 \sum_{ \substack{ 1 \leq i,  l \leq  p \\ 1 <  |i-l| < T}}   \sigma_{ij}^{*4}   + O(n p  \varphi^{6 - \frac{2}{\alpha}})  .
\end{eqnarray*}
We explicit the expected value  with respect to the i.i.d Rademacher random variables $\{ u_{ij} v_{ij}\}_{i<j}$, $\{  u_{il} u_{jl} v_{lj} v_{li}\}_{i<j, l \not \in \{i,j\}}$ and $\{u_{ih} u_{jl} v_{hj} v_{li}\}_{i<j,l<h}$ pairwise distinct and independent: 
\begin{eqnarray}
H_{n,p} &=& \ds\prod_{ \substack{ i < j \\ 1 < |i-j| <T }} \cosh (  n \sigma_{ij}^{*2} )
 \ds\prod_{ \substack{ 1 \leq i , j, l \leq  p \\ i < j \\ 1 < |i-l| , |l-j| < T}}   \cosh ( \frac n2  \sigma_{il}^{*2}   \sigma_{lj}^{*2}  )  \nonumber \\
 && \hspace{-1 cm} \prod_{ \substack{ 1 \leq i, j, h, l \leq  p \\ i < j , l < h \\ 1 < |i-h| , |h-j|, |i-l| , |l-j| < T} } \cosh \Big( n \sigma^{*}_{ih} \sigma^{*}_{jl} \sigma^{*}_{hj} \sigma^{*}_{li} \Big) \,\, \exp \Big(  \ds\frac n2  \sum_{ \substack{ 1 \leq i,  l \leq  p \\ 1 <  |i-l| < T}}  \sigma_{ij}^{*4} + O(n p  \varphi^{6 - \frac{2}{\alpha}})  \Big) . \nonumber 
 \end{eqnarray}
We use the inequality $\cosh(x) \leq \exp(x^2/2)$ and get
\begin{eqnarray*}
H_{n,p} &\leq &  \exp \Big\{ \ds\frac{n^2}2 \Big( \ds\sum_{\substack{ i < j \\ 1 < |i-j| <T }} \sigma_{ij}^{*4} + \frac 14 \sum_{ \substack{ 1 \leq i , j, l \leq  p \\ i < j \\ 1 < |i-l| , |l-j| < T}} \sigma_{il}^{*4} \sigma_{lj}^{*4}  +  \sum_{ \substack{ 1 \leq i, j, h, l \leq  p \\ i < j , l < h \\ 1 < |i-h| , |h-j|, |i-l| , |l-j| < T} } \sigma_{ih}^{*2} \sigma_{jl}^{*2} \sigma_{hj}^{*^2} \sigma_{li}^{*2} \Big) \Big\} \nonumber \\
 && \cdot  \exp \Big(  \frac n2  \sum_{ \substack{ 1 \leq i,  l \leq  p \\ 1 <  |i-l| < T}}  \sigma_{ij}^{*4} + O(n p  \varphi^{6 - \frac{2}{\alpha}})  \Big). \nonumber  \\
\end{eqnarray*}
Or, $\ds\frac{n^2}2  \ds\sum_{\substack{ i < j \\ 1 < |i-j| <T }} \sigma_{ij}^{*4} = n^2p b^2(\varphi) $  
and since $\varphi \to 0$ we have that
\begin{eqnarray*}
&& \frac{ n^2}8 \sum_{ \substack{ 1 \leq i , j, l \leq  p \\ i \ne j \\ 1 < |i-l| , |l-j| < T}} \sigma_{il}^{*4} \sigma_{lj}^{*4} = \frac{n^2}8 \ds\sum_{\substack{ i \ne l \\ 1 < |i-l| <T }} \sigma_{il}^{*4} \sum_{ \substack{ j \\ 1 < |l-j| < T}}  \sigma_{lj}^{*4} =  n^2p b^2(\varphi) \cdot O(\lambda^2 T) =n^2p b^2(\varphi) \cdot o(1)
\end{eqnarray*}
and 
\begin{eqnarray*}
\frac{n^2}2  \sum_{ \substack{ 1 \leq i, j, h, l \leq  p \\ i \ne j , l \ne h \\ 1 < |i-h| , |h-j|, |i-l| , |l-j| < T} } \sigma^{*2}_{ih} \sigma^{*2}_{jl} \sigma^{*2}_{hj} \sigma^{*2}_{li} = n^2  O(p \lambda^4 T^3) = O(n^2 p \varphi^{4 + \frac{1}{\alpha}} \cdot \varphi^4) = n^2p b^2(\varphi) \cdot o(1) .
\end{eqnarray*}
Finally, $np \varphi^{6 -\frac{2}{\alpha}} =  n^2 p \varphi^{4  + \frac{1}{\alpha}} \cdot \ds\frac{\varphi^{2 - \frac{3}{\alpha}}}{n} = o(1)$  as soon as $n^2 p \varphi^{4 + \frac{1}{\alpha}}\to 0$ and   $\alpha>3/2$ or  $ 5/8 < \alpha < 3/2$ and $ p < n^{\frac{8 \alpha -5}{-2\alpha +3}}$. 

As consequence, if $n^2 p b^2(\varphi) \to 0$ with the additional conditions on $\alpha, n $ and $p$ given previously, we get
\[
\mathbb{E}_I \left( \ds\frac{dP_{\pi}}{dP_I} \right)^2 \leq \exp \Big( n^2 pb^2(\varphi) (1 +o(1) ) \Big) = 1 +o(1),
\]
which ends the proof of Theorem~\ref{theo:borneinfO}.

Now, we show (\ref{lan}) in order to finish the proof of Theorem \ref{theo:borneinf}. More explicitly,
\begin{eqnarray}
L_{n,p} & = & \log \frac{f_\pi}{f_I}(X_1,...,X_n)  \nonumber \\
&=& \log \mathbb{E}_U \exp\left(
- \frac 12 \sum_{k=1}^n X_k^\top ((\Sigma^*_U)^{-1}-I) X_k - \frac n2 \log \det(\Sigma^*_U)\right),
\label{logliklihoodratio}
\end{eqnarray}
where $U$ is seen as a randomly chosen matrix with uniform distribution over the set $\mathcal{U}$.
Let us denote $\Delta_U = \Sigma^*_U - I$ and recall that proposition \ref{cor1} implies that $ \| \Delta_U \| \leq O(1) \varphi^{1 - \frac 1{2 \alpha}} = o(1)$ for all $\alpha > 1/2$. We write the following approximations obtained by matrix Taylor expansion:
\begin{eqnarray}
 - \ds\frac 12 ((\Sigma^*_U)^{-1} - I) &=& \ds\frac 12 \ds\sum_{l=1}^5 (-1)^{l+1} \cdot \Delta_U^l + O(1) \Delta_U^6 \label{taylor1} \\
\log \det(\Sigma^*_U) &=& tr \Big( \ds\sum_{l=1}^5 \frac{(-1)^{l+1}}l \cdot \Delta_U^l + O(1) \Delta_U^6 \Big)  \label{taylor2}
\end{eqnarray}
Note that, $tr(\Delta_U) = 0$ and that $tr(\Delta^2_U) = \| \Sigma^* - I \|_F^2 = 2 \ds\sum_{1 \leq i <j \leq p} \sigma_{ij}^{*2} $ does not depend on $U$. Moreover,
\begin{eqnarray}
\mathbb{E}_I (\sum_{k=1}^n X_k^\top \Delta_U^6 X_k )= n \cdot tr(\Delta_U^6) & \leq &  n \cdot  \| \Delta_U \|^4 \cdot  tr( \Delta_U^2) \leq O(1) \cdot n \varphi^{4 - \frac 2{\alpha}} \cdot p \varphi^{2} \nonumber \\
& \leq & O(1)\cdot n \ds\sqrt{p} \varphi^{ 2 + \frac 1{2 \alpha}} \cdot \ds\sqrt{p} \varphi^{4 - \frac 5{2 \alpha}} \nonumber \\
&\leq& O(1) \cdot u_n \cdot \ds\sqrt{p} \varphi^{ 2 - \frac 1{2 \alpha}} \cdot \varphi^{2 - \frac 2{\alpha}} =o(1) \nonumber
 \label{trdelta6}
\end{eqnarray}
for all $\alpha > 1 $ and when $u_n = O(1)$ and $ \ds\sqrt{p} \varphi^{2 - \frac 1{2 \alpha}} = O(1) $.
Also, for all $\alpha > 1$
\[
\Var_I (\sum_{k=1}^n X_k^\top \Delta_U^6 X_k )= 2 n tr( \Delta_U^{12}) \leq O(1) n \varphi^{10 - \frac 5{\alpha}} \cdot p \varphi^2 = o(1).
\]
In conclusion, we use $ Y_n = \mathbb{E}_I(Y_n) + O_P(\ds\sqrt{\Var(Y_n)} )$ for any sequence of random variables $Y_n$, to get
\[
\sum_{k=1}^n X_k^\top \Delta_U^6 X_k  - ntr(\Delta_U^6) = o_P(1), \quad \text{ in } P_I \text{-probability}.
\]
We get
\begin{eqnarray}
L_{n,p} = \log \mathbb{E}_U \exp & \Big( & \ds\frac 12 \sum_{k=1}^n X_k^\top \Delta_U X_k -  \ds\frac 12 \sum_{k=1}^n X_k^\top \Delta_U^2 X_k  + \frac n4 tr(\Delta_U^2) \label{ln1} \\
 & + &  \frac 12 \sum_{l=3}^5 (-1)^{l+1} \sum_{k=1}^n X_k^\top \Delta_U^l X_k - \ds\frac n2 \sum_{l=3}^5 \frac{(-1)^{l+1}}{l} \cdot tr( \Delta_U^l) \Big) + o_P(1) .\nonumber
\end{eqnarray}
From $l=3,4$ and 5, we treat similarly the terms
\begin{eqnarray}
 \ \ds\sum_{k=1}^n X_k^\top \Delta_U^l X_k &=& \mathbb{E}_I (\ds\sum_{k=1}^n X_k^\top \Delta_U^l X_k) +   O_P\Big( \ds\sqrt{\Var_I(\ds\sum_{k=1}^n X_k^\top \Delta_U^l X_k)} \Big) \nonumber \\
&=& n tr(\Delta_U^l) + O_P(1) \cdot \ds\sqrt{ n tr(\Delta_U^{2l})}
\end{eqnarray}
By \eqref{trdelta6}, we have $ntr(\Delta_U^6) =o(1)$, similarly we obtain $n tr(\Delta_U^{2l})= o(1) $ for $l=4$ and 5. Thus \eqref{ln1} becomes :
\begin{eqnarray}
L_{n,p} = \log \mathbb{E}_U \exp & \Big( & \ds\frac 12 \sum_{k=1}^n X_k^\top \Delta_U X_k -  \ds\frac 12 \sum_{k=1}^n X_k^\top \Delta_U^2 X_k  + \frac n4 tr(\Delta_U^2) \nonumber  \\
 & + &  \frac n2 \sum_{l=3}^5 (-1)^{l+1} \Big( 1 - \frac 1l \Big) \cdot tr( \Delta_U^l) \Big) + o_P(1) .\label{ln2}
\end{eqnarray}
We have
\[
tr(\Delta_U^3)= \ds\sum_{\substack{i \ne j \ne k \\ k \ne i}} u_{ij} u_{jk} u_{ki} \sigma_{ij}^*  \sigma_{jk}^* \sigma_{ki}^*  = 3! \ds\sum_{i < j < k } u_{ij} u_{jk} u_{ki} \sigma_{ij}^*  \sigma_{jk}^* \sigma_{ki}^*
\]
and we decompose
\begin{eqnarray}
tr(\Delta_U^4) &=& \ds\sum_{ \substack{i \ne j \ne k \\ i \ne l \ne k}} u_{ij} u_{jk} u_{kl} u_{li} \sigma_{ij}^*  \sigma_{jk}^* \sigma_{kl}^*  \sigma_{li}^* \nonumber \\
 & = & \sum_{i \ne j} \sigma_{ij}^{*4} + 2 \sum_{i \ne j \ne k}   \sigma_{ij}^{*2}  \sigma_{jk}^{*2} \nonumber  + 4! \sum_{i<j<k<l}   u_{ij} u_{jk} u_{kl} u_{li} \sigma_{ij}^*  \sigma_{jk}^* \sigma_{kl}^*  \sigma_{li}^* . \nonumber
\end{eqnarray}
Note that
\begin{eqnarray*}
&&  n   \sum_{i \ne j} \sigma_{ij}^{*4}  =  O(np \varphi^{4 + \frac{1}{ \alpha} })= O(n \ds\sqrt{p} \varphi^{2 + \frac{1}{ 2\alpha} } \cdot \ds\sqrt{p} \varphi^{2 + \frac{1}{ 2\alpha} }) = o(1) \\
&& 2 n \sum_{i \ne j \ne k}   \sigma_{ij}^{*2}  \sigma_{jk}^{*2}
= O(n p \lambda^2 T^2) = O(n \ds\sqrt{p} \varphi^{2 + \frac{1}{ 2\alpha} } \cdot \ds\sqrt{p} \varphi^{2 - \frac{1}{ 2\alpha} }) = O(u_n \cdot \sqrt{p} \varphi^{2 - \frac{1}{ 2\alpha} } ) = o(1),
\end{eqnarray*}
if $ u_n \asymp 1$  and $ \sqrt{p} \varphi^{2 - \frac{1}{ 2\alpha} }  \to 0$.
As for the last term :
\begin{eqnarray*}
tr(\Delta_U^5) &=&  \ds\sum_{ \substack{i \ne j \ne k \\ k \ne l \ne v \\ v \ne i}} u_{ij} u_{jk} u_{kl} u_{lv} u_{vi} \sigma_{ij}^*  \sigma_{jk}^* \sigma_{kl}^*  \sigma_{lv}^* \sigma_{vi}^* \\
& = & 5  \ds\sum_{ \substack{i \ne j \ne k \\ k \ne l \ne j}}  u_{jk} u_{kl} u_{lj}  \sigma_{ij}^{*2}  \sigma_{jk}^* \sigma_{kl}^*  \sigma_{lj}^*
+ 5  \ds\sum_{ \substack{i \ne j \ne k \\ k \ne i}} u_{ij}^3 u_{jk} u_{ki} \sigma_{ij}^{*3}  \sigma_{jk}^* \sigma_{ki}^* \\
&+& 5! \ds\sum_{i<j<k<l<v} u_{ij} u_{jk} u_{kl} u_{lv} u_{vi} \sigma_{ij}^*  \sigma_{jk}^* \sigma_{kl}^*  \sigma_{lv}^* \sigma_{vi}^* .
\end{eqnarray*}
The first two terms in the decomposition of $tr(\Delta_U^5)$ group with $tr(\Delta_U^3)$ with extra factor $\ds\sum_{j : |j-i|<T} \sigma_{ij}^{*2} + \sigma_{ij}^{*2} = O(\lambda \cdot T) + O(\lambda) = o(1)$, therefore we ignore these terms in further calculations.

Let us denote by $W_{ij} = \ds\sum_{k=1}^n X_{k,i} X_{k,j}
$, then
$
\ds\sum_{k=1}^n X_k^\top \Delta_U X_k = \sum_{1 \leq i \ne j \leq p } u_{ij} \sigma^*_{ij}, W_{ij} \, ,
$
\begin{eqnarray*}
\ds\sum_{k=1}^n X_k^\top \Delta_U^2 X_k = \sum_{1 \leq i,j \leq p} [ \Delta_U^2]_{ij} W_{ij} = \sum_{1 \leq i \ne j \leq p } \sum_{h \notin \{ i,j \}} u_{ih} u_{hj} \sigma_{ih}^* \sigma_{hj}^* W_{ij} + \sum_{i=1}^p \sum_{h \ne i} \sigma_{ih}^{*2} W_{ii}
\end{eqnarray*}
and $
\ds\frac n4 tr(\Delta_U^2)= \frac n4 \sum_{1 \leq i \ne h \leq p } \sigma^{*2}_{ih}$. Then, from \eqref{ln2} we get
\begin{eqnarray}
L_{n,p} &=& \log \mathbb{E}_U \exp  \Big(  \ds\frac 12 \sum_{1 \leq i \ne j \leq p } u_{ij} \sigma^*_{ij} W_{ij} - \frac 12 \sum_{1 \leq i \ne h \leq p} \sigma_{ih}^{*2} \Big( W_{ii} - \frac n2 \Big) \nonumber \\
&-& \ds\frac 12 \sum_{1 \leq i \ne j \ne h \ne i \leq p}   u_{ih} u_{hj} \sigma_{ih}^* \sigma_{hj}^* W_{ij}
+  \ds\frac n2 \sum_{l=3}^5 (-1)^{l+1} \cdot \ds\frac{l-1}{l} \cdot tr(\Delta_U^l) \Big) \nonumber  + o_P(1) \nonumber \\
& =&  \log \mathbb{E}_U \exp  \Big(   \sum_{1 \leq i < j \leq p } u_{ij} \sigma^*_{ij} W_{ij} -  \sum_{ \substack{ 1 \leq i \ne j \ne h \ne i \leq p \\ i <j}}  u_{ih} u_{hj} \sigma_{ih}^* \sigma_{hj}^*  W_{ij} \nonumber \\
&+& \ds\frac n2 \sum_{l=3}^5 (-1)^{l+1} \cdot \frac{(l-1)}{l}  \cdot l! \sum_{k_1 < k_2 < \cdots < k_l} u_{k_1 k_2} \cdots u_{k_l k_1} \sigma^*_{k_1 k_2} \cdots \sigma^*_{k_l k_1} \Big) \nonumber \\
& -& \ds\frac 12 \sum_{1 \leq i \ne h \leq p} \sigma_{ih}^{*2} \Big( W_{ii} - \frac n2 \Big) + o_P(1).
\end{eqnarray}
Now, we explicit the expected value with respect to the i.i.d Rademacher random variables $u_{ij}, u_{ih}u_{hj}, u_{k_1k_2} u_{k_2 k_3}u_{k_3 k_1}, \dots $ for all $i<j, h, k_1 < k_2 < \dots < k_l$ pairwise distinct.  Indeed, products of independent Rademacher random variables are still Rademacher and independent.
Thus,
\begin{eqnarray}
L_{n,p} &=& \ds\sum_{1 \leq i < j \leq p} \log \cosh( \sigma_{ij}^* W_{ij}) + \sum_{ \substack{1 \leq i< j \leq p \\ h \notin \{ i,j \}}}
\log\cosh( \sigma_{ih}^* \sigma_{hj}^* W_{ij} ) \nonumber \\
&+& \sum_{l=3}^5 \ds\sum_{k_1 < \dots < k_l} \log\cosh \Big( \ds\frac{n (-1)^{l+1}}{2} \cdot (l-1) \cdot (l-1)! \cdot \sigma^*_{k_1k_2} \cdots \sigma_{k_l k_1}^* \Big) \nonumber \\
&-& \ds\frac 12 \ds\sum_{1 \leq i \ne h \leq p} \sigma_{ih}^{*2}( W_{ii} - \ds\frac n2 ) + o_P(1). \label{ln3}
\end{eqnarray}
We shall use repeatedly the Taylor expansion of $\log\cosh (u) = u^2/2- (u^4/12)(1 +o(1))$ as $u \to 0$.
Indeed, $\mathbb{E}_I(W_{ij})=0$ and $\mathbb{E}_I( |\sigma^*_{ij} W_{ij}|^2) \leq O(1) \cdot \lambda  n = O(1) \cdot n^{- \frac 1{(4 \alpha+ 1)}} p^{- \frac{2 \alpha + 1}{(4 \alpha +1)}} =o(1)$, giving that $|\sigma^*_{ij} W_{ij}| = o_P(1)$. Thus
\begin{equation}
\label{mainterm}
\log \cosh(\sigma^*_{ij} W_{ij}) = \frac 12 (\sigma^*_{ij} W_{ij})^2
-\frac 1{12} (\sigma^*_{ij} W_{ij})^4 (1 +o_P(1)).
\end{equation}
 Similarly, using the first order Taylor expansion, we get
\begin{equation*}
\label{termenegligeable1}
\log\cosh ( \sigma^*_{ih} \sigma^*_{hj} W_{ij}) = \ds\frac 12 ( \sigma^{*}_{ih} \sigma^{*}_{hj} W_{ij})^2 (1 +o_P(1))
\end{equation*}
and for $l=3, 4$ 	and 5,
\[
\log\cosh \Big( \ds\frac{n(-1)^{l+1}}{2} \cdot (l-1) \cdot  (l-1)! \cdot \sigma^*_{k_1k_2} \cdots \sigma_{k_l k_1}^* \Big) = \ds\frac{n^2 \Big((l-1)\cdot (l-1)! \Big)^2}{8} \cdot \sigma^{*2}_{k_1k_2} \cdots \sigma_{k_l k_1}^{*2} (1+ o(1))
\]
Recall now that $\sigma_{ij}^*=0$, for all $i,j$ such that $|i-j| \geq T $ and $\sigma_{ij}^{*2} \leq \lambda = O(1) \varphi^{2 + \frac 1{ \alpha}}$. Then,
$$
\mathbb{E}\Big(\sum_{ \substack{1 \leq i < j \leq p \\ h \notin \{i, j \}}}    \sigma^{*2}_{ih} \sigma^{*2}_{hj} W^2_{ij}  \Big) = O(np \lambda^2 T^2 ) = O( n p \varphi^{4 }) = O(1) \cdot n \ds\sqrt{p} \varphi^{2 + \frac 1{2 \alpha}} \cdot \ds\sqrt{p} \varphi^{2 - \frac 1{2 \alpha}}
= o(1),
$$
as soon as  $u_n \asymp 1$ and $\ds\sqrt{p} \varphi^{2 - \frac 1{2 \alpha}} \to 0$.
In conclusion, as the convergence in $\mathbb{L}_1(P_I)$ implies convergence in $P_I$ probability, we get
\begin{equation}
\label{convenproba1}
\sum_{ \substack{1 \leq i < j \leq p \\ h \notin \{i, j \}}}   \log\cosh ( \sigma^*_{ih} \sigma^*_{hj} W_{ij})  \longrightarrow 0 ~ \text{ in } P_I \text{ probability}.
\end{equation}
 Moreover, for $l=3, 4 $ and 5,
\begin{eqnarray}
n^2 \ds\sum_{k_1 < \dots < k_l} \ds\frac{\Big( (l-1)\cdot (l-1)! \Big)^2}{8} \cdot \sigma^{*2}_{k_1k_2} \cdots \sigma_{k_l k_1}^{*2}  &=& O(n^2 p \lambda^l T^{l-1}) =O(n^2 p \varphi^{2l + \frac 1{\alpha}}) = o(1) . \nonumber \\ \label{termesnegligeable}
\end{eqnarray}
Using \eqref{convenproba1} and \eqref{termesnegligeable}, \eqref{ln3} gives
\begin{equation}
\label{last}
L_{n,p} = \frac 12 \sum_{1 \leq i < j \leq p} \!\! \sigma^{*2}_{ij} W^2_{ij}
-\frac 1{12}  \sum_{1 \leq i < j \leq p} \!\! \sigma^{*4}_{ij} W^4_{ij} (1 +o_P(1))
-  \frac 12 \sum_{1 \leq i < j \leq p} \!\! \sigma^{*2}_{ij}\Big( W_{ii} - \frac n2 \Big) + o_P(1)
\end{equation}
we further decompose as follows :
\[
\frac 12 \sum_{1 \leq i < j \leq p} \sigma^{*2}_{ij} W^2_{ij} = \frac 12 \sum_{k=1}^n\sum_{1 \leq i < j \leq p} \sigma^{*2}_{ij} X^2_{k,i}X^2_{k,j} + \frac 12 \sum_{1\leq k \ne l \leq n}\sum_{1 \leq i < j \leq p} \sigma^{*2}_{ij} X_{k,i}X_{k,j} X_{l,i}X_{l,j}
\]
With our definition: $\sigma^{*2}_{ij} =2 \cdot  w^*_{ij} \cdot  b(\varphi)$, we can write
$$
\frac 12 \sum_{1\leq k \ne l \leq n}\sum_{1 \leq i < j \leq p} \sigma^{*2}_{ij} X_{k,i}X_{k,j} X_{l,i}X_{l,j} = n\sqrt{p} b(\varphi) \cdot (n-1)\sqrt{p} \cdot \widehat{\mathcal{D}}_n
$$
and we put $Z_n = (n-1) \sqrt{p} \cdot \widehat{\mathcal{D}}_n$ which has asymptotically standard Gaussian under $P_I$ probability, by Proposition~\ref{prop:espvar}.

By Proposition \ref{prop:Wijproperty} given in the Appendix, we have $\mathbb{E}(W_{ij}^4) = 3n^2 (1 +o(1))$, then
\[
\frac 1{12} \cdot \mathbb{E}_I \Big( \sum_{1 \leq i<j \leq p} \sigma_{ij}^{*4} W_{ij}^4 \Big) = \frac 1{12} \cdot 3n^2  \sum_{1 \leq i<j \leq p} \sigma_{ij}^{*4} (1 +o(1))= \frac{u_n^2}{2} (1 +o(1)).
\]
Moreover,
\begin{eqnarray*}
\Var_I(\ds\sum_{i<j} \sigma_{ij}^{*4} W_{ij}^4) &=&  \ds\sum_{i<j} \sigma_{ij}^{*8} \Var_I( W_{ij}^4) + \underset{1 \leq i < j \ne j' \leq p}{\sum \sum \sum}  \sigma_{ij}^{*4}  \sigma_{ij'}^{*4} \Cov_I(W_{ij}^4, W_{ij'}^4) \\
& =&  O(n^4 p \lambda^4 T   ) +  O(n^3 p \lambda^4 T^2   )  \\
&= & O( n^4 p \varphi^{8 + \frac{3}{ \alpha}}  ) + O(n^3 p  \varphi^{8 + \frac{2}{ \alpha}} ) =o(1).
\end{eqnarray*}
We deduce  that,
\begin{eqnarray*}
 \frac 1{12} \sum_{1 \leq i < j \leq p} \sigma^{*4}_{ij} W^4_{ij}  =   \ds\frac{u_n^2}{2} +o_P(1)
\end{eqnarray*}
Remaining terms in (\ref{last}) can be grouped as follows:
\begin{eqnarray*}
&& \frac 14 \sum_{k=1}^n \sum_{1 \leq i \ne j \leq p} \sigma^{*2}_{ij} (X^2_{k,i} - 1)( X^2_{k,j} - 1) = o_P(1)
\end{eqnarray*}
since the random variable in the previous display is centered and
\begin{eqnarray*}
\mathbb{E}_I \Big( \sum_{k=1}^n \sum_{1 \leq i \ne j \leq p} \sigma^{*2}_{ij} (X^2_{k,i} - 1)( X^2_{k,j} - 1) \Big)^2
& =&  \sum_{k=1}^n \sum_{1 \leq i \ne j \leq p} \sigma^{*4}_{ij} \cdot  \mathbb{E}_I ( X^2_{k,i} - 1)^2  \mathbb{E}_I ( X^2_{k,j} - 1)^2  \\
&= & 4  n \sum_{1 \leq i < j \leq p} \sigma^{*4}_{ij} = O(n p b^2(\varphi)) = o(1),
\end{eqnarray*}
which concludes the proof of (\ref{lan}). 
\hfill \end{proof}

\begin{proof}[Proof of Theorem \ref{theo:adaptivity}]
The type I error probability tends to $0$ as a consequence of the Berry-Essen type inequality in Lemma 1 in the Appendix applied to the degenerate U-statistic $\widehat{\mathcal{D}}_{n,p}$. We have that, for some $\varepsilon \in (0, 1/2)$ and any $t>0$ :
\[
\Big| \mathbb{P}_I( \widehat{\mathcal{D}}_{n,r} \leq t ) - \Phi ( n \ds\sqrt{p} \cdot t ) \Big| \leq 16 \varepsilon^{1/2} \exp( - \frac { n^2p t^2}4) + O \Big( \ds\frac 1n \Big) + O \Big( \ds\frac 1{pT_r} \Big) \quad  \text{ for all } 1 \leq r \leq N.
\]
We use the relation $1- \Phi (u) \leq (1/u) \exp(-u^2/2)$ for all  $u \in \mathbb{R}$, to deduce that
\[
 \mathbb{P}_I(\widehat{\mathcal{D}}_{n,r} > x )  \leq  \Big( \ds\frac{1}{ n \ds\sqrt{p} \cdot x}  + 16 \varepsilon^{1/2}  \Big) \exp( - \frac { n^2p x^2}4) + O \Big( \ds\frac 1n \Big) + O \Big( \ds\frac 1{pT_r} \Big)
\]

We use this previous result to show  that the type I error probability tends to 0.  See that for all $r \in  \{1, \cdots, N \}$, $C_{\lambda_r} \geq c(\underline{\alpha}, \bar{\alpha})$, where  $c(\underline{\alpha}, \bar{\alpha}) = (2 \underline{\alpha} +1)/( 2 \bar{\alpha} (4\bar{\alpha}+1)^{\frac{1}{2 \underline{\alpha}}})$. Thus since
$ n \sqrt{p} \cdot t_r =  C_{\lambda_r} \ds\sqrt{\ln \ln(n \sqrt{p})} $, we obtain that $ n \sqrt{p} \cdot t_r \geq c(\underline{\alpha}, \bar{\alpha}) \ds\sqrt{\ln \ln(n \sqrt{p})}=:t$ for all $r \in \{1, \cdots, N \}.$ Recall that  $N= \lceil \ln(n\ds\sqrt{p}) \rceil $, therefore
\begin{eqnarray*}
 && \mathbb{P}_I( \Delta^*_{ad} =1) =  \mathbb{P}_I ( \exists r \in \{1, \dots, N \} ;  \widehat{\mathcal{D}}_{n,r} >  \mathcal{C}^* t_r) \leq \ds\sum_{r =1}^N \mathbb{P}_I(\widehat{\mathcal{D}}_{n,r} > \mathcal{C}^* t_r) \\
 & \leq & \ds\sum_{r =1}^N \Big\{ \Big(  \ds\frac{1}{n \ds\sqrt{p} \cdot \mathcal{C}^*  t_r }   + 16 \varepsilon^{1/2}  \Big)   \exp( - \frac { n^2p \cdot \mathcal{C}^{*2} t_r^2}4)+ O \Big( \ds\frac 1n \Big) + O \Big( \ds\frac 1{pT_r} \Big) \Big\}   \\
 & \leq & N \Big(  \ds\frac{1}{n \ds\sqrt{p} \cdot  c(\underline{\alpha}, \bar{\alpha})\mathcal{C}^* t }   + 16 \varepsilon^{1/2}  \Big)   \exp( - \frac { n^2p \cdot c^2(\underline{\alpha}, \bar{\alpha}) \mathcal{C}^{*2} t^2}4) \nonumber
 +  O \Big( \ds\frac Nn \Big) + O \Big( \ds\frac 1p \Big)  \ds\sum_{r =1}^N \ds\frac{1}{T_r}\\
 & \leq &  \Big(  \ds\frac{1}{c(\underline{\alpha}, \bar{\alpha}) \mathcal{C}^* \ds\sqrt{\ln \ln(n \ds\sqrt{p})} }   + 16 \varepsilon^{1/2}  \Big)   (\ln(n\ds\sqrt{p}))^{1 -  (c(\underline{\alpha}, \bar{\alpha})\mathcal{C}^{*}/2)^2}  \nonumber
 +  \ds\frac {O(\ln (n \ds\sqrt{p})}n  +  \ds\frac {O(1)}p  \ds\sum_{r =1}^N \ds\frac{1}{T_r}.
 \end{eqnarray*}
 See that :
 \[
  \ds\frac{1}{p}  \ds\sum_{r =1}^N \ds\frac{1}{T_r} =  \ds\frac{1}{p} \ds\sum_{r =1}^N  ( \psi_{\alpha_r})^{\frac{1}{\alpha_r}} = \ds\frac{1}{p} \sum_{r=1}^N \ds\frac 1p \Big( \ds\frac{\rho_{n,p}}{n\ds\sqrt{p}} \Big)^{\frac 2{4 \alpha_r +1}} \leq \ds\frac{N}{p} \Big( \ds\frac{\rho_{n,p}}{n\ds\sqrt{p}} \Big)^{\frac 2{4 \bar{\alpha} +1}} = o(1).
 \]
Moreover  if $ (\ln p) /n =o(1) $, then $\ln (n \sqrt{p})/n = o(1)$,  and if $\mathcal{C}^* \geq 2/ c( \underline{\alpha} , \bar{\alpha})$ we obtain
\[
 \mathbb{P}_I( \Delta^*_{ad} =1) =o(1).
\]

Now, we move to the type II error probability.
Let us consider $\Sigma \in \mathcal{F}(\alpha,L) $ such that $(1/2p) \| \Sigma -I\|_F^2 = (1/p) \sum_{i<j} \sigma_{ij}^2 \geq (\mathcal{C} \psi_{\alpha})^2$ for some  $\alpha \in \mathcal{A}$. We defined  $\alpha_{r_0}$ as the smallest point on the grid such that $ \alpha \leq \alpha_{r_0}$. We denote by  $\widehat{\mathcal{D}}_{n,r_0}$, $t_{r_0}$, $\lambda_{r_0}$, $b_{r_0}$ and $T_{r_0}$ the test statistic, the threshold and the parameters depending on $\alpha_{r_0}$.
Also we define $C_{T_{r_0}}$, $C_{\lambda_{r_0}}$ and $C_{b_{r_0}}$ the constants define in \eqref{constants} for $\alpha_{r_0}$ instead of $\alpha$ and $L=1$. We  have $C_{b_{r}} < 1$ and $C_{T_r} > 1$, for all $r \in \{1, \cdots N \}$. The type II error probability is bounded from above as follows, $\forall \alpha \in [\underline{\alpha}, \bar{\alpha}]$ and $\forall \Sigma \in Q(\alpha,L, \mathcal{C} \psi_\alpha)$ :
\begin{eqnarray*}
\mathbb{P}_{\Sigma}( \Delta^*_{ad} =0) &= & \mathbb{P}_{\Sigma} \Big( \forall 1 \leq r \leq N,\widehat{\mathcal{D}}_{n,r} \leq \mathcal{C}^* t_r \Big) \leq  \mathbb{P}_{\Sigma} \Big( \widehat{\mathcal{D}}_{n,r_0} \leq  \mathcal{C}^*  t_{r_0} \Big) \\
& \leq &  \mathbb{P}_{\Sigma} \Big( \mathbb{E}_{\Sigma}(\widehat{\mathcal{D}}_{n,r_0}) - \widehat{\mathcal{D}}_{n,r_0}   \geq \mathbb{E}_{\Sigma}(\widehat{\mathcal{D}}_{n,r_0}) -     \mathcal{C}^* t_{r_0} \Big)
\end{eqnarray*}
First we have
 \begin{eqnarray*}
 \mathbb{E}_{\Sigma}(\widehat{\mathcal{D}}_{n,r_0}) & = & \ds\frac 1p \sum_{i<j} w_{ij, r_0}^* \sigma_{ij}^2  =  \ds\frac 1p  \cdot \ds\frac{\lambda_{r_0}}{b_{r_0}} \ds\sum_{i<j} \Big( 1 - \Big(\ds\frac{|i-j|}{T_{r_0}} \Big)^{2 \alpha_{r_0}} \Big)_+ \sigma_{ij}^2  \\
 & \geq & \ds\frac 1p  \cdot \ds\frac{\lambda_{r_0} }{b_{r_0} }  \Big( \ds\sum_{i<j} \sigma_{ij}^2 -  \ds\sum_{\substack{i<j \\  |i-j| > T_{r_0}}} \sigma_{ij}^2 - \ds\sum_{ \substack{i<j \\ | i-j | < T_{r_0}}} \ds\frac{|i-j|^{2 \alpha_{r_0}} }{T_{r_0}^{2 \alpha_{r_0}}} \cdot \sigma_{ij}^2 \Big)  \\
&\geq & \ds\frac{\lambda_{r_0} }{b_{r_0} } \Big( \mathcal{C}^2 \cdot \psi_{\alpha}^2 - \ds\sum_{ i<j } \ds\frac{|i-j|^{2 \alpha} }{T_{r_0}^{2 \alpha}}  \cdot \sigma_{ij}^2  \Big)  \geq  \ds\frac{\lambda_{r_0} }{b_{r_0}} \Big( \mathcal{C}^2 \cdot \psi_{\alpha}^2 - L \cdot T_{r_0}^{-2\alpha} \Big) \\
&\geq & C_{\lambda_{r_{\mathsmaller{0}}}} \cdot( C_{b_{r_0}})^{-\frac 12}  (\psi_{\alpha_{r_0}})^{\frac 1{2 \alpha_{r_0}}} \Big( \mathcal{C}^2 \cdot \psi_{\alpha}^2  - (C_{T_{r_0}})^{-2\alpha} \cdot (\psi_{ \alpha_{r_0}})^{\frac {4 \alpha }{2 \alpha_{r_0}}} \Big) \\
&\geq & C_{\lambda_{r_{\mathsmaller{0}}}}  \Big( \mathcal{C} \cdot (\psi_{ \alpha_{r_0}})^{\frac 1{2 \alpha_{r_0}}} \cdot \psi_{\alpha}^2  - (C_{T_{r_0}})^{-2\alpha} \cdot (\psi_{\alpha_{r_0}})^{\frac {4 \alpha +1}{2 \alpha_{r_0}}} \Big)
= : (E_1 - E_2).
 \end{eqnarray*}
Now we show that, since $\alpha < \alpha_{r_0}$ we have,
 \begin{eqnarray*}
E_1 \cdot  t_{r_0}^{-1} & =& \mathcal{C} \cdot (\psi_{\alpha_{r_0}})^{\frac 1{2 \alpha_{r_0}}} \cdot  (\psi_{\alpha})^{2} \cdot  ((n \sqrt{p}) / \rho_{n,p}) \nonumber \\
&=&  ((n \sqrt{p}) / \rho_{n,p})^{\frac{4(\alpha_{r_0} -\alpha )}{(4\alpha_{r_0} +1)(4\alpha +1)}}  \geq  \mathcal{C}
 \end{eqnarray*}
Moreover, use that $ 0> \alpha - \alpha_{r_0} \geq - (\bar{\alpha} - \underline{\alpha})/ \ln(n \ds\sqrt{p})  $, to obtain
\begin{eqnarray}
t_{r_0} \cdot E_2^{-1} & =&
 (\rho_{n,p}/ (n \sqrt{p}) ) \cdot (C_{T_{r_0}})^{2\alpha} \cdot (\psi_{ \alpha_{r_0}})^{-\frac {4 \alpha +1}{2 \alpha_{r_0}}}
 \geq    ((n \sqrt{p}) / \rho_{n,p})^{ \frac{4 (\alpha - \alpha_{r_0} )}{(4 \alpha_{r_0} +1)}} \nonumber \\[0.4 cm]
& = &   \exp \Big\{ \frac{4 (\alpha - \alpha_{r_0} )}{4 \alpha_{r_0} + 1} \cdot \ln ((n \sqrt{p}) / \rho_{n,p}) \Big\} \nonumber \\
&  \geq &    \exp \Big\{ - \frac{ 4 (\bar{\alpha} - \underline{\alpha})}{4 \underline{\alpha} + 1} (1 +o(1)) \Big\} \geq C( \underline{\alpha}, \bar{\alpha} ) .\nonumber
\end{eqnarray}
We deduce that,
 \[
 \mathbb{E}_{\Sigma}(\widehat{\mathcal{D}}_{n,r_0}) \geq  \Big( \mathcal{C}  - \ds\frac{1}{ C( \underline{\alpha}, \bar{\alpha} )} \Big) \cdot t_{r_0}.
 \]
 Let us denote by $\mathcal{T}_1$ and $\mathcal{T}_2$ the right-hand side termes in (\ref{T1}) and (\ref{T2}), respectively.
Then  by Markov inequality, for $\mathcal{C} -\ds\frac{1}{ C( \underline{\alpha}, \bar{\alpha} )} - \mathcal{C}^* >0 $, we get
 \begin{eqnarray*}
\mathbb{P}_{\Sigma}( \Delta^*_{ad} =0)  &\leq &  \mathbb{P}_{\Sigma} \Big( |\widehat{\mathcal{D}}_{n,r_0} - \mathbb{E}_{\Sigma}(\widehat{\mathcal{D}}_{n,r_0}) |     \geq \mathbb{E}_{\Sigma}(\widehat{\mathcal{D}}_{n,r_0}) -     \mathcal{C}^* t_{r_0} \Big)    \\[0.3 cm]
& \leq & \ds\frac{\Var_\Sigma(\widehat{\mathcal{D}}_{n, r_0} )}{ \Big( \mathbb{E}_{\Sigma}(\widehat{\mathcal{D}}_{n,r_0}) -   \mathcal{C}^*  t_{r_0} \Big)^2 }
 \leq   \ds\frac{\Big( \mathcal{C} - \ds\frac{1}{ C( \underline{\alpha}, \bar{\alpha} )}  \Big)^2 \Var_\Sigma(\widehat{\mathcal{D}}_{n, r_0} )}{  \Big( \mathcal{C} - \ds\frac{1}{ C( \underline{\alpha}, \bar{\alpha} )}  -  \mathcal{C}^* \Big)^2 \mathbb{E}^2_{\Sigma}(\widehat{\mathcal{D}}_{n,r_0}) } \\[0.3 cm]
& \leq & \ds\frac{\Big( \mathcal{C} - \ds\frac{1}{ C( \underline{\alpha}, \bar{\alpha} )}  \Big)^2 \cdot (\mathcal{T}_1 +(n-1)\mathcal{T}_2)}{n(n-1)p^2    \Big( \mathcal{C} - \ds\frac{1}{ C( \underline{\alpha}, \bar{\alpha} )}  -  \mathcal{C}^* \Big)^2 \mathbb{E}^2_{\Sigma}(\widehat{\mathcal{D}}_{n,r_0})}
 :=  F_1 + F_2.
\end{eqnarray*}
We use \eqref{T_1} to show that $F_1$ tends to zero.
\begin{eqnarray*}
F_1 &:=& \ds\frac{\Big( \mathcal{C} - \ds\frac{1}{ C( \underline{\alpha}, \bar{\alpha} )}  \Big)^2 \cdot \mathcal{T}_1 }{n(n-1)p^2    \Big( \mathcal{C} - \ds\frac{1}{ C( \underline{\alpha}, \bar{\alpha} )}  -  \mathcal{C}^* \Big)^2 \mathbb{E}^2_{\Sigma}(\widehat{\mathcal{D}}_{n,r_0})} \\
& \leq  & \ds\frac{1 +o(1)}{n(n-1)p  \Big( \mathcal{C} - \ds\frac{1}{ C( \underline{\alpha}, \bar{\alpha} )}  -  \mathcal{C}^* \Big)^2 t_{r_0}^2 } + \ds\frac{O(T_{r_0}^{\frac 32} \cdot t_{r_0}) }{n(n-1)p t_{r_0}^2}  = O(\rho_{n,p}^{-2}) =o(1),
\end{eqnarray*}
since $T_{r_0}^{\frac 32} \cdot   t_{r_0} = O \Big( ( \rho_{n,p} / n \sqrt{p})^{ - \frac 3{4\alpha_{r_0} +1} +1} \Big) =o(1) $  for $\alpha_{r_0} > 1/2$.
Similarly we use \eqref{T_2} to show that
\[
F_2:=\ds\frac{\Big( \mathcal{C} - \ds\frac{1}{ C( \underline{\alpha}, \bar{\alpha} )}  \Big)^2 \cdot \mathcal{T}_2}{np^2    \Big( \mathcal{C} - \ds\frac{1}{ C( \underline{\alpha}, \bar{\alpha} )}  -  \mathcal{C}^* \Big)^2 \mathbb{E}^2_{\Sigma}(\widehat{\mathcal{D}}_{n,r_0})}  = o(1).
\]
 Thus we get, for $\mathcal{C} -\ds\frac{1}{ C( \underline{\alpha}, \bar{\alpha} )} - \mathcal{C}^* >0 $,
\[
\sup\limits_{ \alpha \in [\,  \underline{\alpha} \, , \bar{\alpha} \,  ]} \sup\limits_{ \substack{   \Sigma \in \mathcal{F}(\alpha,L) \,  ; \\  \frac 1{2p} \| \Sigma -I\|_F^2  \geq \mathcal{C}^2 \psi^2_{ \alpha}}} \mathbb{P}_{\Sigma}( \Delta^*_{ad} =0) =o(1).
\]
\hfill
 \end{proof}




\section{Appendix}

\begin{proof}[Proof of Proposition \ref{prop:espvar}]
We recall that under the null hypothesis the coordinates of the vector $X_k$ are independent, so using this fact we have :

$$
\begin{array}{ll}
\Var_I(\widehat{\mathcal{D}}_n)
   &= \displaystyle\frac{2}{n^2(n-1)^2p^2} \Var( \underset{i<j}{ \ds\sum_{i=1}^p \sum_{j=1}^p}  \underset{k \neq l}{\ds\sum_{l=1}^n\sum_{k=1}^n} w^*_{ij}
   X_{k,i}X_{k,j}X_{l,i}X_{l,j})\\
&= \displaystyle\frac{2}{n(n-1)p^2} \underset{i<j}{ \ds\sum_{i=1}^p \sum_{j=1}^p}w_{ij}^{*2} \mathbb{E}^4(X_{1,i}^2)
   = \ds\frac{2}{n(n-1)p^2}\underset{i<j}{ \ds\sum_{i=1}^p \sum_{j=1}^p} w_{ij}^{*2}
  =  \ds\frac{1}{n(n-1)p}
  \end{array}
$$
For  $\Sigma \in Q(\alpha , L , \varphi)$,
$$
\begin{array}{ll}
\mathbb{E}_{\Sigma}(\widehat{\mathcal{D}}_n )
     &= \ds\frac{1}{n(n-1)p}  \underset{i<j}{ \ds\sum_{i=1}^p \sum_{j=1}^p}  \underset{k \neq l}{\ds\sum_{l=1}^n\sum_{k=1}^n} w^*_{ij} \mathbb{E}(X_{k,i}X_{k,j}X_{l,i}X_{l,j})
     \\[0.5 cm]
     &= \ds\frac{1}{p} \underset{i<j}{ \ds\sum_{i=1}^p \sum_{j=1}^p} w^*_{ij}\mathbb{E}(X_{1,i}X_{1,j})\mathbb{E}(X_{2,i}X_{2,j})
     =  \ds\frac{1}{p} \underset{i<j}{ \ds\sum_{i=1}^p \sum_{j=1}^p} w^*_{ij} \sigma_{ij}^2 \\[0.5 cm]
\end{array}
$$
Remark that $\widehat{\mathcal{D}}_n - \mathbb{E}_{\Sigma}(\widehat{\mathcal{D}}_n)$ can be written as the following form
\begin{eqnarray}\label{decomp}
\widehat{\mathcal{D}}_n - \mathbb{E}_{\Sigma}(\widehat{\mathcal{D}}_n)&=& \ds\frac{1}{n(n-1)p} \underset{k \neq l}{\ds\sum_{l=1}^n\sum_{k=1}^n}\underset{i<j}{ \ds\sum_{i=1}^p \sum_{j=1}^p} w^*_{ij}(X_{k,i}X_{k,j} - \sigma_{ij})
(X_{l,i}X_{l,j} - \sigma_{ij}) \nonumber \\
 &&+ \ds\frac{2}{np} \sum_{k=1}^n  \underset{i<j}{ \ds\sum_{i=1}^p \sum_{j=1}^p} w^*_{ij}
 (X_{k,i}X_{k,j} - \sigma_{ij})\sigma_{ij}
 \end{eqnarray}
Then the  variance of the estimator $\widehat{\mathcal{D}}_n $  is a sum of two uncorrelated terms
\begin{equation}
\label{var2termes}
\begin{array}{lcl}
\Var_{\Sigma}(\widehat{\mathcal{D}}_n)
&=& \ds\frac{2}{n(n-1)p^2} \mathbb{E}_{\Sigma}
 \{\underset{i<j}{ \ds\sum_{i=1}^p \sum_{j=1}^p}  w^*_{ij}(X_{1,i}X_{1,j} - \sigma_{ij})
(X_{2,i}X_{2,j} - \sigma_{ij})\}^2 \\
&&+  ~ \ds\frac{4}{np^2} \mathbb{E}_{\Sigma} \{\underset{i<j}{ \ds\sum_{i=1}^p \sum_{j=1}^p} w^*_{ij}
 (X_{k,i}X_{k,j} - \sigma_{ij})\sigma_{ij}\}^2
 \end{array}
\end{equation}
Now we will give an upper bound for the first term on the right-hand side of (\ref{var2termes}).
Denote by
\[
\begin{array}{lcl}
T_1 &=& 2\mathbb{E}_{\Sigma}
 \{\underset{i<j}{ \ds\sum_{i=1}^p \sum_{j=1}^p}  w^*_{ij}(X_{1,i}X_{1,j} - \sigma_{ij})
(X_{2,i}X_{2,j} - \sigma_{ij})\}^2  \\
 &=&\ds\frac{1}{2}\underset{i \neq j}{ \ds\sum_{i=1}^p \sum_{j=1}^p}\underset{i' \neq j'}{ \ds\sum_{i'=1}^p \sum_{j'=1}^p} w^*_{ij} w^*_{i'j'}\mathbb{E}_{\Sigma}^2\{ (X_{1,i}X_{1,j} - \sigma_{ij})(X_{1,i'}X_{1,j'} - \sigma_{i'j'}) \} \\
  &=&\ds\frac{1}{2}\underset{i \neq j}{ \ds\sum_{i=1}^p \sum_{j=1}^p}\underset{i' \neq j'}{ \ds\sum_{i'=1}^p \sum_{j'=1}^p} w^*_{ij} w^*_{i'j'}
  ( \sigma_{ii'}\sigma_{jj'} + \sigma_{ij'}\sigma_{i'j})^2
\end{array}
\]
We shall distinguish three terms in the previous sum, that is $(i,j,i',j') \in A_1 \cup A_2 \cup A_3,$ where  $A_1, A_2,A_3$ form a partition of the set$ \{(i,j,i',j') \in \{ 1, \dots, p\}^4 \text{ such that } i \neq j, i' \neq j' \}$.  More precisely in $A_1$ we have $(i,j)=(i',j')$ or $(i,j)=(j',i')$, in $A_2$ we have three different indices
 $(i=i' \text{ and } j \neq j')$ or $(j=j' \text{ and } i \neq i')$ or $(i=j'\text{ and } j \neq i')$ or $(j=i' \text{ and } i \neq j') $ and finally in $A_3$ the indices are pairewise distinct.
First, when $(i,j,i',j') \in A_1$, we use that $\Var_{\Sigma}(X_{1,i}X_{1,j}) = (1 + \sigma_{ij}^2)^2$, to get
\begin{eqnarray}
T_{1,1} &=& \underset{i \neq j}{ \ds\sum_{i=1}^p \sum_{j=1}^p}
 w^{*2}_{ij}(1 + \sigma_{ij}^2)^2 = \underset{i \neq j}{ \ds\sum_{i=1}^p \sum_{j=1}^p}  w^{*2}_{ij} + \underset{i \neq j}{ \ds\sum_{i=1}^p \sum_{j=1}^p}w^{*2}_{ij} (2 \sigma_{ij}^2 + \sigma_{ij}^4 ) \nonumber \\
&\leq& p + 3\underset{i \neq j}{ \ds\sum_{i=1}^p \sum_{j=1}^p} w^{*2}_{ij}  \sigma_{ij}^2 \leq p + 6\cdot  p \cdot L \cdot \sup\limits_{i,j} w^{*2}_{ij}  \label{T_{1,1}}
\end{eqnarray}
and this is $p(1+o(1))$ since $ \sup\limits_{i,j} w^{*2}_{ij} \asymp(1/T ) \to 0.$
When the indices are in $A_2$, we have three indices out of four which are equal. We assume $i=i'$, therefore it is sufficient to check that,
\[
\begin{array}{lcl}
T_{1,2} &=& 2\ds\sum_{i=1}^p \underset{j \neq j'}{\sum_{\substack{j=1 \\ j \neq i}}^p\sum_{\substack{j'=1 \\ j' \neq i}}^p} w^{*}_{ij} w^{*}_{ij'} ( \sigma_{jj'} + \sigma_{ij}\sigma_{ij'})^2  \\
&\leq& 4 \ds\sum_{i=1}^p \underset{j \neq j'}{\ds\sum_{\substack{j=1 \\ j \neq i}}^p\sum_{\substack{j'=1 \\ j' \neq i}}^p} w^{*}_{ij} w^{*}_{ij'} \sigma_{jj'}^2 + 4 \ds\sum_{i=1}^p\underset{j \neq j'}{\ds\sum_{\substack{j=1 \\ j \neq i}}^p\sum_{\substack{j'=1 \\ j' \neq i}}^p} w^{*}_{ij} w^{*}_{ij'} \sigma_{ij}^2\sigma_{ij'}^2 \\
\end{array}
\]
Now let us  bound from above the first term of $T_{1,2} $,
\begin{equation}
\label{T_{1,2,1}}
\begin{array}{lcl}
T_{1,2,1} := \ds\sum_{i=1}^p \underset{ j \neq j'}{\ds\sum_{\substack{j=1 \\ j \neq i}}^p\sum_{\substack{j'=1 \\ j' \neq i}}^p} w^{*}_{ij} w^{*}_{ij'} \sigma_{jj'}^2 &\leq &  \ds\sum_{i=1}^p \underset{|j -j'| <T}{\ds\sum_{\substack{j=1 \\ j \neq i}}^p\sum_{\substack{j'=1 \\ j' \neq i}}^p} w^{*}_{ij} w^{*}_{ij'} \sigma_{jj'}^2  +  \underset{ |j- j'|  \geq T}{\ds\sum_{\substack{j=1 \\ j \neq i}}^p\sum_{\substack{j'=1 \\ j' \neq i}}^p}
\ds\frac{|j- j'|^{2 \alpha}}{T^{2 \alpha}} \sigma_{jj'}^2 \ds\sum_{i=1}^p w^{*}_{ij} w^{*}_{ij'} \\[0.5cm]
\end{array}
\end{equation}
Again we will treat each term  of $T_{1,2,1}$ separately. We recall that the weights $w_{ij}^*$  verify the following properties
\[
 ( w_{ij}^* \geq w_{i'j'}^*  \quad \text{for } | i-j| \leq |i'-j'| )~ \quad \mbox{ and } \quad  \sum_{i=1}^p w_{ij}^* \asymp \ds\sqrt{T}.
\]
In the rest of the proof we denote by $ k_0( \alpha , L) , k_1( \alpha , L) , \ldots$ different constants that dependent only  on $ \alpha$ and/or on $L$. We have for $ \alpha > 1/2$,
\begin{eqnarray}
T_{1,2,1,1} &:=& \ds\sum_{i=1}^p \underset{|j -j'| <T}{\ds\sum_{\substack{j=1 \\ j \neq i}}^p \sum_{\substack{j'=1 \\ j' \neq i}}^p} w^{*}_{ij} w^{*}_{ij'} \sigma_{jj'}^2 = \ds\sum_{i=1}^p \underset{|j -j'| \leq |i-j |<T}{\ds\sum_{\substack{j=1 \\ j \neq i}}^p\sum_{\substack{j'=1 \\ j' \neq i}}^p} w^{*}_{ij} w^{*}_{ij'} \sigma_{jj'}^2  + \ds\sum_{i=1}^p \underset{ |i-j |<|j -j'| <T}{\ds\sum_{\substack{j=1 \\ j \neq i}}^p\sum_{\substack{j'=1 \\ j' \neq i}}^p} w^{*}_{ij} w^{*}_{ij'} \sigma_{jj'}^2  \nonumber \\
& \leq &  \underset{|j -j'| <|i-j |<T}{\ds\sum_{\substack{j=1 \\ j \neq i}}^p\sum_{\substack{j'=1 \\ j' \neq i}}^p}  w^{*}_{jj'} \sigma_{jj'}^2   \ds\sum_{i=1}^p w^{*}_{ij'} + \ds\sum_{i=1}^p \underset{ |i-j |<|j -j'| <T}{\ds\sum_{\substack{j=1 \\ j \neq i}}^p\sum_{\substack{j'=1 \\ j' \neq i}}^p} \frac{w^{*}_{ij} w^{*}_{ij'} |j -j'|^{2 \alpha}}{|i-j |^{2 \alpha}} \sigma_{jj'}^2   \nonumber \\
& \leq& k_0(\alpha, L) \cdot \ds\sqrt{T} \cdot p \cdot  \mathbb{E}_{\Sigma}(\widehat{\mathcal{D}}_n) + (\sup\limits_{i,j} w_{ij}^*)^2 \sum_{j=1}^p \sum_{\substack{j'=1 \\ j\neq j'}}^p |j -j'|^{2 \alpha}\sigma_{jj'}^2 (\sum_{i=1}^p \frac{1}{|i-j |^{2 \alpha}})   \nonumber \\
& \leq &k_0(\alpha, L) \cdot \ds\sqrt{T} \cdot p \cdot  \mathbb{E}_{\Sigma}(\widehat{\mathcal{D}}_n) + k_1( \alpha , L)\cdot L \cdot p \cdot (\sup\limits_{i,j} w_{ij}^*)^2 \nonumber \\
&\leq & p \cdot  \mathbb{E}_{\Sigma}(\widehat{\mathcal{D}}_n)   O(\ds\sqrt{T}) +o(p). \label{T_{1,2,1,1}}
\end{eqnarray}
For the second term in (\ref{T_{1,2,1}}), where $|j-j'| \geq T $, we use the following bound:
\[
\begin{array}{lcl}
 \ds\sum_{\substack{i=1 \\ i \neq j,j'}}^p w^{*}_{ij}  w^{*}_{ij'}
& \leq &  \ds\sum_{\substack{i=1 \\ i \neq j,j'}}^p (w^{*}_{ij}) ^2 \leq \frac 12,
 \end{array}
\]
then we prove that,
\begin{equation}
\label{T_{1,2,1,2}}
T_{1,2,1,2} := \underset{ |j- j'|  \geq T}{\ds\sum_{\substack{j=1 \\ j \neq i}}^p\sum_{\substack{j'=1 \\ j' \neq i}}^p}
\ds\frac{|j- j'|^{2 \alpha}}{T^{2 \alpha}} \sigma_{jj'}^2 \ds\sum_{i=1}^p w^{*}_{ij} w^{*}_{ij'} \leq \ds\frac{L \cdot p}{T^{2 \alpha}} = O ( \ds\frac{ p}{2 T^{2 \alpha}} )= o(p) .
\end{equation}
 Note that $\sup\limits_{i,j} \sigma_{ij} \leq 1$.
 The second term of $T_{1,2}$, is bounded as follows:
 \begin{eqnarray}
T_{1,2,2} &:=&  \ds\sum_{i=1}^p \underset{ j \neq j'}{\ds\sum_{\substack{j=1 \\ j \neq i}}^p\sum_{\substack{j'=1 \\ j' \neq i}}^p}w^{*}_{ij} w^{*}_{ij'} \sigma_{ij}^2\sigma_{ij'}^2 = \ds\sum_{i=1}^p\Big( \sum_{\substack{j=1 \\ j \neq i}}^p w^{*}_{ij} \sigma_{ij}^2 \Big)\Big( \sum_{\substack{j'=1 \\ j' \neq i}}^p  w^{*}_{ij'}\sigma_{ij'}^2 \Big) \nonumber \\
 &\leq  &  (\sup\limits_{i,j}w^{*}_{ij}) \sup\limits_{i} \Big(\ds\sum_{\substack{j=1 \\ 1 \leq |j-i | < T}}^p  \sigma_{ij}^2\Big) \Big(  \ds\sum_{i=1}^p\sum_{\substack{j'=1 \\ j' \neq i}}^p w^{*}_{ij'} \sigma_{ij'}^2 \Big) \nonumber \\
 &\leq & 2 L \cdot (\sup\limits_{i,j}w^{*}_{ij})\cdot T \cdot p \cdot  \mathbb{E}_{\Sigma}(\widehat{\mathcal{D}}_n) \leq p \cdot  \mathbb{E}_{\Sigma}(\widehat{\mathcal{D}}_n) \cdot O( \sqrt{T}) \label{T_{1,2,2}}
\end{eqnarray}
As a consequence of (\ref{T_{1,2,1,1}}) to (\ref{T_{1,2,2}}),
\begin{equation}
\label{T_{1,2}}
T_{1,2} \leq   p \cdot  \mathbb{E}_{\Sigma}(\widehat{\mathcal{D}}_n)  \cdot  O( \ds\sqrt{T}) + o(p)
\end{equation}
The last case, where $(i,j,i',j')$ vary in $A_3$, the indices are pairwise distinct,
\[
\begin{array}{lcl}
T_{1,3} &=& \underset{(i,j,i',j')\in A_3}{\ds\sum } w^*_{ij} w^*_{i'j'}
  ( \sigma_{ii'}\sigma_{jj'} + \sigma_{ij'}\sigma_{i'j})^2 \\
  &\leq& 2 \underset{(i,j,i',j')\in A_3}{\ds\sum  } w^*_{ij} w^*_{i'j'}
   \sigma^{2}_{ii'}\sigma^{2}_{jj'} + 2\underset{(i,j,i',j')\in A_3}{\ds\sum } w^*_{ij} w^*_{i'j'}\sigma^{2}_{ij'}\sigma^{2}_{i'j}
   \end{array}
 \]
As the two previous terms  have the same upper bound, let us deal with the first one say $ T_{1,3,1}$.
We should distinguish two cases, the first when  $|i-i'| < T$ and the second when $|i-i'|  \geq T$. We begin by the first case, which in turn will be decomposed into three terms. First,
\begin{eqnarray}
\label{T_{1,3,1,1}}
T_{1,3,1,1} &:=& \underset{\underset{|i-j| \geq |i -i' |, |i'-j'| \geq |i -i' |  }{(i,j,i',j')\in A_3}}{\ds\sum } w^*_{ij} w^*_{i'j'}
   \sigma^2_{ii'} \sigma^2_{jj'} \leq \underset{\underset{|i-j| \geq |i -i' |, |i'-j'| \geq |i -i' |  }{(i,j,i',j')\in A_3}}{\ds\sum  } w^{*2}_{ii'}  \sigma^2_{ii'} \sigma^2_{jj'} \nonumber \\
   &\leq & (\sup\limits_{ij} w^*_{ij})  \underset{1 \leq i,i' \leq p}{\sum } w^*_{ii'} \sigma_{ii'}^2 \underset{1 < |i-j|, |i' - j'| < T}{\underset{1 \leq j, j' \leq p}{\sum }} \sigma_{jj'}^2
   \leq  (\sup\limits_{ij} w^*_{ij}) \cdot T^2 \cdot p \cdot \mathbb{E}_{\Sigma}(\widehat{\mathcal{D}}_n)
\end{eqnarray}
Then,
 \begin{eqnarray}
 T_{1,3,1,2} &:=& \underset{\underset{|i-j| < |i -i' |< T, |i'-j'| \geq |i -i' |  }{(i,j,i',j')\in A_3}}{\ds\sum  } w^*_{ij} w^*_{i'j'}
   \sigma^2_{ii'} \sigma^2_{jj'} \leq  \underset{\underset{|i-j| < |i -i' |< T, |i'-j'| \geq |j -j' |  }{(i,j,i',j')\in A_3}}{\ds\sum  } w^*_{ij} w^*_{ii'} 
   \sigma^2_{ii'} \sigma^2_{jj'}  \nonumber\\
   & \leq &(\sup\limits_{ij} w^*_{ij}) \cdot T^2 \cdot p \cdot \mathbb{E}_{\Sigma}(\widehat{\mathcal{D}}_n) \leq k_2(\alpha ,L) \cdot T \ds\sqrt{T} \cdot p \cdot \mathbb{E}_{\Sigma}(\widehat{\mathcal{D}}_n)
\end{eqnarray}
 Finally, using Cauchy-Schwarz inequality, we have,
 \begin{eqnarray}
T_{1,3,1,3} &:=&\underset{\underset{|i-j| < |i -i' |< T, |i'-j'| < |i -i' | <T }{(i,j,i',j')\in A_3}}{\ds\sum   }  w^*_{ij} w^*_{i'j'}
   \sigma^2_{ii'} \sigma^2_{jj'} \nonumber \\
   & = & \underset{\underset{|i-j| < |i -i' |< T, |i'-j'| < |j -j' | < T }{(i,j,i',j')\in A_3}}{\ds\sum  }  w^*_{ij} w^*_{i'j'} \cdot  \ds\frac{ |i -i' |^{2 \alpha}}{|i-j|^{\alpha} |i'-j'|^{\alpha}}  \cdot
   \sigma^2_{ii'} \sigma^2_{jj'} \nonumber \\
   &\leq & ( \sup\limits_{i,j} w_{ij}^*)^2  \ds\sum_{i=1}^p \sum_{\substack{i'=1 \\ i' \neq i'}}^p |i -i' |^{2 \alpha} \sigma^2_{ii'} \underset{1 \leq |i-j| , |i' -j'| < T}{\underset{1 \leq~ j,j'~ \leq p}{ \sum } } \ds\frac{\sigma_{jj'}^2}{|i-j|^{\alpha} |i'-j'|^{\alpha}}  \nonumber \\
   & \leq & k_3(\alpha ,L)   \cdot T^{-1} \cdot 2 pL \cdot \max\{1, T^{-2\alpha+2}\}
   = o(p) \quad \text{ for } \alpha > \frac{1}{2}.   \label{T_{1,3,1,3}}
  \end{eqnarray}
 Now we suppose that we have $|i-i'| > T$, then,
\begin{eqnarray}
T_{1,3,2} &:=& \underset{\underset{ |i-i'| > T }{(i,j,i',j')\in A_3}}{\ds\sum  }w^*_{ij} w^*_{i'j'}  \sigma^{2}_{ii'}\sigma^{2}_{jj'} =\underset{\underset{ |i-i'| > T }{(i,j,i',j')\in A_3}}{\ds\sum  } w^*_{ij} w^*_{i'j'} \ds\frac{ |i-i'|^{2 \alpha} }{T^{2 \alpha}} \sigma^{2}_{ii'}\sigma^{2}_{jj'}  \nonumber \\
& \leq & \ds\frac{(\sup\limits_{i,j} w^*_{ij})^2}{T^{2 \alpha}}     \underset{1 \leq i,i' \leq p}{\ds\sum } |i-i'|^{2 \alpha}\sigma^{2}_{ii'} \underset{1 \leq |i-j|, |i'-j'| <T}{\underset{1 \leq j,j' \leq p}{\ds\sum  }} \sigma^{2}_{jj'} \nonumber \\
&\leq &\ds\frac{(\sup\limits_{i,j} w^*_{ij})^2 }{T^{2 \alpha }}  \cdot   2pL \cdot T^2  \leq \frac{k_4(\alpha ,L) \cdot p}{T^{2\alpha -1}}  = o(p) \quad \text{ for } \alpha > \frac{1}{2}. \label{T_{1,3,2}}
\end{eqnarray}
Finally we obtain, from (\ref{T_{1,3,1,1}}) to (\ref{T_{1,3,2}}) :
\begin{equation}
\label{T_{1,3}}
T_{1,3} \leq p \cdot \mathbb{E}_{\Sigma}(\widehat{\mathcal{D}}_n) \cdot  O(T \ds\sqrt{T})  +  o(p).
\end{equation}
Put together (\ref{T_{1,1}}), (\ref{T_{1,2}}) and (\ref{T_{1,3}}) to obtain (\ref{T_1}).
Let us give an upper bound for the second term of (\ref{var2termes}),
\[
\begin{array}{lcl}
T_2 &=& 4\mathbb{E}_{\Sigma} \{ \underset{i<j}{\ds\sum_{i=1}^p \ds\sum_{j=1}^p} w^*_{ij}
 (X_{k,i}X_{k,j} - \sigma_{ij})\sigma_{ij}\}^2 \\
  &=&  \underset{i \neq j}{\ds\sum_{i=1}^p \ds\sum_{j=1}^p}\underset{i' \neq j'}{\ds\sum_{i'=1}^p \ds\sum_{j'=1}^p} w^*_{ij} w^*_{i'j'}\sigma_{ij} \sigma_{i'j'} \mathbb{E}_{\Sigma} (X_{1,i}X_{1,j} - \sigma_{ij}) (X_{1,i'}X_{1,j'} - \sigma_{i'j'}) \\
  &=& \underset{i \neq j}{\ds\sum_{i=1}^p \ds\sum_{j=1}^p}\underset{i' \neq j'}{\ds\sum_{i'=1}^p \ds\sum_{j'=1}^p}  w^*_{ij} w^*_{i'j'}\sigma_{ij} \sigma_{i'j'} (  \sigma^*_{ii'}\sigma^*_{jj'} + \sigma^*_{ij'}\sigma^*_{i'j})
\end{array}
\]
Proceeding similarly, we shall distinguish three kind of terms. Let us begin by the case when the indices belong to $A_1$,
\begin{eqnarray}
T_{2,1} &=&  2\underset{i \neq j}{\ds\sum_{i=1}^p \ds\sum_{j=1}^p} w^{*2}_{ij} \sigma_{ij}^2 \mathbb{E}_{\Sigma}  [(X_{1,i}X_{1,j} - \sigma_{ij})^2] = 2 \underset{i \neq j}{\ds\sum_{i=1}^p \ds\sum_{j=1}^p} w^{*2}_{ij} \sigma_{ij}^2 (1+ \sigma_{ij}^2) \nonumber \\
& \leq & 4 (\sup\limits_{i,j} w^*_{ij} ) \underset{i \neq j}{\ds\sum_{i=1}^p \ds\sum_{j=1}^p} w^{*}_{ij} \sigma_{ij}^2
= 8 (\sup\limits_{i,j} w^*_{ij} ) \cdot p \cdot \mathbb{E}_{\Sigma}(\widehat{\mathcal{D}}_n) =  o(1) \cdot p \cdot \mathbb{E}_{\Sigma}(\widehat{\mathcal{D}}_n). \label{T_{2,1}}
\end{eqnarray}
Next, when  $(i,j,i',j') \in A_2$,
\[
\begin{array}{lcl}
 T_{2,2} & =& 4 \ds\sum_{i=1}^p \sum_{\substack{j=1 \\ j \neq i}}^p\sum_{\substack{j'=1 \\ j' \neq i}}^p w^{*}_{ij} w^{*}_{ij'} \sigma_{ij}\sigma_{ij'}( \sigma_{jj'} + \sigma_{ij}\sigma_{ij'}) \\
&=& 4  \ds\sum_{i=1}^p \sum_{\substack{j=1 \\ j \neq i}}^p\sum_{\substack{j'=1 \\ j' \neq i}}^p w^{*}_{ij} w^{*}_{ij'} \sigma_{ij}\sigma_{ij'}\sigma_{jj'}  + 4 \ds\sum_{i=1}^p \sum_{\substack{j=1 \\ j \neq i}}^p\sum_{\substack{j'=1 \\ j' \neq i}}^p w^{*}_{ij} w^{*}_{ij'} \sigma_{ij}^2\sigma_{ij'}^2 \\
\end{array}
\]
We bound from each term of $T_{2,2}$ separately. Using Cauchy-Schwarz inequality two times
we obtain,
\begin{eqnarray}
T_{2,2,1} &:=& \ds\sum_{i=1}^p \sum_{\substack{j=1 \\ j \neq i}}^p\sum_{\substack{j'=1 \\ j' \neq i}}^p w^{*}_{ij} w^{*}_{ij'} \sigma_{ij}\sigma_{ij'}\sigma_{jj'} \leq \ds\sum_{i=1}^p \sum_{\substack{j=1 \\ j \neq i}}^p w^{*}_{ij}\sigma_{ij} \Big( \sum_{\substack{j'=1 \\ j' \neq i}}^pw^{*2}_{ij'} \sigma_{ij'}^2 \Big)^{1/2} \Big( \sum_{\substack{j'=1 \\ j' \neq i}}^p \sigma_{jj'}^2 \Big)^{1/2} \nonumber \\
 & \leq & \Big(\ds\sum_{i=1}^p \sum_{\substack{j=1 \\ j \neq i}}^p w^{*2}_{ij}\sigma^2_{ij} \Big)^{1/2}
  \Big(\ds\sum_{i=1}^p \sum_{\substack{j=1 \\ j \neq i}}^p ( \sum_{\substack{j'=1 \\ j' \neq i}}^pw^{*2}_{ij'} \sigma_{ij'}^2 ) ( \sum_{\substack{j'=1 \\ j' \neq i}}^p \sigma_{jj'}^2 ) \Big)^{1/2}   \nonumber \\
& \leq &  (\sup\limits_{i,j} w^{*}_{ij}) \cdot p \cdot  \mathbb{E}_{\Sigma}(\widehat{\mathcal{D}}_n)   \cdot  O(T) = O(\ds\sqrt{T}) \cdot p \cdot  \mathbb{E}_{\Sigma}(\widehat{\mathcal{D}}_n) .  \nonumber
\end{eqnarray}
The second term in $T_{2,2}$ is $T_{1,2,2}$ and therefore,
\begin{equation}
\label{T_{2,2}}
T_{2,2} = O(\ds\sqrt{T}) \cdot p \cdot  \mathbb{E}_{\Sigma}(\widehat{\mathcal{D}}_n)  .
\end{equation}

 Finally, when $(i,j,i',j') \in A_3$, we have to bound from above
 \[
 T_{2,3}= \underset{(i,j,i',j') \in A_3}{ \sum  } w^*_{ij} w^*_{i'j'}\sigma_{ij} \sigma_{i'j'}  \sigma^*_{ii'}\sigma^*_{jj'} + \underset{(i,j,i',j') \in A_3}{ \sum  \sum  } w^*_{ij} w^*_{i'j'}\sigma_{ij} \sigma_{i'j'} \sigma^*_{ij'}\sigma^*_{i'j}.
 \]
 These last two terms, in $T_{2,3}$, are treated similarly, so let us deal with :
 \[
 \begin{array}{lcl}
& & \underset{(i,j,i',j') \in A_3}{ \ds\sum } w^*_{ij} w^*_{i'j'}\sigma_{ij} \sigma_{i'j'}  \sigma^*_{ii'}\sigma^*_{jj'} \\
  & \leq & \ds\sum_j \ds\sum_{i'} \Big( \sum_i w^*_{ij} \sigma^2_{ii'} \Big)^{1/2}  \Big( \sum_i w^*_{ij}\sigma^2_{ij} \Big)^{1/2}
 \Big( \sum_{j'} w^{*}_{i'j'}\sigma_{i'j'}^2 \Big)^{1/2}  \Big( \sum_{j'} w^{*}_{i'j'} \sigma_{jj'}^2 \Big)^{1/2} \\
 & \leq &  \Big( \ds\sum_j \ds\sum_{i'} (\sum_i w^*_{ij}\sigma^2_{ij})    (\sum_{j'} w^{*}_{i'j'}\sigma_{i'j'}^2)\Big)^{1/2}  \Big( \underset{(i,j,i',j') \in A_3}{ \ds\sum } w^*_{ij}  w^{*}_{i'j'} \sigma_{jj'}^2 \sigma^2_{ii'}  \Big)^{1/2} \\
 & \leq &  p \cdot \mathbb{E}_{\Sigma}(\widehat{\mathcal{D}}_n) \cdot \Big(\underset{(i,j,i',j') \in A_3}{ \ds\sum } w^*_{ij}  w^{*}_{i'j'} \sigma_{jj'}^2 \sigma^2_{ii'}  \Big)^{1/2}
 \end{array}
 \]
 Using the upper bound of $T_{1,3}$ obtained previously, we have

\begin{equation}
\label{T_{2,3}}
  T_{2,3} \leq  p \ds\sqrt{p} \cdot \Big( \mathbb{E}^{3/2}_{\Sigma}(\widehat{\mathcal{D}}_n) \cdot O( T^{3/4}) +  \mathbb{E}_{\Sigma}(\widehat{\mathcal{D}}_n)\cdot o(1) \Big)
 \end{equation}
 Put together (\ref{T_{2,1}}), (\ref{T_{2,2}}) and (\ref{T_{2,3}}) to get (\ref{T_2}).

 The asymptotic normality under the null hypothesis is obvious.
\hfill \end{proof}


\begin{proof}[Proof of Proposition~\ref{AN}]
 We use the decomposition (\ref{decomp}) in the proof of the Proposition~\ref{prop:espvar} and we treat each term separately.
Recall that, by our assumptions, $n\sqrt{p}\cdot \mathbb{E}_{\Sigma} (\widehat{\mathcal{D}}_n)=O(1) $. Use (\ref{T_2}) to get
\begin{eqnarray} \label{convenproba}
&& \Var_{\Sigma} \Big( \ds\frac{2}{\ds\sqrt{p}} \sum_{l=1}^n \underset{1 \leq i <j \leq p}{ \ds\sum } w^*_{ij} (X_{l,i}X_{l,j} - \sigma_{ij})\sigma_{ij} \Big) \nonumber \\
 &\leq & \ds\frac{n}{p} \Big( p^{3/2} \Big(o(1) \cdot \mathbb{E}_{\Sigma}(\widehat{\mathcal{D}}_n)  +   O( T^{3/4}) \mathbb{E}^{3/2}_{\Sigma}(\widehat{\mathcal{D}}_n)  \Big) + p  \cdot  \mathbb{E}_{\Sigma}(\widehat{\mathcal{D}}_n)  O(\ds\sqrt{T}) \Big) \nonumber \\
 & = & o(1)n \ds\sqrt{p} \cdot \mathbb{E}_{\Sigma}(\widehat{\mathcal{D}}_n)  +(n \ds\sqrt{p} \cdot \mathbb{E}_{\Sigma}(\widehat{\mathcal{D}}_n))^{3/2}
 \cdot \ds\frac{O(T^{3/4})}{n^{1/2}p^{1/4}} + n \ds\sqrt{p} \cdot \mathbb{E}_{\Sigma}(\widehat{\mathcal{D}}_n) \cdot o(1) 
\end{eqnarray}
This tends to 0, since $T^3/n^2p  = (n^2 p b^2(\varphi))^{-1} \cdot \varphi^{4 -2/\alpha  }=o(1)$, which is true for all $\alpha >1/2$.

It follows that, for proving the asymptotic normality, it is sufficient to prove the asymptotic normality of
$$
n \sqrt{p}\cdot \frac 1{n(n-1)p} \underset{1 \leq k \ne l \leq n}{ \ds\sum } ~\underset{1 \leq i <j \leq p}{ \ds\sum } w^*_{ij}(X_{k,i}X_{k,j} - \sigma_{ij})(X_{l,i}X_{l,j} - \sigma_{ij}).
$$
We study $V_n$ centered, 1-degenerate  U-statistic, with symmetric kernel $H_n(X_1, X_2)$ defined as follows
\begin{eqnarray*}
V_n &= &  \underset{1 \leq k \ne l \leq n}{ \ds\sum } H_n(X_k,X_l), \\
H_n(X_1,X_2) &= &\ds\frac{1}{n \sqrt{p}} \underset{1 \leq i <j \leq p}{ \ds\sum } w^*_{ij}  (X_{k,i}X_{k,j} - \sigma_{ij})(X_{l,i}X_{l,j} - \sigma_{ij}).
\end{eqnarray*}
We apply Theorem 1 of \cite{Hall84}. Therefore we check that
 $ \mathbb{E}_{\Sigma}(H^2_n(X_1,X_2)) < + \infty $ and that
 \[
 \ds\frac{\mathbb{E}_{\Sigma}(G^2_n(X_1, X_2)) + n^{-1} \mathbb{E}_{\Sigma}(H_n^4(X_1, X_2))}{\mathbb{E}_{\Sigma}^2(H_n^2(X_1, X_2))} \longrightarrow 0 ,
 \]
where $ G_n(x,y):= \mathbb{E} _{\Sigma}(H_n(X_1,x)H_n(X_1,y)) $, for $x,y \in \mathbb{R}^p$. We compute
\[
\begin{array}{lcl}
G_n(x,y) 
&=& \ds\frac{1}{n^2  p} \underset{1 \leq i <j \leq p}{ \ds\sum } ~ \underset{1 \leq i' <j' \leq p}{ \ds\sum }   w^*_{ij} w^*_{i'j'}(x_ix_j - \sigma_{ij})(y_i y_j - \sigma_{ij})   (\sigma_{ii'} \sigma_{jj'} + \sigma_{i'j} \sigma_{ij'}).
\end{array}
\]
Since   $  n \ds\sqrt{p} \cdot \mathbb{E}_{\Sigma}(\widehat{\mathcal{D}}_n) = O(1)$,  and from the inequality (\ref{T_1}), we have
$$
 \mathbb{E}_{\Sigma}(H_n^2(X_1, X_2)) = \ds\frac{1}{2 n^2} (1 + o(1)) \,\, .
 $$
In order to prove that $ \mathbb{E}_{\Sigma}(G^2_n(X_1, X_2)) / \mathbb{E}_{\Sigma}^2(H_n^2(X_1, X_2)) = o(1) $, it is sufficient to show that $  \mathbb{E}_{\Sigma} \Big(  \underset{1 \leq i <j \leq p}{ \ds\sum } ~ \underset{1 \leq i' <j' \leq p}{ \ds\sum }  w^*_{ij} w^*_{i'j'} (X_{1,i} X_{1,j} - \sigma_{ij})(X_{2,i'} X_{2,j'} - \sigma_{i'j'})  (\sigma_{ii'} \sigma_{jj'} + \sigma_{i'j} \sigma_{ij'})  \Big)^2 =o(p^2).$ In fact,
\begin{eqnarray}
&& \mathbb{E}_{\Sigma} \Big(  \underset{1 \leq i <j \leq p}{ \ds\sum } ~ \underset{1 \leq i' <j' \leq p}{ \ds\sum }  w^*_{ij} w^*_{i'j'} (X_{1,i} X_{1,j} - \sigma_{ij})(X_{2,i'} X_{2,j'} - \sigma_{i'j'})  (\sigma_{ii'} \sigma_{jj'} + \sigma_{i'j} \sigma_{ij'})  \Big)^2 \nonumber \\[0.5 cm]
 &=& \underset{1 \leq i_1<j_1 \leq p}{ \ds\sum } ~ \underset{1 \leq i_1' <j_1' \leq p}{ \ds\sum }  ~ \underset{1 \leq i_2 <j_2 \leq p}{ \ds\sum } ~ \underset{1 \leq i'_2 <j'_2 \leq p}{ \ds\sum }   w^{*}_{i_1j_1} w^{*}_{i'_1j'_1} w^{*}_{i_2 j_2} w^{*}_{i_2'j_2'} (\sigma_{i_1 i'_1} \sigma_{j_1 j'_1} + \sigma_{i'_1 j_1} \sigma_{i_1j'_1})  (\sigma_{i_2 i_2'} \sigma_{j_2 j_2'}
\nonumber \\
& & + \sigma_{i_2'j_2} \sigma_{i_2 j_2'})  \cdot \mathbb{E}[(X_{1,i_1} X_{1,j_1} - \sigma_{i_1j_1})(X_{1,i_2} X_{1,j_2} - \sigma_{i_2j_2})]  \mathbb{E}[(X_{2,i'_1} X_{2,j'_1} - \sigma_{i'_1j'_1})(X_{2,i_2'} X_{2,j_2'} - \sigma_{i_2'j_2'})] \nonumber \\[0.5 cm]
&= & \underset{1 \leq i_1 <j_1 \leq p}{ \ds\sum } ~ \underset{1 \leq i'_1 <j'_1 \leq p}{ \ds\sum }  ~ \underset{1 \leq i_2 <j_2 \leq p}{ \ds\sum } ~ \underset{1 \leq i'_2 <j'_2 \leq p}{ \ds\sum }   w^{*}_{i_1 j_1} w^{*}_{i'_1j'_1} w^{*}_{i_2 j_2} w^{*}_{i_2'j_2'} (\sigma_{i_1i'_1} \sigma_{j_1j'_1} + \sigma_{i'_1j_1} \sigma_{i_1j'_1}) \nonumber \\
& &  \cdot  (\sigma_{i_2 i_2'} \sigma_{j_2 j_2'} + \sigma_{i_2'j_2} \sigma_{i_2 j_2'})
 (\sigma_{i_1i_2} \sigma_{j_2j_1} + \sigma_{i_1j_2} \sigma_{i_2j_1})
(\sigma_{i'_1i'_2} \sigma_{j_2'j'_1} + \sigma_{i'_1j_2'} \sigma_{i_2'j'_1})  \label{EG^2}
\end{eqnarray}
To bound from above (\ref{EG^2}), we shall distinguish four cases.
The first one is when all couples of indices are equal,
 \[
 \begin{array}{lcl}
 \mathcal{G}_1 &:=& \underset{1 \leq i_1 <j_1 \leq p}{ \ds\sum }  w^{*4}_{i_1j_1} ( 1 + \sigma_{i_1j_1}^2)^4 \leq ( \sup\limits_{i_1,j_1}w_{i_1j_1}^{*^2}) \cdot  (\sup\limits_{i_1,j_1}(1 + \sigma_{i_1j_1}^2)^4) \cdot \underset{1 \leq i_1 <j_1 \leq p}{ \ds\sum }  w^{*2}_{i_1j_1} \\
 & \leq & 8 \cdot ( \sup\limits_{i_1,j_1}w_{i_1j_1}^{*^2}) \cdot p = o(p) =o(p^2).
 \end{array}
\]
The second one is when we have two different  pairs of couples of indices, which can be obtained by two different combinations of the couples of indices. When we have equal pairs of couples of indices, as for example  $(i_1,j_1)=(i_2, j_2)$,  $(i'_1,j'_1) = (i_2', j_2')$ and $ (i_1,j_1)\neq (i'_1,j'_1)$, we get
\[
\begin{array}{lcl}
\mathcal{G}_{2,1} &:=& \underset{1 \leq i_1 <j_1 \leq p}{ \ds\sum } ~ \underset{1 \leq i'_1 <j'_1 \leq p}{ \ds\sum }  w^{*2}_{i_1j_1} w^{*2}_{i'_1j'_1} (\sigma_{i_1i'_1} \sigma_{j_1j'_1} + \sigma_{i'_1j_1} \sigma_{i_1j'_1})^2  ( 1 + \sigma_{i_1j_1}^2)(1 + \sigma_{i'_1j'_1}^2) \\[0.5cm]
& \leq &  ( \sup\limits_{i_1,j_1}w_{i_1j_1}^{*^2}) \cdot (\sup\limits_{i_1,j_1}(1 + \sigma_{i_1j_1}^2)^2) \cdot \underset{1 \leq i_1 <j_1 \leq p}{ \ds\sum } ~ \underset{1 \leq i'_1 <j'_1 \leq p}{ \ds\sum }  w^{*}_{i_1j_1} w^{*}_{i'_1j'_1} (\sigma_{i_1i'_1} \sigma_{j_1j'_1} + \sigma_{i'_1j_1} \sigma_{i_1j'_1})^2 \\[0.5cm]
& \leq & 4 \cdot  ( \sup\limits_{i_1,j_1}w_{i_1j_1}^{*^2})  \cdot n^2  p \cdot  \mathbb{E}_{\Sigma}(H_n^2(X_1, X_2)) = 4 \cdot  ( \sup\limits_{i_1,j_1}w_{i_1j_1}^{*^2})  \cdot p = o(p^2).
\end{array}
\]
When we have three couples of indices equal, for example $(i_1,j_1)=(i_2,j_2)=(i'_2,j'_2)$ and $ (i_1,j_1) \neq (i'_1,j'_1)$, we get
\[
\begin{array}{lcl}
\mathcal{G}_{2,2} &:=& \underset{1 \leq i_1 <j_1 \leq p}{ \ds\sum } ~ \underset{1 \leq i'_1 <j'_1 \leq p}{ \ds\sum }  w^{*3}_{i_1j_1} w^{*}_{i'_1j'_1}  (\sigma_{i_1i'_1} \sigma_{j_1j'_1} + \sigma_{i'_1j_1} \sigma_{i_1j'_1})^2  ( 1 + \sigma_{i_1j_1}^2)(1 + \sigma_{i'_1j'_1}^2) \\
& \leq &  4 \cdot  ( \sup\limits_{i_1,j_1}w_{i_1j_1}^{*^2})  \cdot n^2  p \cdot  \mathbb{E}_{\Sigma}(H_n^2(X_1, X_2)) = o(p^2).
\end{array}
\]
For the third case, there are  three different couples of pairs of indices, for example, $(i_1,j_1)=(i'_2,j'_2)$ and $ (i_1,j_1) \neq (i'_1,j'_1) \neq (i_2,j_2)$. Using Cauchy-Schwarz inequality several times we obtain,
\[
\begin{array}{lcl}
\mathcal{G}_3 &:=&  \underset{1 \leq i_1 <j_1 \leq p}{ \ds\sum } ~ \underset{1 \leq i'_1 <j_1' \leq p}{ \ds\sum }  ~ \underset{1 \leq i_2 <j_2 \leq p}{ \ds\sum }  w^{*}_{i_1j_1} w^{*}_{i_1'j'_1} w^{*2}_{i_2 j_2}(\sigma_{i_1i'_1} \sigma_{j_1j'_1} + \sigma_{i_1'j_1} \sigma_{i_1j'_1}) \\
&&  \hspace{3cm} \cdot (\sigma_{i_1i_2} \sigma_{j_2j_1} + \sigma_{i_1j_2} \sigma_{i_2j_1}) (\sigma_{i_1'i_2} \sigma_{j_2j'_1} + \sigma_{i'_1j_2} \sigma_{i_2j'_1})( 1 + \sigma_{i_2,j_2}^2) \\[0.5cm]
& \leq & \underset{1 \leq i'_1 <j'_1 \leq p}{ \ds\sum }  ~ \underset{1 \leq i_2 <j_2 \leq p}{ \ds\sum } w^{*}_{i'_1j'_1} w^{*2}_{i_2 j_2} (\sigma_{i'_1i_2} \sigma_{j_2j'_1} + \sigma_{i'_1j_2} \sigma_{i_2j'_1})( 1 + \sigma_{i_2,j_2}^2) \\
 && \cdot \Big(  \underset{1 \leq i_1 <j_1 \leq p}{ \ds\sum }   w^{*}_{i_1j_1} (\sigma_{i_1i'_1} \sigma_{j_1j'_1} + \sigma_{i'_1j_1} \sigma_{i_1j'_1})^2 \Big)^{1/2}  \Big(  \underset{1 \leq i_1 <j_1 \leq p}{ \ds\sum }   w^{*}_{i_1j_1}  (\sigma_{i_1i_2} \sigma_{j_2j_1} + \sigma_{i_1j_2} \sigma_{i_2j_1})^2 \Big)^{1/2} \\[0.5cm]
& \leq &  \underset{1 \leq i_2 <j_2 \leq p}{ \ds\sum }  w^{*2}_{i_2 j_2} ( 1 + \sigma_{i_2,j_2}) ^2 \Big(\underset{1 \leq i'_1 <j'_1 \leq p}{ \ds\sum } w_{i'_1j'_1}^*  (\sigma_{i'_1i_2} \sigma_{j_2j'_1} + \sigma_{i'_1j_2} \sigma_{i_2j'_1})^2\Big)^{1/2} \\
&& \cdot \Big(\underset{1 \leq i'_1 <j'_1 \leq p}{ \ds\sum }  \underset{1 \leq i_1 <j_1 \leq p}{ \ds\sum } w_{i'_1j'_1}^*   w^{*}_{i_1j_1}  (\sigma_{i_1i_1'} \sigma_{j_1'j_1} + \sigma_{i_1j_1'} \sigma_{i_1'j_1})^2 \Big)^{1/2} \\
&& \cdot   \Big(  \underset{1 \leq i_1 <j_1 \leq p}{ \ds\sum }   w^{*}_{i_1j_1}  (\sigma_{i_1i_2} \sigma_{j_2j_1} + \sigma_{i_1j_2} \sigma_{i_2j_1})^2 \Big)^{1/2} .\\[0.5cm]
\end{array}
\]
Moreover, we recognize in these bounds
$$
\underset{ i'_1 <j'_1 }{ \ds\sum }  \underset{ i_1 <j_1 }{ \ds\sum } w_{i'_1j'_1}^*   w^{*}_{i_1j_1}  (\sigma_{i_1i_1'} \sigma_{j_1'j_1} + \sigma_{i_1j_1'} \sigma_{i_1'j_1})^2 = n^2 p \cdot \mathbb{E}_{\Sigma}(H_n^2(X_1, X_2))
$$
which is $O(p)$. Thus,
\[
\begin{array}{lcl}
\mathcal{G}_3 & \leq & \sup\limits_{i_2,j_2}( 1 + \sigma_{i_2 j_2})^2 \cdot \Big( \underset{1 \leq i_2 <j_2 \leq p}{ \ds\sum }  \underset{1 \leq i'_1 <j'_1 \leq p}{ \ds\sum }w^{*2}_{i_2 j_2} w_{i'_1j'_1}^*  (\sigma_{i'_1i_2} \sigma_{j_2j'_1} + \sigma_{i'_1j_2} \sigma_{i_2j'_1})^2 \Big)^{1/2} \\
&&  \cdot \Big(n^2 p \cdot  \mathbb{E}_{\Sigma}(H_n^2(X_1, X_2)) \cdot  \underset{1 \leq i_2 <j_2 \leq p}{ \ds\sum } \underset{1 \leq i_1 <j_1 \leq p}{ \ds\sum } w^{*2}_{i_2 j_2}  w^{*}_{i_1j_1}  (\sigma_{i_1i_2} \sigma_{j_2j_1} + \sigma_{i_1j_2} \sigma_{i_2j_1})^2 \Big)^{1/2} \\
& \leq & 2 ( \sup\limits_{i_1,j_1}w_{i_1j_1}^{*})  \cdot n^3 p^{3/2} \cdot  \mathbb{E}^{3/2}_{\Sigma}(H_n^2(X_1, X_2)) \leq               (\sup\limits_{i_1,j_1}w_{i_1j_1}^{*}) \cdot p^{3/2} = o(p^{3/2})=o(p^2).
\end{array}
\]
 Now we will treat the last case, when the pairs of indices are pairwise distinct, in this case, we have 16 terms to handle. As all terms are treated the same way, let us deal with:
\begin{eqnarray*}
\mathcal{G}_4 &:=& \underset{1 \leq i_1 <j_1 \leq p}{ \ds\sum } ~ \underset{1 \leq i'_1 <j'_1 \leq p}{ \ds\sum }  ~ \underset{1 \leq i_2 <j_2 \leq p}{ \ds\sum } ~ \underset{1 \leq i'_2 <j'_2 \leq p}{ \ds\sum }   w^*_{i_1j_1} w^*_{i'_1j'_1} w^{*}_{i_2 j_2} w^{*}_{i_2'j_2'} \\
&& \cdot \sigma_{i_1i'_1}\sigma_{i_2 i'_2} \sigma_{j_2 j_2'} \sigma_{j_1j'_1}\sigma_{i_1i_2} \sigma_{j_2j_1}\sigma_{i'_1i'_2} \sigma_{j_2'j'_1}
\end{eqnarray*}
In order to find an upper bound for $\mathcal{G}_4 $, we decompose the previous sums, into several sums, similarly to   the  upper bound of  ($\ref{T_{1,3}})$. That is $(i_1, j_1, i'_1,j'_1, i_2, j_2, i'_2,j'_2) \in J_1 \cup J_2 \cup \dots \cup J_{16} $, where $J_1, \dots, J_{16}$, form a partition of the set $  \{ (i_1, j_1, i'_1,j'_1, i_2, j_2, i'_2,j'_2) \in \{1, \dots,p\}^8 \} $. Let us define,
$$
J_1 :=  \{ (i_1, j_1, i'_1,j'_1, i_2, j_2, i'_2,j'_2) \in \{1, \dots,p\}^8; 1 <  |i_1-i'_1| , |i_1-i_2|, |i_2 - i'_2| , |i'_1-i'_2|< T) \},
$$
$$
J_2 :=  \{ (i_1, j_1, i'_1,j'_1, i_2, j_2, i'_2,j'_2) \in \{1, \dots,p\}^8;  1 <  |i_1-i'_1| , |i_1-i_2|, |i_2 - i'_2|< T, \text{ and } |i'_1-i'_2|> T) \},
$$
and so on, for all $J_r \,, r=3, \dots, 16$. To bound from above the sum over $J_1$, we partition again $J_1$, $J_1 = J_{1,1} \cup \dots \cup J_{1,16}$  such that,
$$
\begin{array}{lcl}
J_{1,1} &:=& \{ (i_1, j_1, i'_1,j'_1, i_2, j_2, i'_2,j'_2) \in \{1, \dots,p\}^8; |i_1-i'_1| \leq |i_1-j_1|,|i'_1-i'_2| \leq |i'_1 -j'_1|, \\
  && \hspace{1cm} |i_1 - i_2| \leq |i_2 -j_2| \text{ and } |i_2 - i'_2| \leq |i'_2 -j'_2|  \},
\end{array}
$$
and so on, until we get the partition of $J_1$.

\begin{eqnarray*}
\mathcal{G}_{4,1} &:=& \underset{(i_1, j_1, i'_1,j'_1, i_2, j_2, i'_2,j'_2) \in J_{1,1} }{ \underset{1 \leq i_1 <j_1 \leq p}{ \ds\sum } ~  \underset{1 \leq i'_1 <j'_1 \leq p}{ \ds\sum } {{\underset{1 \leq i_2 <j_2 \leq p}{ \ds\sum } ~ \underset{1 \leq i'_2 <j'_2 \leq p}{ \ds\sum }}}}  w^{*}_{i_1j_1} w^{*}_{i'_1j'_1} w^{*}_{i_2 j_2} w^{*}_{i_2'j_2'} \\
&& \cdot \sigma_{i_1i'_1} \sigma_{j_1j'_1}\sigma_{i_2 i_2'} \sigma_{j_2 j_2'} \sigma_{i_1i_2} \sigma_{j_2j_1}\sigma_{i'_1i'_2} \sigma_{j_2'j'_1}  \nonumber \\ \nonumber \\
& \leq &   \underset{1 \leq i_1 ,i'_1 \leq p}{ \ds\sum } ~  \underset{1 \leq i_2,i'_2 \leq p}{ \ds\sum } w^{*}_{i_1i'_1}w^{*}_{i_1i_2}w^{*}_{i'_1i'_2}w^{*}_{i_2i'_2} \sigma_{i_1i'_1}\sigma_{i_2 i_2'}\sigma_{i_1i_2}\sigma_{i'_1i'_2} \\
&& \cdot \hspace{-0.5 cm} \underset{1 < |i_1-j_1|, |i'_1 -j'_1|,|i_2 -j_2|, |i'_2 -j'_2| <T}{{{\underset{1 \leq j_1,j'_1  \leq p}{ \ds\sum } ~ \underset{1 \leq j_2 , j'_2 \leq p}{ \ds\sum }}}}  \hspace{-0.5 cm} \sigma_{j_2 j_2'} \sigma_{j_1j'_1} \sigma_{j_2j_1} \sigma_{j_2'j'_1}  \nonumber \\
& \leq &  T^4 \cdot (\sup\limits_{i_1,j_1} w^*_{i_1j_1})^2 \cdot  \underset{1 \leq i_1 ,i'_1 \leq p}{ \ds\sum } ~  \underset{1 \leq i_2,i'_2 \leq p}{ \ds\sum } \ds\sqrt{w^{*}_{i_1i'_1}w^{*}_{i_1i_2}w^{*}_{i'_1i'_2}w^{*}_{i_2i'_2}} \sigma_{i_1i'_1}\sigma_{i_2 i_2'}\sigma_{i_1i_2}\sigma_{i'_1i'_2}\nonumber \\
& \leq &  T^4 \cdot (\sup\limits_{i_1,j_1} w^*_{i_1j_1})^2 \Big(  \underset{1 \leq i_1 ,i'_1 \leq p}{ \ds\sum } ~  \underset{1 \leq i_2,i'_2 \leq p}{ \ds\sum } w^{*}_{i_1i'_1}w^{*}_{i_2i'_2}
 \sigma^2_{i_1i'_1}\sigma^2_{i_2 i_2'} \Big)^{1/2} \\
&& \cdot \Big(  \underset{1 \leq i_1 ,i_1' \leq p}{ \ds\sum } ~  \underset{1 \leq i_2,i'_2 \leq p}{ \ds\sum } w^{*}_{i_1i_2}w^{*}_{i'_1i'_2}
 \sigma^2_{i_1i_2}\sigma^2_{i'_1i'_2} \Big)^{1/2} \nonumber \\
 & \leq & T^4 \cdot (\sup\limits_{i_1,j_1} w^*_{i_1j_1})^2 \cdot p^2 \cdot \mathbb{E}^2_{\Sigma}(\widehat{D}_n) \nonumber
\end{eqnarray*}
Again, by our assumption that $n^2 p \cdot \mathbb{E}^2_{\Sigma}(\widehat{D}_n) = O(1)$, we can see that :
\[
\mathcal{G}_{4,1} \leq \kappa_0(\alpha, L) \cdot T^3  \cdot p^2 \cdot \mathbb{E}^2_{\Sigma}(\widehat{D}_n)   =  p^2 \cdot O( \ds\frac{T^3}{n^2 p} ) =p^2 \cdot o(1)
\]
where, from now on, $ \kappa_0(\alpha, L), \kappa_1(\alpha, L), \dots, $ denote constants that depend on $ \alpha $ and $L$.
Now, we define $J_{1,2} := \{(i_1, j_1, i'_1,j'_1, i_2, j_2, i'_2,j'_2) \in \{1,\dots,p \}^8, \text{ such that } |i-i'| \leq |i-j|$, $|i'-i'_1| \leq |i' -j'|$, $|i - i_1| \leq |i_1 -j_1|$ and $|i_1 - i'_1| > |i'_1 -j'_1|\}$, thus we have,
\begin{eqnarray}
\mathcal{G}_{4,2} &:=& \underset{(i_1, j_1, i'_1,j'_1, i_2, j_2, i'_2,j'_2) \in J_{1,2} }{\ds\sum  ~  \ds\sum  ~ \ds\sum  ~\ds\sum }  w^{*}_{i_1j_1} w^{*}_{i'_1j'_1} w^{*}_{i_2 j_2} w^{*}_{i_2'j_2'} \sigma_{i_1i'_1}\sigma_{i_2 i_2'} \sigma_{j_1 j_1'} \sigma_{j_2j'_2}\sigma_{i_1i_2} \sigma_{j_2j_1}\sigma_{i'_1i'_2} \sigma_{j_2'j'_1}  \nonumber \\ \nonumber \\
& \leq &  (\sup\limits_{i_1,j_1} w^*_{i_1j_1})^{5/2} \cdot \underset{1 \leq i_1 ,i'_1 \leq p}{ \ds\sum } ~  \underset{1 \leq i_2,i'_2 \leq p}{ \ds\sum } \ds\sqrt{ w^{*}_{i_1i'_1}w^{*}_{i_1i_2}w^{*}_{i'_1i'_2}} \cdot  |i_2 -i'_2|^{ \alpha} \cdot \sigma_{i_1i'_1}\sigma_{i_2 i_2'}\sigma_{i_1i_2}\sigma_{i'_1i'_2}
\nonumber \\
&& \hspace{2cm} \cdot  \underset{1 < |i_1-j_1|, |i'_1 -j'_1|,|i_2 -j_2|, |i'_2 -j'_2| <T}{{{\underset{1 \leq j_1,j'_1  \leq p}{ \ds\sum } ~ \underset{1 \leq j_2 , j'_2 \leq p}{ \ds\sum }}}} \ds\frac{1}{|i'_2 - j'_2|^{\alpha}} \cdot \sigma_{j_2 j_2'} \sigma_{j_1j'_1} \sigma_{j_2j_1} \sigma_{j_2'j'_1}  \nonumber \\
& \leq &  (\sup\limits_{i_1,j_1} w^*_{i_1j_1})^{5/2} \cdot \Big(  \underset{1 \leq i_1 ,i'_1 \leq p}{ \ds\sum } ~  \underset{1 \leq i_2,i'_2 \leq p}{ \ds\sum } w^{*}_{i_1i'_1} |i_2 -i'_2|^{ 2\alpha}
 \sigma^2_{i_1i'_1}\sigma^2_{i_2 i_2'} \Big)^{1/2} \nonumber \\
 && \hspace{2cm} \cdot \Big(  \underset{1 \leq i_1 ,i'_1 \leq p}{ \ds\sum } ~  \underset{1 \leq i_2,i'_2 \leq p}{ \ds\sum } w^{*}_{i_1i_2}w^{*}_{i'_1i'_2}
 \sigma^2_{i_1i_2}\sigma^2_{i'_1i'_2} \Big)^{1/2}  \cdot  T^3 \cdot \max \{ 1, T^{ - \alpha +1} \}  \nonumber \\
 & \leq & \ds\sqrt{2L} \cdot  (\sup\limits_{i_1,j_1} w^*_{i_1j_1})^{5/2} \cdot  T^3 \cdot \max \{ 1, T^{ - \alpha +1} \} \cdot p^2 \cdot \mathbb{E}^{3/2}_{\Sigma}(\widehat{D}_n) \nonumber
\end{eqnarray}
Therefore,
\begin{eqnarray}
\mathcal{G}_{4,2} & \leq & \kappa_1(\alpha,L)  \cdot  \max \{ T^{7/4}, T^{11/4 - \alpha } \}    \cdot \mathbb{E}^{3/2}_{\Sigma}(\widehat{D}_n) \nonumber \\
& \leq &  \kappa_1(\alpha,L) \cdot \max \{ T^{7/4}, T^{11/4 - \alpha } \}    \cdot O(\ds\frac{1}{n^{3/2} p^{3/4}}) \nonumber \\
& = & o(1) \quad \text{ since } T^3/n^2p \longrightarrow 0 
\end{eqnarray}
Using similar arguments, we can prove that all remaining terms tend to zero. In consequence,
\[
\ds\frac{\mathbb{E}_{\Sigma}(G^2_n(X_1, X_2)) }{\mathbb{E}_{\Sigma}^2(H_n^2(X_1, X_2))} \longrightarrow 0.
\]

Now let us prove that, $\mathbb{E}_\Sigma (H_n^4 (X_1,X_2))/\mathbb{E}_{\Sigma}^2(H_n^2(X_1, X_2)) =o(n) $,
\begin{eqnarray}
&& \mathbb{E}_\Sigma(H_n^4 (X_1,X_2)) = \ds\frac{1}{n^4 p^2}   \underset{i_1<j_1}{ \ds\sum }   \underset{i_2<j_2}{ \ds\sum }    \underset{i_3<j_3}{ \ds\sum }   \underset{i_4<j_4}{ \ds\sum }
w^{*}_{i_1j_1} w^{*}_{i_2j_2} w^{*}_{i_3j_3} w^{*}_{i_4j_4} \nonumber \\
&& \hspace{-0.4cm} \cdot \mathbb{E}_\Sigma^2[ (X_{1,i_1} X_{1,j_1} - \sigma_{i_1j_1}) (X_{1,i_2} X_{1,j_2} - \sigma_{i_2j_2}) (X_{1,i_3} X_{1,j_3} - \sigma_{i_3j_3}) (X_{1,i_4} X_{1,j_4} - \sigma_{i_4j_4})] \nonumber \\ \nonumber
\end{eqnarray}
The above squared expected value is a sum of a large number of terms that are all treated similarly. Let us consider examples of terms containing squared terms and products of terms, respectively.  For $\alpha > 1/2 $,
\begin{eqnarray}
\mathcal{H}_1 &:= & \underset{ i_1<j_1 }{ \ds\sum }   \underset{i_2<j_2}{ \ds\sum}    \underset{i_3<j_3}{ \ds\sum}   \underset{i_4<j_4}{ \ds\sum}
w^{*}_{i_1j_1} w^{*}_{i_2j_2} w^{*}_{i_3j_3} w^{*}_{i_4j_4} \sigma_{i_1 i_2}^2 \sigma_{j_1j_2}^2 \sigma_{i_3i_4}^2 \sigma_{j_3j_4}^2 \nonumber \\
& \leq  & 4 (\sup\limits_{i,j} w^*_{ij})^4 \ds\sum_{i_1 =1}^p \sum_{i_2=1}^p |i_1 - i_2|^{2 \alpha } \sigma^2_{i_1 i_2} \sum_{\substack{j_1 =1 \\ |i_1-j_1| < T }}^p \sup\limits_{j_2 }\sigma_{j_1j_2}^2 \sum_{\substack{j_2 =1 \\ |i_2 -j_2|< T}}^p \frac{1}{|j_1 -j_2|^{2 \alpha}} \nonumber \\
&& \cdot \ds\sum_{i_3 =1}^p \sum_{i_4=1}^p  |i_3 - i_4|^{2 \alpha} \sigma^2_{i_3 i_4} \sum_{\substack{j_3 =1 \\ |i_3-j_3|<T}}^p \sup\limits_{j_4} \sigma_{j_3 j_4}^2 \sum_{\substack{j_4 =1 \\ |i_4 -j_4|< T}}^p
\frac{1}{|j_3 - j_4|^{2 \alpha}} \nonumber \\
& \leq &  16 L^2 \cdot ( 2 \alpha - 1)^{-2} \cdot (\sup\limits_{i,j} w^*_{ij})^4 \cdot p^2 T^2  \leq \kappa_2( \alpha , L) \cdot p^2 \nonumber
\end{eqnarray}
The terms containing no squared values are treated as, e.g.,
\begin{eqnarray}
\mathcal{H}_2&:=& \underset{i_1<j_1}{ \ds\sum}   \underset{i_2<j_2}{ \ds \sum}    \underset{i_3<j_3}{ \ds\sum}   \underset{i_4<j_4}{ \ds\sum}
w^{*}_{i_1j_1} w^{*}_{i_2j_2} w^{*}_{i_3j_3} w^{*}_{i_4j_4} \sigma_{i_1 i_2} \sigma_{j_1j_2} \sigma_{i_3i_4} \sigma_{j_3j_4}\sigma_{i_1 i_3} \sigma_{j_1j_3} \sigma_{i_2i_4} \sigma_{j_2j_4} \nonumber
\end{eqnarray}
We can see that $\mathcal{H}_2 $ coincides with  $\mathcal{G}_{4,2}$.
Then we can deduce that ,
\[
\ds\frac{\mathbb{E}_{\Sigma}(H_n^4(X_1, X_2))}{\mathbb{E}_{\Sigma}^2(H_n^2(X_1, X_2))} = O(1) = o(n).
\]
Finally  we can apply \cite{Hall84}, and we obtain:
\begin{equation}
\label{convenloi}
V_n = \ds\frac{1}{n \ds\sqrt{p}} \underset{1 \leq k \neq l \leq n}{ \ds\sum } ~\underset{1 \leq i <j \leq p}{ \ds\sum } w^*_{ij}(X_{k,i}X_{k,j} - \sigma_{ij})(X_{l,i}X_{l,j} - \sigma_{ij}) \stackrel{\mathcal{L}}{\longrightarrow} N(0,1).
\end{equation}
Combining (\ref{convenproba}) and (\ref{convenloi}), we have by Slutsky theorem that:
\[
 n \sqrt{p} \cdot (\widehat{\mathcal{D}}_n - \mathbb{E}_{\Sigma}(\widehat{\mathcal{D}}_n)) \stackrel{\mathcal{L}}{\longrightarrow} N(0,1).
\]
\hfill \end{proof}

\begin{proof}[Proof of Proposition~\ref{cor1} ]
Let us check the case where $u_{ij} = 1$ for all $i,\,j$ such that $|i-j| \leq T$ and the generalization to all $U$ in $\mathcal{U}$ will be obvious.
Using Gershgorin's Theorem we get that each eigenvalue of $\Sigma^*_U
 = [ u_{ij} \sigma^*_{ij} ]_{1 \leq i,j \leq p} $ lies in one of the disks centered in $\sigma_{ii} = 1$ and radius $R_i = \ds\sum_{\substack{j=1 \\ j\neq i }}^p |u_{ij} \sigma^*_{ij}| = \ds\sum_{\substack{j=1 \\ j\neq i }}^p \sigma^*_{ij}$.
We have,
\[
\begin{array}{lcl}
\ds\sum_{\substack{j=1 \\ j\neq i }}^p \sigma^*_{ij} &= &\sqrt{\lambda} \ds\sum_{\substack{j=1 \\ j\neq i }}^p\left( 1- (\ds\frac{|i-j|}{T})^{2\alpha}\right)_+^{\frac{1}{2}} \leq 2 \sqrt{\lambda} \sum_{k=1}^T \Big( 1 - (\frac{k}{T})^{2\alpha } \Big)^{\frac{1}{2}} \\[0.5 cm]
& \leq& 2 \sqrt{\lambda} \Big( \ds\sum_{k=1}^T ( 1 - (\ds\frac{k}{T})^{2\alpha })  \Big)^{\frac{1}{2}} T^{^{\frac{1}{2}}} =
O(1) T\sqrt{\lambda }
\\
&\leq &  O(1) \varphi^{1 - \frac{1}{2 \alpha}} \to 0 \text{ provided that } \alpha > 1/2.
\end{array}
\]
We deduce that the smallest eigenvalue is bounded from below by
$$
\min_{i=1,...,p}\lambda_{i,U} \geq \min_{i} \{ \sigma^*_{ii} - \ds\sum_{\substack{j=1 \\ j\neq i }}^p \sigma^*_{ij} \} = 1 - \max_{i} \ds\sum_{\substack{j=1 \\ j\neq i }}^p \sigma^*_{ij} \geq 1 - O(1) \varphi^{1-\frac 1{2\alpha}}
$$
which is strictly positive for $\varphi >0$ small enough.
\end{proof}

\begin{proposition}
\label{prop:Wijproperty}
For all $1 \leq i < j \leq p$, $W_{ij}$ is a centered random variable with variance, $\Var_I (W_{ij})= n$. Moreover, for $1 \leq i < j \ne j' \leq p$, we have
\[
\mathbb{E}_I(W_{ij}^4)  = 3n^2 + 6n \, , \quad  \mathbb{E}_I( W_{ij}^2 W_{ij'}^2)= n^2 + 2 n \, , \quad \mathbb{E}_I( W_{ij}^4 W_{ij'}^4) =  9( n^4 + 12 n^3 + 44n^2 + 48n).
\]
Also we have that $\mathbb{E}_I( W_{ij}^8) = 105n^4(1 +o(1))$.
Note that if we have $ i \ne i'$  and $j \ne j'$, then $W^d_{ij}$ and $W^d_{i'j'}$ are not correlated for $d$ finite integer.
\noindent
Moreover, for all $1 \leq i \leq p$, the random variables $W_{ii}$ are such that,
\[
\mathbb{E}(W_{ii}) = n \, , \quad \mathbb{E}_I(W_{ii}^2)= n^2 + 2n \, , \quad  \mathbb{E}_I(W_{ii}^4)= n^4 + 12 n^3 + 44 n^2 + 48n.
\]
\end{proposition}

\begin{proof}[Proof of Proposition \ref{prop:Wijproperty}]
To show the results we use lemma 3  and  some technical computation of \citep{CaiMa13}.
\begin{eqnarray*}
\Var_I(W_{ij}) & =& \mathbb{E}_I(W_{ij}^2)= \mathbb{E}_I( X_{\cdot  i}^\top X_{\cdot  j})^2 = \mathbb{E}_I ( tr( X_{\cdot  i} X_{\cdot  i}^\top X_{\cdot  j} X_{\cdot  j}^\top)) = tr(I_n^2)= n. \\
\mathbb{E}_I(W_{ij}^4) &=& \mathbb{E}_I(X_{\cdot  i}^\top X_{\cdot  j})^4= 3 tr^2(I_n^2) + 6 tr(I_n^4) = 3n^2 + 6n \nonumber \\
\mathbb{E}_I( W_{ij}^2 W_{ij'}^2) &=& \mathbb{E}_I ((X_{\cdot  i}^\top X_{\cdot  j})^2 (X_{\cdot  i}^\top X_{\cdot  j'})^2) = tr^2(I_n^2) + 2 tr(I_n^4)= n^2 + 2 n \\
\mathbb{E}_I( W_{ij}^4 W_{ij'}^4) &=& \mathbb{E}_I ((X_{\cdot  i}^\top X_{\cdot  j})^4 (X_{\cdot  i}^\top X_{\cdot  j'})^4 ) = \mathbb{E}_I \Big( \mathbb{E}_I \Big( (X_{\cdot  i}^\top X_{\cdot  j})^4 (X_{\cdot  i}^\top X_{\cdot  j'})^4 | X_{\cdot i} \Big) \Big).
\end{eqnarray*}
Or $\mathbb{E}_I \Big( (X_{\cdot  i}^\top X_{\cdot  j})^4 (X_{\cdot  i}^\top X_{\cdot  j'})^4 | X_{\cdot i} \Big) = g(X_{\cdot  i})$, where
\begin{eqnarray*}
g(x_{\cdot i}) = \mathbb{E}_I \Big( (x_{\cdot  i}^\top X_{\cdot  j})^4 (x_{\cdot  i}^\top X_{\cdot  j'})^4 \Big) =  \mathbb{E}_I \Big( (x_{\cdot  i}^\top X_{\cdot  j})^4 \Big) \mathbb{E}_I \Big( (x_{\cdot  i}^\top X_{\cdot  j'})^4 \Big)
\end{eqnarray*}
\begin{eqnarray*}
 \mathbb{E}_I \Big( (x_{\cdot  i}^\top X_{\cdot  j})^4 \Big) &=&  \mathbb{E}_I \Big( \ds\sum_{k=1}^p x_{k ,i}X_{k, j} \Big)^4 = \ds\sum_{k=1}^p  x_{k, i}^4 \mathbb{E}_I(X_{k ,j}^4) + 3 \underset{ k_1 \ne k_2}{\sum \sum}  x_{k_{1}, i}^2 x_{k_{2}, i}^2  \mathbb{E}_I(X_{k_{1}, j}^2) \mathbb{E}_I(X_{k_{2} , j}^2)  \\
 &=&  3 \Big( \ds\sum_{k=1}^p x_{k,i}^2 \Big)^2 = 3  ( x_{ \cdot i}^\top  x_{ \cdot i} )^2
\end{eqnarray*}
then we obtain that
\begin{eqnarray*}
\mathbb{E}_I( W_{ij}^4 W_{ij'}^4)
= 9 \mathbb{E}_I  (X_{ \cdot i}^\top X_{ \cdot i} )^4 = 9( n^4 + 12 n^3 + 44n^2 + 48n).
\end{eqnarray*}
Also we have that
\begin{eqnarray*}
\mathbb{E}_I( W_{ij}^8) &=&  \mathbb{E}_I \Big( \ds\sum_{k=1}^n  X_{ k ,i} X_{ k, j} \Big)^8 = \ds\sum_{k=1}^n  \mathbb{E}_I^2 ( X_{ k, i}^8)  + C_8^6 \cdot \underset{ k_1 \ne k_2}{\sum \sum} \mathbb{E}_I^2 ( X_{ k_1 ,i}^6) \cdot \mathbb{E}_I^2 ( X_{ k_2 ,i}^2) \\
& + & \ds\frac{C_8^4}{2!} \cdot  \underset{ k_1 \ne k_2}{\sum \sum} \mathbb{E}_I^2 ( X_{ k_1 ,i}^4) \cdot \mathbb{E}_I^2 ( X_{ k_2 ,i}^4)  + \ds\frac{C_8^4 \cdot C_4^2}{2!} \cdot \underset{ k_1 \ne k_2 \ne k_3}{\sum \sum \sum } \mathbb{E}_I^2 ( X_{ k_1, i}^4) \cdot  \mathbb{E}_I^2 ( X_{ k_2 ,i}^2) \cdot \mathbb{E}_I^2 ( X_{ k_3 ,i}^2) \\
& +&  \ds\frac{C_8^2 \cdot C_6^2 \cdot C_4^2 }{4!} \underset{ k_1 \ne k_2 \ne k_3 \ne k_4}{\sum \sum \sum } \mathbb{E}_I^2 ( X_{ k_1 ,i}^2) \cdot  \mathbb{E}_I^2 ( X_{ k_2 ,i}^2) \cdot  \mathbb{E}_I^2 ( X_{ k_3 ,i}^2) \cdot \mathbb{E}_I^2 ( X_{ k_4, i}^2) \\
&=& 105^2 n + (28 \times 15^2 + 35 \times 9)n(n-1) + (210 \times 9)n(n-1)(n-2) \\
& + & 105n(n-1)(n-2)(n-3) \\
&=&  105n^4 + 1220 n^3 + 2100 n^2  + 7560 n
\end{eqnarray*}
We use similar arguments to calculate the moments of $W_{ii}$.
\end{proof}


\begin{lemma}
\label{lem:berry-essen}
 Let $ 0 < \varepsilon < 1/2$, for any  $t>0$ we have that,
\[
\Big| \mathbb{P}_I( \widehat{\mathcal{D}}_{n,r} \leq t ) - \Phi ( n \ds\sqrt{p} \cdot t ) \Big| \leq 16 \varepsilon^{1/2} \exp( - \frac { n^2p t^2}4) + O \Big( \ds\frac 1n \Big) + O \Big( \ds\frac 1{pT_r} \Big) \quad  \text{ for all } 1 \leq r \leq N.
\]
\end{lemma}
\begin{proof}[Proof of Lemma \ref{lem:berry-essen}]
For each $r \in \{1, \dots, N \}$, $\widehat{\mathcal{D}}_{n,r}$ is a degenerated U-statistic of order 2, and can be written as follows:
 $$
\widehat{\mathcal{D}}_{n,r} = \underset{1 \leq k \ne l \leq n }{\sum \sum} K(X_k, X_l), \quad \text{ where } K(X_k, X_l) = \ds\frac{1}{n(n-1)p} \underset{1 \leq i <j \leq p }{\sum \sum} w_{ij,r}^* X_{k,i} X_{k,j} X_{l,i}X_{l,j}.
 $$
 Define,
 $$
Z_k = \ds\frac 1{\ds\sqrt {\Var_I(\widehat{\mathcal{D}}_{n,r})}}  \sum_{l=1}^{k-1} K(X_k, X_l)  \quad  \text{ and } V_n^2 = \sum_{k=2}^n \mathbb{E}_I(Z_k^2 / \mathcal{F}_{k-1})
$$
where $ \mathcal{F}_k$ is the $\sigma$-field generated by the random variables $ \{ X_1, \dots, X_k \}$.
Moreover, fix $ 0 < \delta \leq  1$,
and define
$$
J_n = \ds\sum_{k=2}^n \mathbb{E}_I (Z_k)^{2 +2\delta} + \mathbb{E}_I |V_n^2 -1 |^{1+ \delta}. $$
Then by Theorem 3 of \citep{ButuceaMatiasPouet09} we get that, there exists a positive constant $k$ depending only on $ \delta $ such that for any $0 < \varepsilon < 1/2 $ and any real $t$,
$$
 \Big| \mathbb{P}_I( \widehat{\mathcal{D}}_{n,r} \leq t ) - \Phi \Big( \ds\frac x{ \ds\sqrt{\Var_I(\widehat{\mathcal{D}}_{n,r}) }} \Big) \Big| \leq 16 \varepsilon^{1/2} \exp( - \frac{t^2}{4 \Var_I(\widehat{\mathcal{D}}_{n,r})}) + \ds\frac{k}{\varepsilon^{1+ \delta }} \cdot J_n.
 $$
Now, we give  upper bounds for  $ \sum_{k=2}^n \mathbb{E}_I (Z_k)^{2 +2\delta}$  and $\mathbb{E}_I |V_n^2 -1 |^{1+ \delta} $ for $ \delta =1$ and get,
\begin{eqnarray*}
\ds\sum_{k=2}^n \mathbb{E}_I(Z_k)^4 &=& \ds\frac{1}{n^2(n-1)^2p^2} \ds\sum_{k=2}^n\mathbb{E}_I \Big( \sum_{l=1}^{k-1} \underset{1 \leq i <j \leq p }{\sum \sum} w_{ij,r}^* X_{k,i} X_{k,j} X_{l,i}X_{l,j} \Big)^4 \\
&=& \ds\frac{1}{n^2(n-1)^2p^2} \ds\sum_{k=2}^n \Big\{(k-1) \Big( 3^4 \underset{1 \leq i <j \leq p }{\sum \sum} w_{ij,r}^{*4} + 3 \underset{1 \leq i <j \leq p }{\sum \sum} \underset{1 \leq i' <j' \leq p }{\sum \sum} w_{ij,r}^{*2}  w_{i'j',r}^{*2} \Big) \\
&& + 3 (k-1)(k-2)  \Big( 3^2\underset{1 \leq i <j \leq p }{\sum \sum} w_{ij,r}^{*4} + \underset{1 \leq i <j \leq p }{\sum \sum} \underset{1 \leq i' <j' \leq p }{\sum \sum} w_{ij,r}^{*2}  w_{i'j',r}^{*2} \Big) \Big\} \\
& \leq & \ds\frac{1}{n^2(n-1)^2p^2}  \Big\{ \frac{n(n-1)}2 \Big( \frac{81}2 \cdot p (\sup\limits_{i,j} w_{ij,r}^{*2})  +  \frac{3p^2}{4} \Big) \\
&& \hspace{2 cm} + \frac{(n-1)n(2n-1)}{6} \Big(   \frac{9}2 \cdot p (\sup\limits_{i,j} w_{ij,r}^{*2}) + \frac{p^2}{4} \Big) \Big\} \\
&=& O \Big( \ds\frac 1n \Big).
\end{eqnarray*}
Similarly we can show that $\mathbb{E}_I (V_n^2 -1 )^2 =  O \Big( \ds\frac 1n \Big) +  O \Big( \ds\frac 1{p T_r} \Big)$. Thus we obtain the desired result.
\end{proof}

\end{document}